\documentclass[11pt,reqno,a4paper]{amsart}
\usepackage{setspace}
\usepackage[english]{babel}
\usepackage[utf8]{inputenc}
\usepackage{tikz}
\usetikzlibrary{cd}
\usepackage{amsfonts}  
\usepackage{amsmath} 
\usepackage{amsthm}
\numberwithin{equation}{section}
\usepackage{amssymb}
\usepackage{mathtools}
\usepackage{mathrsfs}
\usepackage{hyperref}
\usepackage{setspace}
\usepackage{csquotes}

\allowdisplaybreaks

\usepackage{units}
\usepackage{icomma}
\usepackage{graphicx}
\usepackage{tcolorbox}
\usepackage{listings}
\usepackage{relsize}
\usepackage{subcaption}
\usepackage{mdframed}
\usepackage{caption}
\usepackage{amsthm}
\usepackage{rotating}
\usepackage{tikz-cd}
\usepackage[margin=1in,footskip=.25in]{geometry}
\usepackage{fancyhdr}

\edef\restoreparindent{\parindent=\the\parindent\relax}
\usepackage{parskip}
\restoreparindent

\patchcmd{\subsection}{-.5em}{.5em}{}{}
\patchcmd{\subsubsection}{-.5em}{.5em}{}{}

\makeatletter
\def\thm@space@setup{%
  \thm@preskip=\parskip \thm@postskip=0pt
}
\makeatother

\newcommand{\e}{\ensuremath{\mathrm{e}}} 

\newcommand{\Proj}{\mathrm{Proj}}

\newcommand{\Z}{\mathbb{Z}}
\newcommand{\Q}{\mathbb{Q}}
\newcommand{\R}{\mathbb{R}}
\newcommand{\C}{\mathbb{C}}

\renewcommand{\H}{\mathbb{H}}

\newcommand{\GL}{\mathrm{GL}}
\newcommand{\SO}{\mathrm{SO}}
\newcommand{\U}{\mathrm{U}}

\newcommand{\SU}{\mathrm{SU}}
\newcommand{\SL}{\mathrm{SL}}
\newcommand{\Sp}{\mathrm{Sp}}

\newcommand{\PGL}{\mathrm{PGL}}

\newcommand{\Isom}{\mathrm{Isom}}
\newcommand{\Stab}{\mathrm{Stab}}
\newcommand{\Vol}{\mathrm{Vol}}

\newcommand{\Var}{\mathrm{Var}}

\newcommand{\supp}{\mathrm{supp}}

\newcommand{\rad}{\mathrm{rad}}

\newcommand{\Emu}{\mathbb{E}_{\mu}}

\newcommand{\Varmu}{\mathrm{Var}_{\mu}}

\newcommand{\NVmu}{\mathrm{NV}_{\mu}}
\newcommand{\NV}{\mathrm{NV}}

\newcommand{\Poi}{\mathrm{Poi}}

\newcommand{\arccosh}{\mathrm{arccosh}}

\renewcommand{\Re}{\mathrm{Re}}
\renewcommand{\Im}{\mathrm{Im}}

\newcommand{\even}{\mathrm{even}}

\newcommand{\Ccinfty}{C_c^{\infty}}

\newcommand{\Borelinfty}{\mathscr{L}^{\infty}}
\newcommand{\Borelbndinfty}{\mathscr{L}_{c}^{\infty}}

\newcommand{\Radonplus}{\mathscr{M}_+}

\newcommand{\PW}{\mathrm{PW}}

\renewcommand{\hat}{\widehat}
\renewcommand{\tilde}{\widetilde}

\newcommand{\norm}[1]{\lVert#1\rVert}

\newcommand{\cA}{\mathcal{A}}

\newcommand{\cH}{\mathcal{H}}

\newcommand{\cK}{\mathcal{K}}
\newcommand{\cL}{\mathcal{L}}

\newcommand{\cO}{\mathcal{O}}
\newcommand{\cP}{\mathcal{P}}

\newcommand{\cS}{\mathcal{S}}

\newcommand{\bD}{\mathbb{D}}
\newcommand{\bE}{\mathbb{E}}

\newcommand{\bH}{\mathbb{H}}

\newcommand{\bS}{\mathbb{S}}
\newcommand{\bT}{\mathbb{T}}

\newcommand{\sA}{\mathscr{A}}
\newcommand{\sB}{\mathscr{B}}
\newcommand{\sC}{\mathscr{C}}

\newcommand{\sL}{\mathscr{L}}

\newcommand{\sS}{\mathscr{S}}

\newcommand{\sU}{\mathscr{U}}

\definecolor{lichtgrijs}{gray}{1.00}
\mdfdefinestyle{mystyle}{ %
    backgroundcolor=lichtgrijs, %
    linewidth=0pt, %
    innertopmargin=5pt, %
    innerbottommargin=10pt,%
}

\newtheorem{theorem}{Theorem}[section]
\newtheorem{corollary}[theorem]{Corollary}
\newtheorem{proposition}[theorem]{Proposition}
\newtheorem{lemma}[theorem]{Lemma}

\theoremstyle{definition}
\newtheorem{definition}[theorem]{Definition}
\newtheorem{remark}[theorem]{Remark}

\newtheorem{example}{Example}[section]

\usepackage[bibstyle=numeric]{biblatex}

\renewbibmacro{in:}{}
\bibliography{Bibliography} 

\begin{document}

\newgeometry{top=3cm, bottom=3.5cm,left=3cm,right=3cm}

\title{Hyperuniformity and hyperfluctuations of random measures in commutative spaces}
\author{Michael Björklund}\author{Mattias Byl\'ehn}
\subjclass[2020]{Primary: 60G57, 43A90 Secondary: 11N45}
\keywords{Hyperuniformity, Gelfand pair, spherical functions, spectral measures}

\pagestyle{plain}
\pagenumbering{arabic}

\begin{abstract}
We introduce the notion of Bartlett spectral measure for isometrically invariant random measures on proper metric commutative spaces. When the underlying Gelfand pair corresponds to a higher-rank, connected, simple matrix Lie group with finite center and a maximal compact subgroup, we prove that the number variance is asymptotically bounded above, uniformly across all random measures, by the volume raised to a power strictly less than 2. On Euclidean and real hyperbolic spaces we define a notion of heat kernel hyperuniformity for random distributions that generalizes hyperuniformity of random measures, and we prove that every sufficiently well behaved spectral measure can be realized as the Bartlett spectral measure of an invariant Gaussian random distribution. We also compute Bartlett spectral measures for invariant determinantal point processes in commutative spaces, providing a spectral proof of hyperuniformity for infinite polyanalytic ensembles in the complex plane, as well as heat kernel non-hyperuniformity of the zero set of the standard hyperbolic Gaussian analytic function on the Poincar\'e disk. 
\end{abstract}

\maketitle

\section{Introduction}

In the context of statistical mechanics and condensed matter, notions of \emph{long-range order} of a physical configuration is of fundamental interest, and are in many settings intimately related to the suppression of some type of \emph{fluctuation} in the thermodynamic limit. For spatial homogeneous point configurations, realized mathematically as a translation invariant point process in Euclidean space, a fluctuation whose behavior indicates long-range order or disorder is that of mass or density. More precisely, the mass fluctuations in the thermodynamical limit is recorded by the \emph{number variance} of the point process with respect to asymptotically large metric balls, and a translation invariant point process is \emph{hyperuniform} as defined by Stillinger-Torquato in \cite{StillingerTorquato} if the number variance is asymptotically smaller than the volume of balls in the large radial limit, indicating that it is more well ordered than a classical ideal gas. Equivalently, one can define hyperuniformity of a point process spectrally in terms of its \emph{structure factor}, or more generally \emph{Bartlett spectrum}, having asymptotically faster decay than the volume of balls in the small frequency limit. In physical terms, hyperuniformity corresponds to sufficient suppression of low frequencies in the (idealized) Fraunhofer diffraction picture. 

Hyperuniformity of invariant point processes has been well studied in Euclidean spaces, and have recently been extended to compact homogeneous geometries in \cite{StepanyukHyperuniformityOnFlatTori, Grabner2024HyperuniformPointsetsOnProjectiveSpaces, GrabnerHyperuniformityOnSpheresDeterministicAspects, GrabnerHyperuniformityOnSpheresProbabilisticAspects} as well as hyperbolic spaces and regular trees in \cite{Björklund2024HyperuniformityOfRandomMeasuresOnEuclideanAndHyperbolicSpaces, byléhn2024hyperuniformityregulartrees}. Invariant point processes and random measures on general locally compact second countable Hausdorff homogeneous spaces have been studied in \cite{LastHomogeneousSpaces}, and the motivation for this paper is to set up a framework for studying number variances and Bartlett spectra of point processes in non-compact \emph{commutative spaces}, a class of homogeneous spaces on which there is a particularly well-behaved notion of Fourier transform for radial functions. These spaces include Euclidean and hyperbolic spaces, (bi-)regular trees, but also higher rank geometries including non-compact Riemannian symmetric spaces and Bruhat-Tits buildings. 

We provide a notion of \emph{Bartlett spectral measure} for random measures on commutative spaces, generalizing the Euclidean setting, and we extend this notion to invariant random distributions. Using the existence of such a spectral measure together with techniques of Gorodnik-Nevo we can show that for higher rank symmetric spaces that are quotients of connected simple matrix Lie groups with finite center by a maximal compact subgroup, number variances are asymptotically bounded from above, uniformly among all invariant random measures, by a power of the volume of balls strictly less than 2. We also provide a general framework for the number variance of the Bartlett spectral measure of invariant determinantal point processes on commutative spaces, which includes invariant Weyl-Heisenberg ensembles on Euclidean spaces and certain Gaussian analytic functions on complex Euclidean spaces and the hyperbolic plane. Finally, we provide a notion of \emph{heat kernel hyperuniformity} on Euclidean and real hyperbolic spaces, which is equivalent to spectral hyperuniformity introduced in \cite{Bjorklund2022HyperuniformityAN, Björklund2024HyperuniformityOfRandomMeasuresOnEuclideanAndHyperbolicSpaces}, and show that every sufficiently well behaved positive measure on the Fourier domain can be realized as the Bartlett spectral measure of some invariant random (tempered) distribution.

\subsection{Spectral measures of random measures}

Let $(X, d, x_o)$ be a proper pointed metric space and $G$ a group of isometries of $(X, d)$ acting transitively on $X$ with compact stabilizer $K$ of $x_o$, so that $X$ is homeomorphic to $G/K$. By an \emph{invariant random measure} on $X$ we mean a probability measure $\mu$ on the space $\Radonplus(X)$ of positive Radon measures on $X$ which is invariant under the push-forward action of $G$. On such a probability space one has natural families of random variables readily available: to every complex-valued bounded measurable function $f \in \Borelbndinfty(X)$ with compact support one can associate the \emph{linear statistic} $\bS f : \Radonplus(X) \rightarrow \C$,
\begin{align*}
\bS f(p) = \int_X f(x) dp(x) 
\end{align*}
as well as the \emph{discrepancy statistic}
\begin{align*}
\bS_0 f(p) = \bS f(p) - \Emu(\bS f) = \int_X f(x) dp(x) - i_{\mu} \int_X f(x) dm_X(x) \, . 
\end{align*}
Here $m_X$ denotes a fixed $G$-invariant reference measure on $X$ and $i_{\mu} > 0$ is the \emph{intensity} of $\mu$, namely the expected measure of a unit $m_X$-volume Borel set. An invariant random measure $\mu$ is \emph{locally square-integrable} if all linear statistics are square-integrable with respect to $\mu$, that is
\begin{align*}
\int_{\Radonplus(X)} \Big|\int_X f(x) dp(x) \Big|^2 d\mu(p) < +\infty \, , \quad \forall \, f \in \Borelbndinfty(X) \,  . 
\end{align*}
In the context of hyperuniformity and various related notions, one is often interested in understanding the variance of linear statistics with respect to a given invariant locally square-integrable random measure $\mu$, 
\begin{align*}
\Varmu(\bS f) = \int_{\Radonplus(X)} |\bS_0 f(p)|^2 d\mu(p) \, .
\end{align*}
In particular, when $f = \chi_{B_r(x_o)}$ is the indicator function of a $K$-invariant ball of radius $r > 0$ with respect to the fixed basepoint $x_o \in X$, we define the \emph{number variance} of $\mu$ as
\begin{align*}
\NVmu(r) = \Varmu(\bS \chi_{B_r(x_o)}) \, . 
\end{align*}
Heuristically, the physical significance of the number variance is that it records the macroscopic mass fluctuations of the random measure $\mu$ along the sequence $B_r(x_o)$ when $r > 0$ is large. 

We are interested in spectral decompositions of the variance of $K$-invariant linear statistics, especially for the indicator functions $\chi_{B_r(x_o)}$. For the setup, our proper pointed metric space $X = G/K$ as previously described is in addition assumed to be a \emph{commutative space}, equivalently that $(G, K)$ forms a \emph{Gelfand pair}, which ensures the existence of \emph{spherical functions} $\omega : G \rightarrow \C$, meaning non-zero bi-$K$-invariant functions satisfying the functional equation 
\begin{align*}
\omega(g_1)\omega(g_2) = \int_K \omega(g_1 k g_2) dm_K(k) \, , \quad g_1, g_2 \in G \, . 
\end{align*}
The family of all \emph{positive-definite} spherical functions for $(G, K)$ are of main interest and we denote it by $\cS^+ = \cS^+(G, K)$. For Euclidean spaces these are simply the complex exponentials, for the hyperbolic plane they are Legendre functions and for regular trees they are combinations of Chebyshev polynomials. If $\varphi : G \rightarrow \C$ is an integrable bi-$K$-invariant function, for example the indicator of a $K$-invariant ball, then one defines its \emph{spherical transform} $\hat{\varphi} : \cS^+ \rightarrow \C$ to be 
$$ \hat{\varphi}(\omega) = \int_G \varphi(g) \omega(g^{-1}) dm_G(g) \, . $$
For commutative spaces $X = G/K$ as described, we show that $K$-invariant discrepancy statistics in $L^2_0(\mu)^K$ for an invariant locally square-integrable random measure $\mu$ on $X$ can be decomposed in terms of a spectral measure on the space $\cS^+$. To state the Theorem, we introduce for every $f \in \Borelbndinfty(X)$ a right-$K$-invariant function $\varphi_f : G \rightarrow \C$ by $\varphi_f(g) = f(g.x_o)$. 

\begin{theorem}
\label{Theorem1}
Let $\mu$ be an invariant locally square-integrable random measure on $X$. Then there is a unique positive Radon measure $\sigma_{\mu} \in \Radonplus(\cS^+)$ such that
\begin{align*}
\Varmu(\bS f) = \int_{\cS^+} |\hat{\varphi}_f(\omega)|^2 d\sigma_{\mu}(\omega) 
\end{align*}
for every $K$-invariant function $f \in \Borelbndinfty(X)$.
\end{theorem}
The existence and uniqueness of the spectral measure $\sigma_{\mu}$ has been utilized in \cite{Björklund2024HyperuniformityOfRandomMeasuresOnEuclideanAndHyperbolicSpaces} and \cite{byléhn2024hyperuniformityregulartrees}, where the main results rely on Theorem \ref{Theorem1}. When $X = \R^d$, then the measure $\sigma_{\mu}$ is the classical \emph{Bartlett spectral measure} of $\mu$ introduced by Bartlett in \cite{Bartlett}, and for the commutative spaces $X$ that we consider we will adopt the same terminology for $\sigma_{\mu}$. The existence and uniqueness of Bartlett spectral measures has been proven for invariant translation bounded random measures in the series \cite{BHPI, BHPII, BHPIII}, and we emphasize that Theorem \ref{Theorem1} includes many interesting random measures and point processes that are typically not translation bounded, for example determinantal/permanental point processes. Other names for the Bartlett spectral measure $\sigma_{\mu}$ in the literature include \emph{diffraction measure} and \emph{spectral measure}. We note as a canonical example that if $\mu = \Poi(m_X)$ is the invariant Poisson point process with intensity measure $m_X$ then  the corresponding Bartlett spectral measure is the \emph{spherical Plancherel measure} $\sigma_{\cP} \in \Radonplus(\cS^+)$, satisfying 
\begin{align*}
\norm{\varphi}_{L^2(G)}^2 = \int_{\cS^+} |\hat{\varphi}(\omega)|^2 d\sigma_{\cP}(\omega)  
\end{align*}
for all bi-$K$-invariant $\varphi \in L^2(G)$. 
\begin{remark}
\label{RemarkSpectralMeasures}
The setup for Theorem \ref{Theorem1} is an instance of a more general framework that we develop in Section \ref{Spectral measures}. In Theorem \ref{TheoremSpectralMeasureForAlgebraGMaps} we show that any linear $G$-equivariant map $\alpha : \sA(G) \rightarrow \cH$ between a ''sufficiently large'' $*$-algebra $\sA(G)$ of integrable functions to a $K$-spherical unitary representation $(\pi, \cH)$ admits a spectral measure $\sigma_{\alpha} \in \Radonplus(\cS^+)$ in the sense that
\begin{align*}
\norm{\alpha(\varphi)}_{\cH}^2 = \int_{\cS^+} |\hat{\varphi}(\omega)|^2 d\sigma_{\alpha}(\omega) 
\end{align*}
for all bi-$K$-invariant $\varphi \in \sA(G)$. This general framework allows us to extend the notion of Bartlett spectral measures to random \emph{distributions} on $X$. In the Appendix, we prove that every positive Radon measure $\sigma \in \Radonplus(\cS^+)$ whose spherical transform is a bi-$K$-invariant positive-definite distribution on $G$ can be realized as the Bartlett spectral measure of a Gaussian random distribution on $X$.
\end{remark}

\subsection{Maximal fluctuations in higher rank symmetric spaces}

For invariant locally square-integrable random measures $\mu$ on Euclidean space $\R^d$, an application of the mean ergodic theorem allows one to show that
\begin{align*}
\limsup_{r \rightarrow +\infty} \frac{\NVmu(r)}{\Vol_{\R^d}(B_r(0))^2} = 0 \, . 
\end{align*}
From works of Gorodnik and Nevo, see for example the book \cite{GorodnikNevoBook}, this upper bound for the number variance holds for many non-compact symmetric spaces $X = G/K$ in the sense that the metric balls $B_r(x_o) \subset X$, $r > 0$, form a mean ergodic sequence. How close one is able to get to this upper asymptotic bound is the topic of this Subsection.

\noindent\textbf{Question}: \textit{Given a proper metric commutative space $X = G/K$, is there for every $\varepsilon > 0$ an invariant locally square-integrable random measure $\mu_{\varepsilon}$ on $X$ such that}
\begin{align*}
\limsup_{r \rightarrow +\infty} \frac{\NV_{\mu_{\varepsilon}}(r)}{\Vol_{X}(B_r(x_o))^{2 - \varepsilon}} > 0 \, ?
\end{align*}
This question has been addressed in several works, and we recall some of them here. 

In \cite[Section 5]{deConickDunlopHuillet}, it is shown that for every $\varepsilon > 0$, there is a permanental point process $\mu_{L_{\varepsilon}}$ on the real line $\R$ with positive-definite kernel $L_{\varepsilon}(t_1, t_2) = (1 + |t_1 - t_2|^{\alpha})^{-\zeta/2}$ under the condition $\alpha \zeta = \varepsilon$ for $\alpha \in (0, 2]$ and $\zeta > 0$, such that
\begin{align*}
\liminf_{r \rightarrow +\infty} \frac{\NV_{\mu_{L_{\varepsilon}}}(r)}{\Vol_{\R}([-r, r])^{2 - \varepsilon}} > 0 \, .
\end{align*}
We expect that higher dimensional analogues of such permanental point processes also should fulfil the analogous result.

This has also been answered in hyperbolic geometries: In \cite[Theorem 1(c)]{Buckley}, Buckley considers point processes governed by the zero sets of Gaussian analytic functions on the hyperbolic plane $\bH^2 = \SU(1, 1)/\U(1)$, realized in the Poincar\'e disk model. These Gaussian analytic functions are given by
\begin{align*}
F_t(z) = \sum_{n = 0}^{\infty} \Big(\frac{\Gamma(t + n)}{n! \Gamma(t)}\Big)^{\frac{1}{2}} a_n z^n \, , \quad |z| < 1 \, ,
\end{align*}
where $a_n$ are complex standard Gaussian random variables and $t > 0$. The point process associated with the random zero set of $F_t$ is
\begin{align*}
\mu_t = \sum_{z : F_t(z) = 0} \delta_z \, . 
\end{align*}
Buckley proves among other things that for every $\varepsilon > 0$ there is a $t_{\varepsilon} > 0$ such that 
\begin{align*}
\liminf_{r \rightarrow +\infty} \frac{\NV_{\mu_{t_{\varepsilon}}}(r)}{\Vol_{\H^2}(B_r(0))^{2 - \varepsilon}} > 0 \, .
\end{align*}
This result was in turn generalized to higher dimensions by Massaneda and Pridhnani in \cite[Theorem 1.2]{MassanedaPridhnani}, where they consider Gaussian analytic functions on higher-dimensional complex hyperbolic spaces $\bH_{\C}^d = \SU(1, d)/\SU(d)$ in the unit ball model, which in multi-index notation is defined by
\begin{align*}
F_t(z) = \sum_{m \in \Z_{\geq 0}^d} \Big(\frac{\Gamma(t + |m|)}{m! \Gamma(t)}\Big)^{\frac{1}{2}} a_m z^m \, , \quad \norm{z} < 1 \, .
\end{align*}
Similarly to Buckley's result, it is shown that for every $\varepsilon > 0$ there is a $t_{\varepsilon} > 0$ such that the random zero variety $\mu_{t_{\varepsilon}}$ of $F_t$ in $\bH_{\C}^d$ satisfies
\begin{align*}
\liminf_{r \rightarrow +\infty} \frac{\NV_{\mu_{t_{\varepsilon}}}(r)}{\Vol_{\H^d_{\C}}(B_r(0))^{2 - \varepsilon}} > 0 \, .
\end{align*}

Using techniques due to Gorodnik-Nevo in \cite{GorodnikNevoBook}, we answer the Question above in the negative when $G$ is a higher rank connected simple matrix Lie group with finite center and $K < G$ a maximal compact subgroup. These groups have \emph{Kazhdan's Property} (T), meaning that the trivial $G$-representation is isolated in the unitary dual of $G$, and moreover the stronger property that there is a $q \in [2, +\infty)$ such that all matrix coefficients of non-trivial irreducible representations with $K$-finite vectors are in $L^q(G)$. In particular, this integrability condition holds for the non-trivial positive-definite spherical functions on $G$. Using this, we prove that the number variance is asymptotically bounded above, uniformly across all random measures, by the volume raised to a power strictly less than 2. 

\begin{theorem}
\label{Theorem2}
Suppose $G$ is a higher rank connected simple matrix Lie group with finite center and $K < G$ a maximal compact subgroup. Then there is a constant $\delta = \delta(G) \in (0, 1)$ such that for any invariant locally square-integrable random measure $\mu$ on $X = G/K$,  
\begin{align*}
\limsup_{r \rightarrow +\infty} \frac{\NVmu(r)}{\Vol_X(B_r(x_o))^{2 - \delta}} = 0 \, . 
\end{align*}
\end{theorem} 

Gelfand pairs $(G, K)$ as in Theorem \ref{Theorem2} include for example 
$$(\SL_3(\R), \SO(3))\quad \mbox{ and } \quad (\SO(2, 2), \SO(2) \times \SO(2))\, . $$
%
%
In light of Theorem \ref{Theorem2}, we emphasize that the results mentioned above for permanental point processes and Gaussian analytic functions do not hold in the context of complex symmetric spaces coming from Gelfand pairs as in Theorem \ref{Theorem2}. To give an explicit example, consider the Siegel upper half space
\begin{align*}
\mathbb{SH}_d = \Big\{ Z = X + iY \in \mathrm{Mat}_{d \times d}(\C) : X^t = X \, , \, Y^t = Y \, , \,\, Y > 0 \Big\} 
\end{align*}
of genus $d \geq 2$, which naturally carries the structure of a complex manifold. There is a metric on $\mathbb{SH}_d$ compatible with the complex structure such that the group $\Sp(2d, \R)$ of $(2d \times 2d)$-real symplectic matrices acts transitively and isometrically on the Siegel upper half space by fractional transformations,
\begin{align*}
\begin{pmatrix} A & B \\ C & D \end{pmatrix}.Z = (AZ + B)(CZ + D)^{-1} \, . 
\end{align*}
Moreover, the stabilizer for this action is the maximal compact subgroup 
$$\U(d) = \Sp(2d, \R) \cap \GL(d, \C) \, , $$
so that $\mathbb{SH}_d = \Sp(2d, \R)/\U(d)$ defines a complex commutative space for the Gelfand pair $(\Sp(2d, \R), \U(d))$. Since $\Sp(2d, \R)$ is a simple connected Lie group and has real rank greater than or equal to $2$ for $d \geq 2$, we can conclude the following from Theorem \ref{Theorem2}.

\begin{corollary}
For the Siegel upper half spaces $\mathbb{SH}_d$, $d \geq 2$, the answer to the Question above is negative, in particular for all families of invariant permanental point processes and invariant zero sets of Gaussian analytic functions.
\end{corollary}

\subsection{Hyperuniformity and hyperfluctuations}

A translation invariant locally square-integrable random measure $\mu$ on Euclidean space $X = \R^d$ is \emph{geometrically hyperuniform} if 
\begin{align*}
\limsup_{r \rightarrow +\infty} \frac{\NVmu(r)}{\Vol_{\R^d}(B_r(0))} = 0 \, . 
\end{align*}
If this upper limit is instead positive and finite, we say that $\mu$ has \emph{Poissonian} fluctuations, and if the upper limit is infinite then we say that $\mu$ is \emph{hyperfluctuating}. In the spectral picture, an invariant locally square-integrable random measure $\mu$ on $\R^d$ is \emph{spectrally hyperuniform} if
\begin{align*}
\limsup_{\varepsilon \rightarrow 0^+} \frac{\sigma_{\mu}(B_{\varepsilon}(0))}{\Vol_{\R^d}(B_{\varepsilon}(0))} = 0 \, ,
\end{align*}
and these to notions of hyperuniformity are equivalent, see \cite[Proposition 3.3]{Bjorklund2022HyperuniformityAN} for a proof. 

For non-amenable Gelfand pairs, specifically real hyperbolic spaces 
$$\bH^d = \SO^{\circ}(1, d)/\SO(d) \, , $$
of dimension $d \geq 2$ and $(p + 1)$-regular trees 
$$\bT_{p + 1} = \PGL_2(\Q_p)/\PGL_2(\Z_p)$$
for primes $p$ (and more generally valence $d \geq 3$), a similar notion of geometric hyperuniformity as in the Euclidean case fails to admit any examples. The notion of spectral hyperuniformity however has been extended to these spaces in \cite{Björklund2024HyperuniformityOfRandomMeasuresOnEuclideanAndHyperbolicSpaces, byléhn2024hyperuniformityregulartrees}. In the case of real hyperbolic space $\bH^d$, the positive-definite spherical functions are parameterized by two intervals $\lambda \in i(0, \tfrac{d-1}{2}] \cup [0, +\infty)$, corresponding to complementary and principal $\SO(d)$-spherical unitary irreducible representations respectively. An invariant locally square-integrable random measure $\mu$ on real hyperbolic space is spectrally hyperuniform if $\sigma_{\mu}(i[0, \tfrac{d-1}{2}]) = 0$ and 
\begin{align}
\label{EqIntroSpectralHyperuniformity}
\limsup_{\varepsilon \rightarrow 0^+} \frac{\sigma_{\mu}((0, \varepsilon])}{\sigma_{\cP}((0, \varepsilon])} = 0 \, ,
\end{align}
where $d\sigma_{\cP}(\lambda) = \chi_{(0, +\infty)}(\lambda)|c(\lambda)|^{-2} d\lambda$ is the Plancherel measure and $c(\lambda)$ is the Harish-Chandra $c$-function for $\SO^{\circ}(1, d)$. In the Appendix, we show for Euclidean and real hyperbolic spaces that spectral hyperuniformity of an invariant locally square-integrable random measure is equivalent to what we call \emph{heat kernel hyperuniformity}. To state the definition we consider the bi-$K$-invariant heat kernel $h_t$ for $t > 0$, which in Euclidean spaces is given by
\begin{align*}
h_t(x) = \frac{1}{(4\pi t)^{d/2}} \e^{-\frac{\norm{x}^2}{4t}} 
\end{align*}
and in real hyperbolic spaces by
\begin{align*}
h_t(g) = \int_0^{\infty} \e^{-t(\frac{(d-1)^2}{4} + \lambda^2)} \omega_{\lambda}(g) \frac{d\lambda}{|c(\lambda)|^2} \, . 
\end{align*}
where $\omega_{\lambda} : G \rightarrow \C$ are the principal series positive-definite spherical functions for $\SO^{\circ}(1, d)$. The heat kernel $h_t$ is smooth with Gaussian decay in both geometries, see for example \cite[Theorem 3.1 p.186]{Davies1988HeatKB} for uniform estimates in the hyperbolic case, so we define heat kernel hyperuniformity in the generality of invariant random distributions as follows.

\begin{definition}
Let $\mu$ denote an invariant locally square-integrable random distribution on $X = \R^d, \bH^d$. 
\begin{enumerate}
    \item If $X = \R^d$, we say that $\mu$ is heat kernel hyperuniform if 
    \begin{align*}
        \limsup_{t \rightarrow +\infty} t^{d/2} \Varmu(\bS h_t) = 0 \, .
    \end{align*}
    \item If $X = \bH^d$, we say that $\mu$ is heat kernel hyperuniform if
    \begin{align*}
        \limsup_{t \rightarrow +\infty} t^{3/2} \e^{\frac{(d-1)^2}{2}t} \Varmu(\bS h_t) = 0 \, .
    \end{align*}
\end{enumerate}
\end{definition} 

\begin{proposition}
A random distribution $\mu$ on $X = \R^d, \bH^d$ is spectrally hyperuniform if and only if it is heat kernel hyperuniform.
\end{proposition}
We prove this Proposition in the Appendix. As mentioned in Remark \ref{RemarkSpectralMeasures}, we also show that every positive Radon measure $\sigma \in \Radonplus(\cS^+)$ whose spherical transform is a positive-definite distribution on $G$ can be realized as the Bartlett spectral measure of a Gaussian random distribution on $X = G/K$. For Euclidean and real hyperbolic geometries, such measures $\sigma$ are precisely the tempered positive Radon measures on the positive-definite spherical functions. In the real hyperbolic case we are aware of essentially one spectrally hyperuniform random measure, an invariant random lattice orbit in the hyperbolic plane coming from a lattice $\Gamma < \SL_2(\R)$ constructed by Jenni in \cite[Section 1.2, Eq. B]{Jenni1984UeberDE}. For any tempered measure $\sigma$ on the positive-definite spherical functions satisfying the spectral hyperuniformity criterion in Equation \ref{EqIntroSpectralHyperuniformity}, we can construct from it a spectrally hyperuniform Gaussian random (tempered) \emph{distribution}.

\begin{corollary}
On real hyperbolic spaces $\bH^d$, there are infinitely many heat kernel hyperuniform Gaussian random distributions. 
\end{corollary}

\subsection{Determinantal point processes with equivariant kernels}

We set up a general framework for determinantal point processes $\mu_L$ goverened by continuous positive-definite hermitian kernels $L : X \times X \rightarrow \C$ satisfying the equivariance property
\begin{align}
\label{EqIntroEquivariantKernel}
L(g.x_1, g.x_2) = c(g, x_1)\overline{c(g, x_2)} L(x_1, x_2) \, , \quad g \in G \, , \, x_1, x_2 \in X  
\end{align}
for some measurable map $c : G \times X \rightarrow \bT^1$, where $\bT^1 \subset  \C$ denotes the unit circle. If the bi-$K$-invariant function $\kappa_L(g) = |L(x_o, g.x_o)|^2$ is square-integrable on $G$, then we compute the Bartlett spectral measure of $\mu_L$ from Theorem \ref{Theorem1} to be 
\begin{align*}
d\sigma_{\mu_L}(\omega) = (L(x_o, x_o) - \hat{\kappa}_L(\omega)) d\sigma_{\cP}(\omega)  
\end{align*}
where $\sigma_{\cP}$ denotes the Plancherel measure for $(G, K)$. We provide explicit formulas for the Bartlett spectral measure for some well-studied invariant determinantal point processes.

For \emph{infinite polyanalytic ensembles of pure type} in $\C$, that is, determinantal point processes $\mu_{\lambda, n}$ on $\C$ with kernel of the form
\begin{align*}
L_{\lambda, n}(z_1, z_2) =  \cL_n(\tfrac{1}{2} |\lambda| |z_1 - z_2|^2) \e^{-\frac{1}{4}|\lambda| (|z_1|^2 + |z_2|^2 - 2 z_1\overline{z}_2 )} \, , \,\, \lambda \in \R \backslash \{0\} , n \in \Z_{\geq 0} \, , 
\end{align*}
where $\cL_n$ denotes the $n$'th Laguerre polynomial, the Bartlett spectral measure is given by
\begin{align*}
d\sigma_{\mu_{\lambda, n}}(\zeta) = \Big(1 - \frac{2\pi}{|\lambda|} \cL_n(\tfrac{2\pi^2}{|\lambda|} |\zeta|^2)^2 \e^{-\frac{2\pi^2}{|\lambda|} |\zeta|^2}\Big) dA(\zeta) \, , \quad \zeta \in \C \, . 
\end{align*}
Here, $A$ denotes the Lebesgue measure on $\C$ and we note that $\mu_{2\pi, 0}$ is the well-known \emph{infinite Ginibre ensemble}.   From the Bartlett spectral measure one sees that $\mu_{\pm 2\pi, n}$ are the (spectrally) hyperuniform point processes in this family. 

\noindent The processes $\mu_{\lambda, n}$ are examples of \emph{Weyl-Heisenberg ensembles} introduced in \cite{AbreuPereiraRomeroTorquato}, and the kernels $L_{\lambda, n}$ can be defined in terms of unitary irreducible representations of the Heisenberg motion group $\U(d) \ltimes \bH_d$, which is a covering group of the unitary motion group $\U(d) \ltimes \C^d$. We detail a general approach for  determinantal point processes of Weyl-Heisenberg type on commutative spaces $X = G/K$ with kernels coming from unitary representations of a covering group $\tilde{G}$ of $G$. In light of this, we observe that hyperuniformity of the processes $\mu_{\lambda, n}$ on $\C^d$ corresponds to the governing representation of the Heisenberg motion group having \emph{formal dimension $1$}, meaning
\begin{align*}
\int_{\tilde{G}/Z} |\langle \pi_{\lambda, n}(\tilde{g}) v_{\lambda, n}, v_{\lambda, n} \rangle|^2 dm_{\tilde{G}/Z}(\tilde{g}Z) = 1
\end{align*}
where $Z < \tilde{G}$ is the center and $\pi_{\lambda, n}$ is the irreducible unitary representation of the Heisenberg motion group with $\U(d)$-invariant cyclic unit vector $v_{\lambda, n}$ defining the kernel $L_{\lambda, n}$.

We also compute the Bartlett spectral measure for the invariant zero set of the hyperbolic Gaussian analytic function 
\begin{align*}
F_1(z) = \sum_{n = 0}^{\infty} a_n z^n \,  , \quad |z| < 1 \, , 
\end{align*}
where $a_n$ are again complex standard Gaussian random variables. Peres and Vir\'ag in \cite{PeresVirag} famously proved that the random zero set of $F_1$ is equivalent as a point process to the determinantal point process in $\bH^2$ defined by the \emph{Bergman kernel}, which on the Euclidean unit disk is given by $L_o(z_1, z_2) = \pi^{-1} (1 - z_1 \overline{z}_2)^{-2}$. The modified kernel
\begin{align*}
L(z_1, z_2) = \frac{(1 - |z_1|^2)(1 - |z_2|^2)}{(1 - z_1 \overline{z_2})^2}
\end{align*}
defines the same determinantal point process, which we denote $\mu_L$, and satisfies the equivariance property in Equation \ref{EqIntroEquivariantKernel}. Using a formula for the Berezin transform of spherical functions due to Unterberger-Upmeier in \cite[Prop. 3.39, p.591]{UnterbergerUpmeier}, we calculate the Bartlett spectral measure of the determinantal point process $\mu_L$ to be
\begin{align*}
d\sigma_L(\lambda) &= (1 - |\Gamma(\tfrac{3}{2} + i\lambda)|^2) d\sigma_{\cP}(\lambda) = \chi_{(0, +\infty)}(\lambda) (1 - |\Gamma(\tfrac{3}{2} + i\lambda)|^2) \pi\lambda \tanh(\tfrac{\pi\lambda}{2}) d\lambda \, . 
\end{align*}
In particular, the determinantal point process $\mu_L$ is not heat kernel/spectrally hyperuniform since 
\begin{align*}
\limsup_{\varepsilon \rightarrow 0^+} \frac{\sigma_{\mu}((0, \varepsilon])}{\sigma_{\cP}((0, \varepsilon])} = 1 - \Gamma(\tfrac{3}{2})^2 > 0 \, . 
\end{align*}

\subsection{Structure of the paper} 

We recall some fundamentals of spherical harmonic analysis for Gelfand pairs in Section \ref{Spherical harmonic analysis on Gelfand pairs}, introducing the spherical transform, the spherical Bochner Theorem and spectral projections on the $K$-spherical unitary dual. In Section \ref{Spectral measures} we use this to prove the existence and uniqueness of spectral measures in the general setting, which includes Theorem \ref{Theorem1} as a special case, as well as determine the mass of the atom at the trivial spherical function for such spectral measures. The relevant commutative spaces and random measures on such spaces are introduced in Section \ref{Random measures on commutative spaces}, and in Section \ref{Upper bounds for number variances with spectral gap} we prove Theorem \ref{Theorem2}. Invariant determinantal point processes are treated in Section \ref{Determinantal point processes}, and heat kernel hyperunifromity and Gaussian random distributions are discussed in the Appendix.

\subsection{Acknowledgement} 

This paper is part of the second author’s doctoral thesis at the
University of Gothenburg under the supervision of the first author. The first author was supported by the grant 11253320 from the Swedish Research Council. We sincerely thank Genkai Zhang and Tobias Hartnick for their valuable insights. The second author is grateful to Tobias Hartnick for useful input and discussions regarding the project as well as hospitality during visits to the Karlsruhe Institute of Technology.


\begin{spacing}{0.1}
\tableofcontents
\end{spacing}

\section{Spherical harmonic analysis on Gelfand pairs}
\label{Spherical harmonic analysis on Gelfand pairs}

We introduce notation for function spaces and basic definitions and conventions for groups and Gelfand pairs in the first Subsections, before defining spherical functions and the spherical transform in Subsections \ref{Positive-definite spherical functions} and \ref{The spherical transform}. We then state the spherical analogue of Bochner's Theorem for Gelfand pairs and introduce spectral projections in Subsections \ref{Spherical Bochner's Theorem} and \ref{Linfty action on the spherical unitary dual}, both of which being crucial results for the following Section.

\subsection{Function spaces}
\label{Function spaces}

Given a locally compact second countable Hausdorff (lcsc) space $Y$ equipped with its Borel $\sigma$-algebra, we let $\Radonplus(Y)$ denote the space of positive Radon measures on $Y$. We will write $\sL(Y)$ for the complex-valued measurable functions on $Y$ and let $\Borelinfty(Y), \Borelbndinfty(Y) \subset \sL(Y)$ be the subspaces of bounded measurable functions and bounded measurable functions with compact support respectively. Given a measure $\nu \in \Radonplus(Y)$ and $p \in [1, +\infty)$, we write 
\begin{align*}
\sL^p(Y, \nu) = \Big\{ f \in \sL(Y) : \int_Y |f(y)|^p d\nu(y) < +\infty \Big\} \, . 
\end{align*}
The corresponding Banach space of such functions modulo $\nu$-null equivalence is $L^p(Y, \nu)$. We let $C(Y) \subset \sL(Y)$ be the space of continuous functions on $Y$, $C_b(Y) = C(Y) \cap \Borelinfty(Y)$ the bounded continuous functions, $C_0(Y) \subset C_b(Y)$ the closed subspace of functions vanishing at infinity and $C_c(Y) = C(Y) \cap \Borelbndinfty(Y)$ the compactly supported continuous functions.

If $Y$ has in addition a smooth structure, then we write $C^{\infty}(Y)$ for the smooth functions on $Y$ and $\Ccinfty(Y) = C^{\infty}(Y) \cap \Borelbndinfty(Y)$ for the subspace of compactly supported smooth functions.

Regarding actions, if $K$ is a group acting on $Y$ from the left and $\sA(Y)$ is a function space on $Y$ we write $\sA(Y)^K$ for the subspace of left-$K$-invariant functions in $\sA(Y)$. For a group $G$ with subgroup $K < G$ and a function space $\sA(G)$ on $G$, we will write $\sA(G, K)$ for the space of \emph{bi-$K$-averages} of functions in $\sA(G)$ unless told otherwise, see Equation \ref{EqBiKAveraging} for the definition of bi-$K$-averaging.

\subsection{Preliminaries on Gelfand pairs}
\label{preliminaries on Gelfand pairs}

Let $G$ denote a unimodular, lcsc group with identity element $e \in G$ and fix a Haar measure $m_G \in \Radonplus(G)$ on $G$. If $\varphi_1, \varphi_2 \in \sL^1(G)$ then their \emph{convolution} is the function
\begin{align*}
(\varphi_1 * \varphi_2)(g) = \int_G \varphi_1(h) \varphi_2(h^{-1}g) dm_G(h) \, .
\end{align*}
Alternatively, a change of variables allows one to write
\begin{align*}
(\varphi_1 * \varphi_2)(g) = \int_G \varphi_1(gh^{-1}) \varphi_2(h) dm_G(h) \, .
\end{align*}
Moreover, we will consider the involutions $\overline{\varphi}(g) = \overline{\varphi(g)}, \check{\varphi}(g) = \varphi(g^{-1})$ and the $*$-involution $\varphi^*(g) = \overline{\varphi(g^{-1})}$ for measurable functions $\varphi \in \sL(G)$, in particular $\varphi \in \sL^1(G)$. Then $\sL^1(G)$ inherits a $*$-algebra structure under the operations of convolution and $*$-involution, which carries over to the quotient space $L^1(G)$, making it a Banach $*$-algebra. The left-regular action of $G$ on measurable functions $\varphi \in \sL(G)$ will be denoted by
\begin{align*}
\lambda_G(g)\varphi(h) = \varphi(g^{-1}h)
\end{align*}
and the associated action of $\sL^1(G)$ on bounded measurable functions $\varphi \in \Borelinfty(G)$ is 
\begin{align*}
\lambda_G(\beta)\varphi(g) = \int_G \beta(h) \lambda_G(h)\varphi(g) dm_G(h) = (\beta * \varphi)(g) \, , \quad \beta \in \sL^1(G) \, . 
\end{align*}
Using this action one shows that the left-regular group action commutes with convolution,
\begin{align*}
\lambda_G(g)(\varphi_1 * \varphi_2) = (\lambda_G(g)\varphi_1) * \varphi_2
\end{align*}
for all $\varphi_1, \varphi_2 \in \sL^1(G)$ and all $g \in G$.

Given a compact subgroup $K < G$, the \emph{Hecke algebra} of the pair $(G, K)$ is the algebra $C_c(G, K)$ of bi-$K$-invariant compactly supported continuous complex-valued functions on $G$, endowed with the operation of convolution and the $*$-involution as above. 

\begin{definition}[Gelfand pair]
\label{DefGelfandPair}
Let $G$ be a lcsc group and $K < G$ a compact subgroup. The pair $(G, K)$ is \emph{Gelfand} if the Hecke algebra $(C_c(G, K), *)$ is commutative.
\end{definition}

We remark that the Gelfand property is equivalent to convolution being commutative on the space $\sL^1(G, K)$ of measurable bi-$K$-invariant integrable functions. Gelfand pairs include abelian groups $G$ with $K$ trivial, connected semisimple Lie groups $G$ with finite center and maximal compact subgroup $K < G$, automorphism groups of biregular trees with the corresponding root stabilizer as well as Euclidean and Heisenberg motion groups with the corresponding maximal compact subgroup of orthogonal/unitary transformations.

\subsection{Positive-definite spherical functions}
\label{Positive-definite spherical functions}

Let $(G, K)$ be a lcsc Gelfand pair with $K < G$ compact. A \emph{$K$-spherical function}, or \emph{spherical function} for short, for the pair $(G, K)$ is a bi-$K$-invariant continuous function $\omega \in C(G, K)$ such that $\omega(e) = 1$ satisfying the functional equation
\begin{align}
\label{EqSphericalFunctionFunctionalEquation}
\int_K \omega(gkh) dm_K(k) = \omega(g)\omega(h) \, . 
\end{align}
We denote the space of spherical functions for $(G, K)$ by $\cS = \cS(G, K)$, endowed with the weak$^*$-topology from $C_c(G)^*$. Special attention will be paid to the \emph{positive-definite} spherical functions. A continuous function $\beta \in C(G)$ is positive-definite if, for every $\varphi \in \Borelbndinfty(G)$,
\begin{align*}
\int_G (\varphi^* * \varphi)(g) \beta(g) dm_G(g) \geq 0 \, . 
\end{align*}
Given a positive-definite function $\beta \in C(G)$, the classical GNS-construction yields a unitary $G$-representation $(\pi_{\beta}, \cH_{\beta})$ and a non-zero cyclic vector $v_{\beta} \in \cH_{\beta}$ such that $\beta(g) = \langle \pi_{\beta}(g) v_{\beta}, v_{\beta} \rangle_{\cH_{\beta}}$. Moreover, the triple $(\pi_{\beta}, \cH_{\beta}, v_{\beta})$ is unique up to unitary equivalence. In particular, such $\beta$ are uniformly continuous and bounded with $\norm{\beta}_{\infty} = \beta(e)$. We denote by $\cS^+ \subset \cS$ the locally compact Hausdorff space of positive-definite spherical functions for $(G, K)$. A spherical function $\omega \in \cS$ is positive-definite if and only if it is the matrix coefficient of a unitary $K$-spherical irreducible unitary $G$-representation $(\pi_{\omega}, \cH_{\omega})$. Here, $K$-spherical means that the subspace $\cH_{\omega}^K$ of $K$-invariant vectors is non-trivial, and by irreducibility we must have $\cH_{\omega}^K = \C v_{\omega}$. In particular, the unit vector $v_{\omega} \in \cH_{\omega}^K$ is uniquely defined up to multiplication by a unit norm complex number. In this sense, $\cS^+$ can be identified with the \emph{$K$-spherical unitary dual} $\hat{G}^K$ of equivalence classes of $K$-spherical irreducible unitary $G$-representations. 

\subsection{The spherical transform}
\label{The spherical transform}

The positive-definite spherical functions introduced for Gelfand pairs $(G, K)$ are in a sense a generalization of unitary characters on abelian groups, and they allow us to define an analogue of the Fourier transform for functions for the bi-$K$-invariant on $G$.

\begin{definition}[Spherical transform]
\label{DefSphericalTransform}
Let $\varphi \in \sL^1(G, K)$. The \emph{spherical transform} of $\varphi$ is the function $\hat{\varphi} \in C_0(\cS^+)$ given by
\begin{align*}
\hat{\varphi}(\omega) = \int_G \varphi(g) \omega(g^{-1}) dm_G(g) \, . 
\end{align*}
\end{definition}

One verifies that the spherical transform satisfies $\hat{\varphi^*}(\omega) = \overline{\hat{\varphi}(\omega)}$ for all $\omega \in \cS^+$ and when $\varphi, \psi \in \sL^1(G)$ are \emph{bi-$K$-invariant}, the functional equation in Equation \ref{EqSphericalFunctionFunctionalEquation} can be used to show that
\begin{align}
\label{EqSphericalTransformConvolutionToProductFormula}
\hat{(\psi^* * \varphi)}(\omega) = \hat{\varphi}(\omega) \overline{\hat{\psi}(\omega)} 
\end{align}
for all $\omega \in \cS^+$. We also note that the image $\{ \hat{\varphi} \in C_0(\cS^+) : \varphi \in \sL^1(G, K) \}$ can be shown to be dense in $C_0(\cS^+)$ using the Stone-Weierstrass Theorem, since it separates points and vanishes nowhere.

Before moving on we introduce the bi-$K$-averaging map $\sL^1(G) \rightarrow \sL^1(G, K)$, $\varphi \mapsto \varphi^{\natural}$ by
\begin{align}
\label{EqBiKAveraging}
\varphi^{\natural}(g) = \int_{K \times K} \varphi(k_1 g k_2) dm_K^{\otimes 2}(k_1, k_2) \, , 
\end{align}
which allows us to naturally extend the spherical transform to $\sL^1(G)$ as $\varphi \mapsto \hat{\varphi}^{\natural}$. We note however that the $*$-homomorphism property in Equation \ref{EqSphericalTransformConvolutionToProductFormula} is in general \emph{not} satisfied for functions $\varphi, \psi \in \sL^1(G)$.

\subsection{Spherical Bochner's Theorem}
\label{Spherical Bochner's Theorem}

The classical Bochner Theorem states that a continuous function on $\R^d$ is positive-definite if and only if it is the Fourier transform of a positive finite Borel measure on $\R^d$. The analogue of this statement for Gelfand pairs is the following, and a proof can be found in \cite[Theorem 2.5, p.58]{BarkersThesis}.
\begin{theorem}[Godement-Bochner]
\label{TheoremBochner}
Let $\beta \in C(G)$ be a bi-$K$-invariant positive-definite function. Then there is a unique finite positive Borel measure $\sigma_{\beta}$ on $\cS^+$ such that
\begin{align*}
\beta(g) = \int_{\cS^+} \omega(g) d\sigma_{\beta}(\omega)  \, . 
\end{align*}
\end{theorem}
\begin{remark}
This result in \cite[Theorem 2.5, p.58]{BarkersThesis} is stated for Lie groups, but we emphasize that the proof works verbatim for a general lcsc Gelfand pair $(G, K)$ with $K < G$ compact.
\end{remark}
Writing this result in terms of $K$-spherical matrix coefficients, we can using polarization also define complex-valued spectral measures for the non-diagonal matrix coefficients, see \cite[Prop. 2.53, p. 111]{EinsiedlerWardUnitaryBook} for the abelian case.
\begin{corollary}[Spectral measures for spherical matrix coefficients]
\label{CorollarySpectralMeasuresForMatrixCoefficients}
Let $(\pi, \cH)$ be a $K$-spherical unitary $G$-representation. Then for every $v, w \in \cH^K$ there is a unique complex-valued Borel measure $\sigma_{v, w}$ on $\cS^+$ with total variation $|\sigma_{v, w}|(\cS^+) \leq \norm{v}\norm{w}$ such that
\begin{align*}
\langle \pi(g)v, w \rangle_{\cH} = \int_{\cS^+} \omega(g) d\sigma_{v, w}(\omega)
\end{align*}
for all $g \in G$. When $v = w$, the measure $\sigma_v := \sigma_{v, v}$ is in addition a positive measure.
\end{corollary}

\subsection{$\sL^{\infty}$-action on the spherical unitary dual}
\label{Linfty action on the spherical unitary dual}

If $(\pi, \cH)$ is a unitary $G$-representation, it can be lifted to a bounded Banach $*$-algebra representation $\pi : L^1(G) \rightarrow \sB(\cH)$, uniquely defined by
\begin{align*}
\langle \pi(\beta)v, w \rangle_{\cH} = \int_G \beta(g) \langle \pi(g)v, w \rangle_{\cH} dm_G(g)
\end{align*}
for all $g \in G$ and all $v, w \in \cH$. If $\beta \in L^1(G, K)$,  $(\pi, \cH)$ is $K$-spherical and $v, w \in \cH^K$, Corollary \ref{CorollarySpectralMeasuresForMatrixCoefficients} allows us to write this matrix coefficient as
\begin{align}
\label{EqSphericalFormulaForTheL^1(G)MatrixCoefficients}
\langle \pi(\beta)v, w \rangle_{\cH} = \int_{\cS^+} \hat{\beta}(\omega) d\sigma_{v, w}(\omega) \, . 
\end{align}
We will now extend this family of operators associated to functions in $L^1(G, K)$ to functions in $\sL^{\infty}(\cS^+)$, following the same argument as in \cite[Prop. 2.58, p. 115]{EinsiedlerWardUnitaryBook}.
\begin{lemma}
\label{LemmaSphericalLinftyRepresentation}
Let $(\pi, \cH)$ be a $K$-spherical unitary $G$-representation. Then the map $\hat{\pi} : \sL^{\infty}(\cS^+) \rightarrow \sB(\cH^K)$ given implicitly by
\begin{align*}
\langle \hat{\pi}(\phi)v, w \rangle_{\cH} = \int_{\cS^+} \phi(\omega) d\sigma_{v, w}(\omega)
\end{align*}
for all $v, w \in \cH^K$ is a well-defined $*$-algebra homomorphism. 
\end{lemma}

\begin{proof}
If $\beta \in L^1(G, K)$, then we set $\hat{\pi}(\hat{\beta}) = \pi(\beta)$. Since the image of $L^1(G, K)$ under the spherical transform is dense in $C_0(\cS^+)$ as mentioned in Subsection \ref{The spherical transform}, every $\phi \in \sL^{\infty}(\cS^+)$ can be approximated pointwise by some sequence $\hat{\beta}_n \in C_0(\cS^+)$ with $\beta_n \in L^1(G, K)$. Since $\sigma_{v, w}$ has finite total variation by Corollary \ref{CorollarySpectralMeasuresForMatrixCoefficients} for every $v, w \in \cH^K$, Equation \ref{EqSphericalFormulaForTheL^1(G)MatrixCoefficients} and dominated convergence implies that
\begin{align*}
\langle \hat{\pi}(\phi)v, w \rangle_{\cH} := \lim_{n \rightarrow +\infty} \langle \hat{\pi}(\hat{\beta}_n)v, w \rangle_{\cH} = \lim_{n \rightarrow +\infty} \int_{\cS^+} \hat{\beta}_n(\omega) d\sigma_{v, w}(\omega) = \int_{\cS^+} \phi(\omega) d\sigma_{v, w}(\omega) \, . 
\end{align*}
Moreover, $\sB(\cH^K)$ is preserved under weak limits, so we see that $\hat{\pi}(\phi)$ defines a bounded linear operator on $\cH^K$. To see that this yields a $*$-representation, note that
\begin{align*}
\hat{\pi}(\hat{\beta}_1) \hat{\pi}(\hat{\beta}_2)^* = \pi(\beta_1) \pi(\beta_2)^* = \pi(\beta_1 * \beta_2^*) = \hat{\pi}(\hat{\beta}_1 \overline{\hat{\beta}_2})  
\end{align*}
for all $\beta_1, \beta_2 \in L^1(G, K)$, so if $\phi_1, \phi_2 \in \sL^{\infty}(\cS^+)$ then using the same limiting argument as before yields $\hat{\pi}(\phi_1)\hat{\pi}(\phi_2)^* = \hat{\pi}(\phi_1 \overline{\phi_2})$. 
\end{proof}

\section{Spectral measures}
\label{Spectral measures}

We prove existence and uniqueness of spectral measures as given in Theorem \ref{TheoremSpectralMeasureForAlgebraGMaps} stated below, and compute the mass of such spectral measures at the trivial spherical function $1 \in \cS^+$ in Subsection \ref{Mass of the atom at the constant spherical function}.

Let $(G, K)$ be a lcsc Gelfand pair with $K$ compact and let $\sA(G) \subset \sL^1(G)$ be a $\lambda_G$-invariant $*$-subalgebra such that the $*$-algebra of bi-$K$-averaged functions $\sA(G, K) = \sA(G)^{\natural}$ is contained in $\sA(G)$. 

\noindent\textbf{Assumption 1}: We will assume that for every spherical function $\omega \in \cS^+$, there is a function $\varphi \in \sA(G, K)$ such that 
$$\hat{\varphi}(\omega) \neq 0 \, . $$  

\begin{remark}
The $\lambda_G$-invariance of $\sA(G)$ together with Assumption 1 implies that spherical transforms of functions in $\sA(G, K)$ separate points in $\cS^+$. To see this, let $\omega_1, \omega_2 \in \cS^+$ and suppose $\varphi \in \sA(G, K)$ such that $\hat{\varphi}(\omega_1) = \hat{\varphi}(\omega_2) \neq 0$. Taking $g \in G$ such that $\omega_1(g) \neq \omega_2(g)$ and considering $(\lambda_G(g)\varphi)^{\natural} \in \sA(G, K)$, then a computation utilizing the functional equation in Equation \ref{EqSphericalFunctionFunctionalEquation} yields
\begin{align*}
\hat{(\lambda_G(g)\varphi)^{\natural}}(\omega_1) = \omega_1(g^{-1}) \hat{\varphi}(\omega_1) \neq \omega_2(g^{-1}) \hat{\varphi}(\omega_2) = \hat{(\lambda_G(g)\varphi)^{\natural}}(\omega_2) \, . 
\end{align*}
\end{remark}
Moreover, fix a $K$-spherical unitary $G$-representation $(\pi, \cH)$. We will consider a linear $G$-equivariant map $\alpha : \sA(G) \rightarrow \cH$, meaning 
\begin{align*}
\alpha(\lambda_G(g)\varphi) = \pi(g) \alpha(\varphi) 
\end{align*}
for all $g \in G$ and all $\varphi \in \sA(G)$.
%
%
%
Note that $\alpha$ maps $\sA(G, K)$ to $\cH^K$. We moreover assume that $\alpha$ is a linear $L^1(G)$-map in the sense that
\begin{align*}
\alpha(\lambda_G(\beta)\varphi) = \pi(\beta) \alpha(\varphi)
\end{align*}
for all $\beta \in L^1(G)$. The goal of this Section is to prove the following refinement of Bochner's Theorem.
\begin{theorem}
\label{TheoremSpectralMeasureForAlgebraGMaps}
There is a unique positive Radon measure $\sigma_{\alpha} \in \Radonplus(\cS^+)$ such that
\begin{align*}
\norm{\alpha(\varphi)}_{\cH}^2 = \int_{\cS^+} |\hat{\varphi}(\omega)|^2 d\sigma_{\alpha}(\omega) 
\end{align*}
for all $\varphi \in \sA(G, K)$. 
\end{theorem}

\begin{remark}[The spherical Plancherel measure]
\label{RemarkPlancherelMeasure}
If $\sA(G) \subset \sL^1(G) \cap \sL^2(G)$ satsfies Assumption 1 and $\alpha : \sA(G) \rightarrow L^2(G)$ is the canonical map, then the \emph{spherical Plancherel measure} $\sigma_{\cP} := \sigma_{\alpha}$ defined through Theorem \ref{TheoremSpectralMeasureForAlgebraGMaps} is the unique positive Radon measure on $\cS^+$ satisfying
\begin{align*}
\norm{\varphi}_{L^2(G)}^2 = \int_{\cS^+} |\hat{\varphi}(\omega)|^2 d\sigma_{\cP}(\omega)  
\end{align*}
for all $\varphi \in \sA(G, K)$.
\end{remark}

\subsection{Mutual absolute continuity of Bochner measures}
\label{Mutual absolute continuity of Bochner measures}

By Bochner's Theorem, Theorem \ref{TheoremBochner}, we know that for every $\varphi \in \sA(G, K)$ there is a unique finite positive Borel measure $\sigma_{\alpha(\varphi)}$ on $\cS^+$ such that
\begin{align*}
\langle \pi(g) \alpha(\varphi), \alpha(\varphi) \rangle_{\cH} = \int_{\cS^+} \omega(g) d\sigma_{\alpha(\varphi)}(\omega) \, . 
\end{align*}

Theorem \ref{TheoremSpectralMeasureForAlgebraGMaps} is then equivalent to showing that there is a unique positive Radon measure $\sigma_{\alpha}$ on $\cS^+$ such that
\begin{align*}
d\sigma_{\alpha(\varphi)}(\omega) = |\hat{\varphi}(\omega)|^2 d\sigma_{\alpha}(\omega) 
\end{align*}
as measures for all $\varphi \in \sA(G, K)$.

The first step in proving Theorem \ref{TheoremSpectralMeasureForAlgebraGMaps} is to show that for $\varphi_1, \varphi_2 \in \sA(G, K)$, the measure $\sigma_{\alpha(\varphi_1 * \varphi_2)}$ is absolutely continuous with respect to the measures $\sigma_{\alpha(\varphi_1)}$ and $\sigma_{\alpha(\varphi_2)}$. The precise statement is the following.

\begin{lemma}
\label{LemmaMutualAbsoluteContinuityOfBochnerMeasures}
Let $\varphi_1, \varphi_2 \in \sA(G, K)$. Then for every $\phi \in \sL^{\infty}(\cS^+)$ we have
\begin{align*}
\int_{\cS^+} \phi  |\hat{\varphi}_1|^2 d\sigma_{\alpha(\varphi_2)} = \int_{\cS^+} \phi \, d\sigma_{\alpha(\varphi_1 * \varphi_2)} = \int_{\cS^+} \phi  |\hat{\varphi}_2|^2 d\sigma_{\alpha(\varphi_1)} \, . 
\end{align*}
In particular, $|\hat{\varphi}_1|^2  \sigma_{\alpha(\varphi_2)} = |\hat{\varphi}_2|^2  \sigma_{\alpha(\varphi_1)}$ as measures on $\cS^+$. 
\end{lemma}

\begin{proof}
Since $(G, K)$ is Gelfand and $\varphi_1, \varphi_2 \in \sA(G, K)$, then $\varphi_1 * \varphi_2 = \varphi_2 * \varphi_1$, so it suffices to prove the first equality in the statement. To do this we use Bochner's Theorem for matrix coefficients in Corollary \ref{CorollarySpectralMeasuresForMatrixCoefficients} and the $\sL^{\infty}(\cS^+)$-homomorphism $\hat{\pi}$ from Lemma \ref{LemmaSphericalLinftyRepresentation} to write
\begin{align*}
\int_{\cS^+} \phi \, d\sigma_{\alpha(\varphi_1 * \varphi_2)} = \Big\langle \hat{\pi}(\phi)\alpha(\varphi_1 * \varphi_2), \alpha(\varphi_1 * \varphi_2) \Big\rangle_{\cH} \, . 
\end{align*}
From the $L^1(G)$-equivariance of $\alpha$,
\begin{align*}
\alpha(\varphi_1 * \varphi_2) = \alpha(\lambda_G(\varphi_1) \varphi_2) = \pi(\varphi_1)\alpha(\varphi_2) \, , 
\end{align*}
so that
\begin{align*}
\int_{\cS^+} \phi \, d\sigma_{\alpha(\varphi_1 * \varphi_2)} &= \Big\langle \hat{\pi}(\phi)\alpha(\varphi_1 * \varphi_2), \alpha(\varphi_1 * \varphi_2) \Big\rangle_{\cH} \\
&= \Big\langle \hat{\pi}(\phi)\pi(\varphi_1)\alpha(\varphi_2), \pi(\varphi_1)\alpha(\varphi_2) \Big\rangle_{\cH} \\
&= \Big\langle \pi(\varphi_1^*) \hat{\pi}(\phi)\pi(\varphi_1)\alpha(\varphi_2), \alpha(\varphi_2) \Big\rangle_{\cH} \, . 
\end{align*}
By definition of the $\sL^{\infty}(\cS^+)$-homomorphism $\hat{\pi}$, the operator in the latter inner product can be rewritten using Lemma \ref{LemmaSphericalLinftyRepresentation} as
\begin{align*}
\pi(\varphi_1^*) \hat{\pi}(\phi)\pi(\varphi_1) = \hat{\pi}(\overline{\hat{\varphi}_1}) \hat{\pi}(\phi) \hat{\pi}(\hat{\varphi}_1) = \hat{\pi}(\phi |\hat{\varphi}_1|^2) \, . 
\end{align*}
Finally, using Bochner's Theorem for matrix coefficents in Corollary \ref{CorollarySpectralMeasuresForMatrixCoefficients} again, we get that
\begin{align*}
\int_{\cS^+} \phi \, d\sigma_{\alpha(\varphi_1 * \varphi_2)} &= \Big\langle \hat{\pi}(\phi |\hat{\varphi}_1|^2) \alpha(\varphi_2), \alpha(\varphi_2) \Big\rangle_{\cH} = \int_{\cS^+} \phi |\hat{\varphi}_1|^2 d\sigma_{\alpha(\varphi_2)} \, .
\end{align*}
\end{proof}

\subsection{Extending to the cone of positive spherical transforms}
\label{Extending to the cone of positive spherical transforms}
From Lemma \ref{LemmaMutualAbsoluteContinuityOfBochnerMeasures} we know that $|\hat{\varphi}_1|^2 \cdot \sigma_{\alpha(\varphi_2)} = |\hat{\varphi}_2|^2 \cdot \sigma_{\alpha(\varphi_1)}$ as finite measures on $\cS^+$ for all $\varphi_1, \varphi_2 \in \sA(G, K)$. In proving Theorem \ref{TheoremSpectralMeasureForAlgebraGMaps} we need to extend this identity to what we will call the \emph{cone of positive spherical transforms} for the algebra $\sA(G, K)$, defined as
\begin{align*}
\sC_+(\cS^+) = \Big\{ \sum_{i = 1}^n |\hat{\varphi}_i|^2 \in C_0(\cS^+) :  \varphi_i \in \sA(G, K) \,\, \forall \, i = 1, \dots n \Big\} \, .  
\end{align*}
If $F = \sum_{i = 1}^n |\hat{\varphi}_i|^2 \in \sC_+(\cS^+)$ then we define a finite positive Borel measure $\sigma_F$ on $\cS^+$ by 
\begin{align*}
\sigma_F = \sum_{i = 1}^n \sigma_{\alpha(\varphi_i)} \, . 
\end{align*}
Note that if $F_1, F_2 \in \sC_+(\cS^+)$ are of the form
\begin{align*}
F_1 = \sum_{i = 1}^m |\hat{\varphi}_i|^2 \quad \mbox{ and } F_2 = \sum_{j = 1}^n |\hat{\psi}_j|^2
\end{align*}
then by Lemma \ref{LemmaMutualAbsoluteContinuityOfBochnerMeasures},
\begin{align*}
F_1 \cdot \sigma_{F_2} = \sum_{i = 1}^m \sum_{j = 1}^n |\hat{\varphi}_i|^2 \sigma_{\alpha(\psi_j)} =  \sum_{i = 1}^m \sum_{j = 1}^n |\hat{\psi}_j|^2 \sigma_{\alpha(\varphi_i)} = F_2 \cdot \sigma_{F_1} \, . 
\end{align*}
as measures on $\cS^+$. More concretely, we have that for every $F_1, F_2 \in \sC_+(\cS^+)$ and every $\phi \in L^{\infty}(\cS^+)$,
\begin{align}
\label{EqMainMeasureIdentity}
\int_{\cS^+} \phi(\omega) F_1(\omega) d\sigma_{F_2}(\omega) = \int_{\cS^+} \phi(\omega) F_2(\omega) d\sigma_{F_1}(\omega) \, . 
\end{align}
Note that this equality of measures implies that 
$$\supp(\sigma_{F_1}) \subset \supp(F_1) \quad \mbox{ and } \quad \supp(\sigma_{F_2}) \subset \supp(F_2) \, . $$
The remaining tool for the proof of Theorem \ref{TheoremSpectralMeasureForAlgebraGMaps} is the following Lemma, which is where we will make use of the Assumption 1 for $\sA(G, K)$ in the beginning of Section \ref{Spectral measures}.

\begin{lemma}
\label{LemmaPositiveSphericalTransformsOnCompactSubsets}
For every compact subset $Q \subset \cS^+$ there is an $F \in \sC_+(\cS^+)$ such that $F(\omega) > 0$ for all $\omega \in Q$. 
\end{lemma}

\begin{proof}
By Assumption 1 in the beginning of Section \ref{Spectral measures}, there is for every $\omega \in \cS^+$ a $\varphi_{\omega} \in \sA(G, K)$ such that $\hat{\varphi}_{\omega}(\omega) \neq 0$, so that $|\hat{\varphi}_{\omega}(\omega)|^2 > 0$. Since $\hat{\varphi}_{\omega}$ is continuous, there is an open subset $U_{\omega} \subset \cS^+$ such that $|\hat{\varphi}_{\omega}(\omega')|^2 > 0$ for all $\omega' \in U_{\omega}$, which yields an open cover
\begin{align*}
Q \subset \bigcup_{\omega \in Q} U_{\omega} \, . 
\end{align*}
Since $Q$ is compact, there are $\omega_1, \dots , \omega_n \in Q$ such that
\begin{align*}
Q \subset \bigcup_{i = 1}^n U_{\omega_i} \, , 
\end{align*}
so taking $F = |\hat{\varphi}_{\omega_1}|^2 + \dots + |\hat{\varphi}_{\omega_n}|^2 \in \sC_+(\cS^+)$ finishes the proof.
\end{proof}

\subsection{Proof of Theorem \ref{TheoremSpectralMeasureForAlgebraGMaps}}
\label{Proof of Theorem 3.1}

We now prove Theorem \ref{TheoremSpectralMeasureForAlgebraGMaps} using an argument that can be found for example in \cite[Section 26J-26K, p.98-100]{Loomis1953AnIT}.

\begin{proof}[Proof of Theorem \ref{TheoremSpectralMeasureForAlgebraGMaps}]
As stated in Equation \ref{EqMainMeasureIdentity} we have that 
\begin{align*}
\int_{\cS^+} \phi(\omega) F_1(\omega) d\sigma_{F_2}(\omega) = \int_{\cS^+} \phi(\omega) F_2(\omega) d\sigma_{F_1}(\omega)
\end{align*}
for all $\phi \in \sL^{\infty}(\cS^+)$ and all $F_1, F_2 \in \sC_+(\cS^+)$. By Lemma \ref{LemmaPositiveSphericalTransformsOnCompactSubsets} we may for each $\phi \in C_c(\cS^+)$ take a function $F_{\phi} \in \sC_+(\cS^+)$ that is strictly positive on the compact support $\supp(\phi) \subset \cS^+$ and consider
\begin{align*}
\sigma_{\alpha}(\phi) := \int_{\cS^+} \phi(\omega) \frac{d\sigma_{F_{\phi}}(\omega)}{F_{\phi}(\omega)} \, . 
\end{align*}
We claim that $\sigma_{\alpha}$ is a well-defined positive continuous linear functional on $C_c(\cS^+)$ with respect to the inductive topology over compact subsets.

First, to see that $\sigma_{\alpha}$ is well-defined, suppose that $F_1, F_2 \in \sC_+(\cS^+)$ are strictly positive on $\supp(\phi)$. Then $\phi/F_1, \phi/F_2$ and $\phi/(F_1F_2)$ are compactly supported and continuous, so by Equation \ref{EqMainMeasureIdentity},
\begin{align*}
\int_{\cS^+} \phi(\omega) \frac{d\sigma_{F_{1}}(\omega)}{F_{1}(\omega)} &= \int_{\cS^+} \phi(\omega) \frac{F_2(\omega) d\sigma_{F_{1}}(\omega)}{F_2(\omega)F_{1}(\omega)} \\
&= \int_{\cS^+} \phi(\omega) \frac{F_1(\omega) d\sigma_{F_{2}}(\omega)}{F_2(\omega)F_{1}(\omega)} = \int_{\cS^+} \phi(\omega) \frac{d\sigma_{F_{2}}(\omega)}{F_{2}(\omega)} \, . 
\end{align*}
To see that $\sigma_{\alpha}$ is linear, let $\phi_1, \phi_2 \in C_c(\cS^+)$, $c_1, c_2 \in \C$ and let $F \in \sC_+(\cS^+)$ be strictly positive on $\supp(\phi_1) \cup \supp(\phi_2) \supset \supp(c_1\phi_1 + c_2\phi_2)$, so that
\begin{align*}
\int_{\cS^+} (c_1\phi_1(\omega) + c_2\phi_2(\omega)) \frac{d\sigma_{F}(\omega)}{F(\omega)} = c_1 \int_{\cS^+} \phi_1(\omega) \frac{d\sigma_{F}(\omega)}{F(\omega)} + c_2 \int_{\cS^+} \phi_2(\omega) \frac{d\sigma_{F}(\omega)}{F(\omega)} \, . 
\end{align*}
For the continuity of $\sigma_{\alpha}$, note that if $Q \subset \cS^+$ is compact and $F_Q \in \sC_+(\cS^+)$ is strictly positive on $Q$, then
\begin{align*}
\Big|\int_{\cS^+} \phi(\omega) \frac{d\sigma_{F_Q}(\omega)}{F_Q(\omega)} \Big| \leq \frac{\sigma_{F_Q}(\cS^+)}{\min_{\omega \in Q} F_Q(\omega)} \norm{\phi}_{\infty}
\end{align*}
for all $\phi \in C_c(\cS^+)$ with $\supp(\phi) \subset Q$.
We conclude that $\sigma_{\alpha}$ is a positive continuous linear functional on $C_c(\cS^+)$, and hence a positive Radon measure on $\cS^+$ by the Riesz representation Theorem.

Next, we claim that $\sigma_{F} = F \cdot \sigma_{\alpha}$ as measures on $\cS^+$ for any $F \in \sC_+(\cS^+)$. To see this, let $\phi \in C_b(\cS^+)$ be a positive bounded and continuous function and $F \in \sC_+(\cS^+)$. Take an increasing sequence of positive functions $\chi_n \in C_c(\cS^+)$ such that the compact support $\supp(\chi_n)$ is contained in the open subset $U_{F} = \{ \omega \in \cS^+ : F(\omega) > 0 \}$ and such that $\chi_n \rightarrow 1$ uniformly on $U_{F}$ as $n \rightarrow +\infty$. Then $F$ is strictly positive on the compact support of $F \cdot \chi_n$, so by the monotone convergence Theorem we see that
\begin{align*}
\int_{\cS^+} \phi(\omega) d\sigma_{F}(\omega) &=  \int_{U_{F}} \phi(\omega) F(\omega) \frac{d\sigma_{F}(\omega)}{F(\omega)} \\
&= \lim_{n \rightarrow +\infty} \int_{U_{F}} \phi(\omega) F(\omega) \chi_n(\omega) \frac{d\sigma_{F}(\omega)}{F(\omega)} \\
&= \lim_{n \rightarrow +\infty} \int_{U_{F}} \phi(\omega) F(\omega) \chi_n(\omega) d\sigma_{\alpha}(\omega) = \int_{\cS^+} \phi(\omega) F(\omega) d\sigma_{\alpha}(\omega) \, , 
\end{align*}
which proves the claim. 
Finally, the conclusion of Theorem \ref{TheoremSpectralMeasureForAlgebraGMaps} now follows by taking $F = |\hat{\varphi}|^2$ for $\varphi \in \sA(G, K)$,
\begin{align*}
\norm{\alpha(\varphi)}_{\cH}^2 &= \int_{\cS^+} d\sigma_{\alpha(\varphi)}(\omega) = \int_{\cS^+} d\sigma_F(\omega) = \int_{\cS^+} F(\omega) d\sigma_{\alpha}(\omega) =  \int_{\cS^+} |\hat{\varphi}(\omega)|^2 d\sigma_{\alpha}(\omega) \, . 
\end{align*}
\end{proof}

\subsection{Mass of the atom at the constant spherical function}
\label{Mass of the atom at the constant spherical function}

In this Subsection we show that the spectral measure $\sigma_{\alpha}$ from Theorem \ref{TheoremSpectralMeasureForAlgebraGMaps} has an atom at the constant spherical function $\omega = 1$, corresponding to the trivial $G$-representation, whenever the image $\alpha(\sA(G)) \subset \cH$ projects non-trivially to the subspace $\cH^G$ of $G$-invariant vectors. We will need the following additional assumption on the algebra $\sA(G) \subset \sL^1(G)$ on which $\alpha$ is defined.

\noindent\textbf{Assumption 2}: There is a positive function $\beta \in \sA(G, K)$ whose support $\supp(\beta)$ generates $G$ as a group.

Note that if $v \in \cH^G$ is a $G$-invariant vector, then there is a constant $i(v) \in \C$ such that
\begin{align*}
\langle \alpha(\varphi), v \rangle_{\cH} = i(v) \int_G \varphi(g) dm_G(g) 
\end{align*}
for every $\varphi \in \sA(G)$. In particular, for the projection $v = \Proj_{G}(\alpha(\varphi)) \in \cH^G$ one readily shows that there is a constant $i_{\alpha} \geq 0$ such that 
\begin{align*}
\norm{\Proj_G(\alpha(\varphi))}_{\cH}^2 = i_{\alpha}^2 \, \Big| \int_G \varphi(g) dm_G(g) \Big|^2 \, . 
\end{align*}
We refer to $i_{\alpha}$ as the \emph{intensity} of $\alpha$. The goal of this Subsection is to prove the following.

\begin{proposition}
\label{PropMassOfSpectralMeasureAtOneIsTheIntensity}
The spectral measure of $\alpha$ satisfies $\sigma_{\alpha}(\{1\}) = i_{\alpha}^2$. 
\end{proposition}

To prove this we will let $\beta \in \sA(G, K)$ be a positive function whose support generates $G$ according to Assumption 2 with 
\begin{align*}
\int_G \beta(g) dm_G(g) = 1 \, . 
\end{align*}
For every integer $n \geq 2$ we let
\begin{align*}
\beta_n = \frac{1}{n} \sum_{j = 1}^n \beta^{* j} \in \sA(G, K) \, , 
\end{align*}
where $\beta^{* j}$ is the convolution of $\beta$ with itself $j$ times. We will now take $\varphi \in \sA(G, K)$ such that
\begin{align*}
\int_G \varphi(g) dm_G(g) = 1
\end{align*}
and proceed to compute the limit 
\begin{align*}
\lim_{n \rightarrow +\infty} \norm{\pi(\beta_n)\alpha(\varphi)}_{\cH}^2
\end{align*}
in two ways, resulting in $i_{\alpha}^2$ and $\sigma_{\alpha}(\{1\})$. For both computations we'll need the following lemmata.

\begin{lemma}
\label{LemmaTauBetaFixedVectorsAreGFixedVectors}
Let $(\tau, \cK)$ be a unitary $G$-representation. Then $\cK^{\tau(\beta)} = \cK^G$. 
\end{lemma}

\begin{proof}
First, if $v \in \cK^G$ then
\begin{align*}
\langle \tau(\beta) v, w \rangle_{\cK} = \int_G \beta(g) \langle \tau(g) v, w \rangle_{\cK} dm_G(g) = \langle v, w \rangle_{\cK} 
\end{align*}
for all $w \in \cK$, using that the integral of $\beta$ with respect to $m_G$ is $1$, so $v \in \cK^{\tau(\beta)}$. 

Conversely, if $v \in \cK^{\tau(\beta)}$ then 
\begin{align*}
\int_G \beta(g) \langle \tau(g) v, w \rangle_{\cK} dm_G(g) = \langle v, w \rangle_{\cK}
\end{align*}
for all $w \in \cK$. Letting $w = v$ and taking real parts of both sides yields
\begin{align*}
0 = \int_G \beta(g) (\norm{v}^2 - \Re\langle \tau(g)v, v \rangle_{\cK} ) dm_G(g) = \frac{1}{2} \int_G \beta(g) \norm{\tau(g)v - v}_{\cK}^2 dm_G(g) \, , 
\end{align*}
so $\tau(g)v = v$ for all $g \in \supp(\beta)$. Since the support $\supp(\beta)$ generates $G$ and $\tau$ is a group homomorphism then $\tau(g)v = v$ for all $g \in G$, so $v \in \cK^G$. 
\end{proof}

Since $\tau(\beta)$-invariant vectors are $G$-invariant and vice versa, we can formulate the mean ergodic Theorem for unitary $G$-representations.

\begin{lemma}[Mean ergodic Theorem]
\label{LemmaMeanErgodicTheoremForUnitaryRepresentations}
Let $(\tau, \cK)$ be a unitary $G$-representation. Then, for every $v \in \cK$,
\begin{align*}
\norm{\tau(\beta_n)v - \Proj_{G}(v)}_{\cK} \underset{n \rightarrow +\infty}{\longrightarrow} 0 \, . 
\end{align*} 
\end{lemma}

\begin{proof}[Proof of Proposition \ref{PropMassOfSpectralMeasureAtOneIsTheIntensity}]
First we make use of the mean ergodic Theorem in Lemma \ref{LemmaMeanErgodicTheoremForUnitaryRepresentations} to find that 
\begin{align*}
\lim_{n \rightarrow +\infty} \norm{\pi(\beta_n)\alpha(\varphi)}_{\cH}^2 &= \norm{\Proj_G(\alpha(\varphi))}_{\cH}^2 = i_{\alpha}^2 \Big| \int_G \varphi(g) dm_G(g) \Big|^2 = i_{\alpha}^2 \, . 
\end{align*}
On the other hand, if $\omega \in \cS^+$ is a positive-definite spherical function with associated $K$-spherical irreducible GNS-representation $(\pi_{\omega}, \cH_{\omega}, v_{\omega})$, then
\begin{align*}
\hat{\beta}_n(\omega) = \langle \pi_{\omega}(\beta_n)v_{\omega}, v_{\omega} \rangle_{\cH_{\omega}} \underset{n \rightarrow +\infty}{\longrightarrow} \langle \Proj_G(v_{\omega}), v_{\omega} \rangle_{\cH_{\omega}} = \begin{cases} 1 &\mbox{ if } \omega = 1 \\ 0 &\mbox{ else }. \end{cases} 
\end{align*}
Thus $\hat{\beta}_n \rightarrow \chi_{\{1\}}$ pointwise on $\cS^+$. Since 
\begin{align*}
|\hat{\beta}_n(\omega)|^2 \leq \norm{v_{\omega}}_{\cH_{\omega}}^2 \Big| \int_G \beta(g) dm_G(g) \Big|^2 = 1
\end{align*}
for every $\omega \in \cS^+$ and 
\begin{align*}
\int_{\cS^+} |\hat{\varphi}(\omega)|^2 d\sigma_{\alpha}(\omega) = \norm{\alpha(\varphi)}_{\cH}^2 < +\infty
\end{align*}
by Theorem \ref{TheoremSpectralMeasureForAlgebraGMaps} then Theorem \ref{TheoremSpectralMeasureForAlgebraGMaps} and dominated convergence allows us to compute the desired limit as 
\begin{align*}
\lim_{n \rightarrow +\infty} \norm{\pi(\beta_n)\alpha(\varphi)}_{\cH}^2 &= \lim_{n \rightarrow +\infty} \norm{\alpha(\lambda_G(\beta_n)\varphi)}_{\cH}^2 = \lim_{n \rightarrow +\infty} \norm{\alpha(\beta_n * \varphi)}_{\cH}^2 \\
&= \lim_{n \rightarrow +\infty} \int_{\cS^+} |\hat{\beta}_n(\omega)|^2 |\hat{\varphi}(\omega)|^2 d\sigma_{\alpha}(\omega) \\
&= |\hat{\varphi}(1)|^2 \sigma_{\alpha}(\{1\}) = \sigma_{\alpha}(\{1\}) \, .
\end{align*}
\end{proof}

\section{Random measures on commutative spaces}
\label{Random measures on commutative spaces}

The commutative metric spaces that we consider are introduced in the first Subsection, and in Subsection \ref{Invariant random measures} we define invariant locally square-integrable random measures on such spaces as well as the associated Bartlett spectral measures.

\subsection{Commutative spaces}
\label{Commutative spaces}

Let $(X, d)$ be a non-compact proper metric space, fix a basepoint $x_o \in X$ and let $\Isom(X, d)$ denote the group of isometries of $(X, d)$ endowed with the topology of pointwise convergence, which is lcsc by \cite[Lemma 5.B.4, p.143]{CornulierDeLaHarpeIsometryGroupsOfProperMetricSpaces}. Suppose that there is a closed unimodular subgroup $G < \Isom(X, d)$ of isometries acting transitively on $X$, so that $X$ is homeomorphic to $G/K$ with $K = \Stab_G(x_o)$ being the stabilizer of $x_o$ in $G$. Assume in addition that $K < G$ is compact. The space $X$ is called \emph{commutative} if $(G, K)$ is a Gelfand pair. We fix a Haar measure $m_G$ on $G$, the Haar probability measure $m_K$ on $K$ and consider the following $G$-invariant positive Radon measure on $X$,
\begin{align*}
m_X(B) = \int_G \chi_B(g.x_o) dm_G(g) \, , \quad B \subset X \mbox{ Borel}\, . 
\end{align*}
Note that $m_X$ is the unique positive $G$-invariant measure on $X$ up to scaling. We will assume the following small volume property on $m_X$, known as \emph{finite upper local dimension}, see for example \cite[Definition 3.5, p.21]{GorodnikNevoBook}.

\noindent\textbf{Assumption 2}: There are constanst $m_o, d > 0$ such that $m_X(B_{\varepsilon}(x_o)) \geq m_o\varepsilon^d$ as $\varepsilon \rightarrow 0^+$.

Regarding functions, if $f : X \rightarrow \C$ is a measurable function on $X$ we write $\varphi_f : G \rightarrow \C$ for the corresponding right-$K$-invariant measurable function on $G$ defined by $\varphi_f(g) = f(g.x_o)$ for all $g \in G$. Note that $\varphi_f$ is compactly supported if and only if $f$ is. In order to reverse this correspondence, we will use the existence of a Borel measurable \emph{section} $s : X \rightarrow G$ satisfying $s(x).x_o = x$ for every $x \in X$ and such that $s(B) \subset G$ is pre-compact for every bounded $B \subset X$, see \cite[Lemma 1.1]{MackeyInducedRepsI}. Fixing such a section $s$, we denote by $P_K : \Borelinfty(G) \rightarrow \Borelinfty(X)$ the right-$K$-averaging operator
\begin{align*}
P_K\varphi(x) = \int_K \varphi(s(x)k) dm_K(k) \, . 
\end{align*}
Since $K < G$ is a compact subgroup, $P_K$ takes compactly supported functions to compactly supported functions. Moreover, we will denote the left-regular action of $G$ on measurable functions $f : X \rightarrow \C$ by $\lambda_X(g)f(x) = f(g^{-1}.x)$.

\subsection{Invariant random measures}
\label{Invariant random measures}

Denote by $\Radonplus(X)$ the space of positive Radon measures on $X$, endowed with the smallest $\sigma$-algebra such that for every $f \in \Borelbndinfty(X)$, the \emph{linear statistic} $\bS f : \Radonplus(X) \rightarrow \C$ given by
\begin{align*}
\bS f(p) = \int_X f(x) dp(x) \, , \quad p \in \Radonplus(X)
\end{align*}
is measurable. By a \emph{random measure on $X$} we mean a probability measure $\mu$ on $\Radonplus(X)$. Such a random measure $\mu$ is \emph{invariant} if $g_*\mu = \mu$ for all $g \in G$ and \emph{locally $q$-integrable} for $q \in [1, +\infty)$ if $\bS f \in L^q(\mu)$ for every $f \in \Borelbndinfty(X)$, in other words
\begin{align*}
\int_{\Radonplus(X)} \Big| \int_X f(x) dp(x) \Big|^q d\mu(p) < +\infty \, . 
\end{align*}
Note that a locally $q$-integrable random measure is also locally $q'$-integrable for every $q' \leq q$. 

\textit{We will from now on assume that random measures $\mu$ are invariant locally and square-integrable unless told otherwise.}

For such a random measure $\mu$ we can consider its \emph{$\mu$-expectation}, which by the uniqueness of the invariant measure $m_X$ on $X$ up to scaling is 
\begin{align*}
\Emu(\bS f) = \int_{\Radonplus(X)} \Big( \int_X f(x) dp(x) \Big) d\mu(p) = i_{\mu} \int_X f(x) dm_X(x) \, , \quad f \in \Borelbndinfty(X)
\end{align*}
where $i_{\mu} > 0$ is the \emph{intensity} of $\mu$. With the $\mu$-expectation defined we introduce the \emph{discrepancy statistic} of a function $f \in \Borelbndinfty(X)$ as
\begin{align*}
\bS_0 f = \bS f - i_{\mu} \int_X f(x) dm_X \in L^2_0(\mu)   
\end{align*}
and define the \emph{$\mu$-variance} of a linear statistic $\bS f$, $f \in \Borelbndinfty(X)$, as the squared $L^2$-norm of the corresponding discrepancy statistic,
\begin{align*}
\Varmu(\bS f) = \norm{\bS_0 f}_{L^2(\mu)}^2 = \int_{\Radonplus(X)} \Big|\bS f(p) - i_{\mu}\int_X f(x) dm_X(x) \Big|^2 d\mu(p) \, . 
\end{align*}
To put ourselves into the context of Section \ref{Spectral measures}, consider the associated \emph{Koopman representation} $\pi_{\mu} : G \rightarrow \sU(L^2(\mu))$ given by
\begin{align*}
\pi_{\mu}(g)\Psi(p) = \Psi(g^{-1}_*p) \, , \quad g \in G \, , \Psi \in L^2(\mu) \, .
\end{align*}
The lifted $*$-algebra representation $\pi_{\mu} : L^1(G) \rightarrow \sB(L^2(\mu))$ is 
\begin{align*}
\pi_{\mu}(\varphi)\Psi(p) = \int_G \varphi(g) \Psi(g^{-1}_*p) dm_G(g) 
\end{align*}
and if $\Psi = \bS f$ is a linear statistic then one verifies the equivariance properties
\begin{align*}
\pi_{\mu}(g)\bS f(p) = \bS(\lambda_X(g)f)(p) \, , \quad \pi_{\mu}(\varphi)\bS f(p) = \bS(\lambda_X(\varphi)f)(p) \, . 
\end{align*}
With the unitary $G$-representation $(\pi_{\mu}, L^2(\mu))$ and the algebra $\sA(G) = \Borelbndinfty(G)$ we are in the setting of Section \ref{Spectral measures}. By Theorem \ref{TheoremSpectralMeasureForAlgebraGMaps} we have the following:
\begin{enumerate}
    \item For the map $\alpha : \Borelbndinfty(G) \rightarrow L^2(\mu)$ given by
    \begin{align*}
    \alpha(\varphi)(p) = \bS(P_K \varphi)(p) = \int_{X \times K} \varphi(s(x)k) d(p \otimes m_K)(x, k) \, , 
    \end{align*}
    there is a measure $\sigma_{\mu}^+ := \sigma_{\alpha} \in \Radonplus(\cS^+)$ such that 
    \begin{align*}
    \Emu(|\bS f|^2) = \norm{\alpha(\varphi_f)}_{L^2(\mu)}^2 = \int_{\cS^+} |\hat{\varphi}_f(\omega)|^2 d\sigma_{\mu}^+(\omega) \, , \quad f \in \Borelbndinfty(X)^K \, . 
    \end{align*}
    \item For $\alpha_0 : \Borelbndinfty(G) \rightarrow L^2_0(\mu)$ given by
    \begin{align*}
    \alpha_0(\varphi)(p) = \alpha(\varphi)(p) - i_{\mu} \int_G \varphi(g) dm_G(g) = \bS_0(P_K \varphi)(p)\, , 
    \end{align*}
    there is a measure $\sigma_{\mu} := \sigma_{\alpha_0} \in \Radonplus(\cS^+)$ such that
    \begin{align*}
    \Varmu(\bS f) = \norm{\alpha_0(\varphi_f)}_{L^2(\mu)}^2 = \int_{\cS^+} |\hat{\varphi}_f(\omega)|^2 d\sigma_{\mu}(\omega) \, , \quad f \in \Borelbndinfty(X)^K \, . 
    \end{align*}
    The precise relation between the two spectral measures is $\sigma_{\mu} = \chi_{\cS^+ \backslash \{1\}} \cdot \sigma_{\mu}^+$ by Proposition \ref{PropMassOfSpectralMeasureAtOneIsTheIntensity}, and we refer to $\sigma_{\mu}$ as the \emph{Bartlett spectral measure} of $\mu$.
\end{enumerate}

\section{Upper bounds for number variances with spectral gap}
\label{Upper bounds for number variances with spectral gap}

In this Section we prove Theorem \ref{Theorem2} using important techniques developed by Gorodnik-Nevo in \cite{GorodnikNevoBook}. We will first provide upper and lower bounds on counting statistics over general bounded Borel sets in proper metric homogeneous spaces as in the last Section, which will then provide us with a regularized upper bound on the variance of counting statistics. Then in Subsection \ref{Spherical integrability coefficients} we make assumptions on the uniform integrability of spherical functions in the support of the Bartlett spectral measure, followed by an admissibility assumption on counting statistics in Subsection \ref{Admissible sequences} to get sharper explicit bounds. Finally, we conclude the proof of Theorem \ref{Theorem2} in Subsection \ref{Proof of Theorem 1.3}.

\subsection{Upper and lower bounds for counting statistics}
\label{Upper and lower bounds for counting statistics}

Let $V \subset G$ denote an open pre-compact neighbourhood of the identity $e \in G$ and let $\rho_V \in \Borelbndinfty(G)$ be a right-$K$-invariant positive function with support in $V$ such that
\begin{align*}
\int_G \rho_V(g) dm_G(g) = 1 \, . 
\end{align*}
If $\tilde{D} \subset G$ is a Borel set, we let $D = \tilde{D}.o$ and define
\begin{align*}
\tilde{D}_V^- = \bigcap_{v \in V} v\tilde{D} \, , \quad \tilde{D}_V^+ = \bigcup_{v \in V} v\tilde{D} = V\tilde{D} \, . 
\end{align*}
Then $\tilde{D}_V^-$ is closed in $G$ and $\tilde{D}_V^+$ is open in $G$, and the following inclusions hold,
\begin{align*}
\bigcup_{v \in V} v\tilde{D}_V^- \subset \tilde{D} \subset \bigcap_{v \in V} v\tilde{D}_V^+ \, . 
\end{align*}

\begin{lemma}
\label{LemmaUpperAndLowerConvolutionBoundsOnIndicatorFunctions}
Let $V \subset G$ be an open pre-compact symmetric neighbourhood of the identity. Then for every Borel set $\tilde{D} \subset G$ with $D = \tilde{D}.o$,
\begin{align*}
(\rho_V * \chi_{\tilde{D}_V^-})(s(x)) \leq \chi_{D}(x) \leq (\rho_V * \chi_{\tilde{D}_V^+})(s(x)) 
\end{align*}
for every $x \in X$. 
\end{lemma}

\begin{proof}
For the first inequality, a change of variables yields
\begin{align*}
(\rho_V * \chi_{\tilde{D}_V^-})(s(x)) = \int_{\tilde{D}_V^-} \rho_V(s(x)g^{-1}) dm_G(g) = \int_{s(x)(\tilde{D}_V^-)^{-1} \cap V} \rho_V(g) dm_G(g) \, . 
\end{align*}
Note that if $s(x)(\tilde{D}_V^-)^{-1} \cap V \neq \varnothing$ then $s(x) \in V \tilde{D}_V^- \subset \tilde{D}$, so that $x \in D$. Since $\rho_V$ integrates to $1$ over $G$, we get that 
\begin{align*}
(\rho_V * \chi_{\tilde{D}_V^-})(s(x)) = \chi_D(x) \int_{s(x)(\tilde{D}_V^-)^{-1} \cap V} \rho_V(g) dm_G(g) \leq \chi_D(x) \, . 
\end{align*}
For the second inequality, we similarly write
\begin{align*}
(\rho_V * \chi_{\tilde{D}_V^+})(s(x)) = \int_{\tilde{D}_V^+} \rho_V(s(x)g^{-1}) dm_G(g) = \int_{s(x)(\tilde{D}_V^+)^{-1} \cap V} \rho_V(g) dm_G(g) \, . 
\end{align*}
If $x \in D$, then $s(x) \in \tilde{D} \subset v \tilde{D}_V^+$ for every $v \in V$, meaning that $V \subset s(x) (\tilde{D}_V^+)^{-1}$ and hence
\begin{align*}
\chi_D(x) \leq \int_{s(x)(\tilde{D}_V^+)^{-1} \cap V} \rho_V(g) dm_G(g) = (\rho_V * \chi_{\tilde{D}_V^+})(s(x)) \, . 
\end{align*}
\end{proof}

It follows from Lemma \ref{LemmaUpperAndLowerConvolutionBoundsOnIndicatorFunctions} that the linear statistic $\bS\chi_{D}$ can be bounded from above and below by
\begin{align}
\label{EqLinearStatisticOfIndicatorUpperAndLowerBounds}
\bS(P_K(\rho_V * \chi_{\tilde{D}_V^-})) \leq \bS\chi_{D} \leq \bS(P_K(\rho_V * \chi_{\tilde{D}_V^+}))
\end{align}
where again $P_K\varphi(x) = \int_K \varphi(s(x)k) dm_K(k)$ for $\varphi \in \Borelbndinfty(G)$. We will next bound the discrepancy statistic
\begin{align*}
\bS_0 \chi_D = \bS\chi_D - i_{\mu} m_X(D)
\end{align*}
from above to produce the following upper bound on the number variance of $D$. To keep sight of the goal of this Section, which is to prove Theorem \ref{Theorem2}, we remind ourselves that
\begin{align*}
\Varmu(\bS \chi_D) = \norm{\bS_0 \chi_D}_{L^2(\mu)}^2  \, . 
\end{align*}

\begin{corollary}
\label{CorollaryTriangleInequalityUpperBoundOnL2Discrepancy}
Let $\tilde{D} \subset G$ be a Borel set with $D = \tilde{D}.o$. Then
\begin{align*}
\norm{\bS_0\chi_D}_{L^2(\mu)} \leq \norm{\bS_0(P_K(\rho_V * \chi_{\tilde{D}_V^-}))}_{L^2(\mu)} + \norm{\bS_0(P_K(\rho_V * \chi_{\tilde{D}_V^+}))}_{L^2(\mu)} + i_{\mu} m_G(\tilde{D}_V^+ \backslash \tilde{D}_V^-)  \, . 
\end{align*}
\end{corollary}

\begin{proof}
The inequalities in Equation \ref{EqLinearStatisticOfIndicatorUpperAndLowerBounds} can be written in terms of discrepancies as
\begin{align*}
\bS_0(P_K(\rho_V * \chi_{\tilde{D}_V^-})) + i_{\mu}(m_G(\rho_V &* \chi_{\tilde{D}_V^-}) - m_G(\tilde{D})) \\
&\leq \bS_0\chi_D \\
&\leq \bS_0(P_K(\rho_V * \chi_{\tilde{D}_V^+})) + i_{\mu}(m_G(\rho_V * \chi_{\tilde{D}_V^+}) - m_G(\tilde{D})) \, .
\end{align*}
Note that 
\begin{align*}
m_G(\rho_V * \chi_{\tilde{D}_V^{\pm}}) = m_G(\rho_V) m_G(\tilde{D}_V^{\pm}) = \begin{cases} m_G(\tilde{D}) + m_G(\tilde{D}_V^+\backslash\tilde{D}) &\mbox{ if } + \\ m_G(\tilde{D}) - m_G(\tilde{D}\backslash\tilde{D}_V^-) &\mbox{ if } - \, ,\end{cases}
\end{align*}
which yields the inequalities
\begin{align*}
\bS_0(P_K(\rho_V * \chi_{\tilde{D}_V^-})) - i_{\mu} m_G(\tilde{D}\backslash\tilde{D}_V^-) \leq \bS_0\chi_D \leq \bS_0(P_K(\rho_V * \chi_{\tilde{D}_V^+})) + i_{\mu} m_G(\tilde{D}^+_V\backslash\tilde{D})\, .
\end{align*}
Using the triangle inequality, we get that 
\begin{align*}
\norm{\bS_0 \chi_D}_{L^2(\mu)} \leq \norm{\bS_0(P_K(\rho_V * \chi_{\tilde{D}_V^-}))}_{L^2(\mu)} + \norm{\bS_0(P_K(\rho_V * \chi_{\tilde{D}_V^+}))}_{L^2(\mu)} + i_{\mu} m_G(\tilde{D}_V^+\backslash\tilde{D}_V^-)
\end{align*}
as desired.
\end{proof}


\subsection{Spherical integrability coefficients}
\label{Spherical integrability coefficients}

The main property that we will need for the proof of Theorem \ref{Theorem2} is the following notion of uniform $L^q$-integrability of the positive-definite spherical functions for $(G, K)$ for some $q \in [2, +\infty)$.

\begin{definition}[Spherical integrability coefficients]
\label{DefSphericalSpectralGap}
Let $(\pi, \cH)$ be a $K$-spherical unitary $G$-representation and $\alpha : \sA(G) \rightarrow \cH$ a linear continuous $G$-map as in Section \ref{Spectral measures}. If $q \in [2, +\infty]$ then the \emph{spherical $L^q$-coefficient} of $\alpha$ is 
\begin{align*}
I_q(\alpha) := \sup\Big\{ \norm{\omega}_{L^q(G)} : \omega \in \supp(\sigma_{\alpha}) \Big\} \in (0, +\infty] \, . 
\end{align*}
If $\mu$ is an invariant locally square-integrable random measure on $X = G/K$, then we write $I_q(\mu) := I_q(\alpha_0)$, where we recall $\alpha_0(\varphi) = \bS_0(P_K\varphi) \in L^2_0(\mu)$ for every $\varphi \in \Borelbndinfty(G)$.  
\end{definition}

\begin{remark}
\label{RemarkMatrixCoefficientBoundsForHigherRankSimpleLieGroups}
If $G$ is a higher rank connected simple matrix Lie group with finite center and $K < G$ maximal compact then there are explicit quantitative bounds on matrix coefficients of unitary $G$-representations with $K$-finite vectors, found for example in Borel and Wallach's book \cite{BorelWallachBook}. For a proof of the following we refer to \cite[Theorem 4.1]{GorodnikHigherOrderCorrelations}. In the case of $K$-spherical representations, the statement is that there is a constant $\varepsilon = \varepsilon(G) > 0$ such that
\begin{align*}
|\omega(g)| \leq C \e^{-\varepsilon d(g.o, o)} \, , \quad \forall \, \omega \in \cS^+ \backslash \{1\} \, , 
\end{align*}
where $d$ is the invariant reference metric on $X = G/K$ and $C > 0$ is independent of $\omega \in \cS^+ \backslash \{1\}$. In particular, since $X = G/K$ has at most exponential volume growth, meaning that there is a $\delta_X > 0$ such that 
\begin{align*}
\int_G \e^{-\delta_X d(g.o, o)} dm_G(g) < +\infty
\end{align*}
then taking $q \geq \delta_X/\varepsilon$ we get that
\begin{align*}
\int_G |\omega(g)|^q dm_G(g) \leq C^q \int_G \e^{- \delta_X d(g.o, o)} dm_G(g) < +\infty 
\end{align*}
for all $\omega \in \cS^+ \backslash \{1\}$. In particular, $I_q(\alpha) < +\infty$ for \emph{every} choice of $\alpha$ and hence $I_q(\mu) < +\infty$ for every invariant locally square-integrable random measure $\mu$. 
\end{remark}

\begin{lemma}
\label{LemmaFiniteLqCoefficientImpliesBoundOnConvolutionStatistic}
Let $\mu$ be an invariant locally square-integrable random measure on $X$ with finite spherical $L^q$-coefficient $I_q(\mu) < +\infty$ for some $q \in [2, +\infty)$. Then, for any Borel set $\tilde{B} \subset G$ and any $\rho \in \Borelbndinfty(G)$,
\begin{align*}
\norm{\bS_0(P_K(\rho * \chi_{\tilde{B}}))}_{L^2(\mu)} \leq m_G(\tilde{B})^{\frac{q - 1}{q}} I_q(\mu) \norm{\bS_0(P_K\rho)}_{L^2(\mu)} \, . 
\end{align*}
\end{lemma}

\begin{proof}
By defintion of the diffraction measure $\sigma_{\mu}$ on $\cS^+$ we have 
\begin{align*}
\norm{\bS_0(P_K(\rho * \chi_{\tilde{B}}))}_{L^2(\mu)} = \norm{\hat{\rho} \cdot \hat{\chi}_{\tilde{B}}}_{L^2(\sigma_{\mu})} \, . 
\end{align*}
By Hölder's inequality, we can bound the $L^{\infty}(\sigma_{\mu})$-norm of $\hat{\chi}_{\tilde{B}}$ by
\begin{align*}
\norm{\hat{\chi}_{\tilde{B}}}_{L^{\infty}(\sigma_{\mu})} = \sup_{\omega \in \supp(\sigma_{\mu})} \Big| \int_{\tilde{B}} \omega(g^{-1}) dm_G(g) \Big| \leq m_G(\tilde{B})^{\frac{q-1}{q}} I_q(\mu) 
\end{align*}
to get the desired bound,
\begin{align*}
\norm{\bS_0(P_K(\tilde{\rho} * \chi_{\tilde{B}}))}_{L^2(\mu)} \leq  m_G(\tilde{B})^{\frac{q-1}{q}} I_q(\mu) \norm{\hat{\rho}}_{L^2(\sigma_{\mu})} = m_G(\tilde{B})^{\frac{q - 1}{q}} I_q(\mu) \norm{\bS_0(P_K\rho)}_{L^2(\mu)}\, . 
\end{align*}
\end{proof}

\subsection{Admissible sequences}
\label{Admissible sequences}

To proceed, we will consider $V = \cO_{\varepsilon}$, where
\begin{align*}
\cO_{\varepsilon} = \Big\{ g \in G : d(g.o, o) < \varepsilon \Big\} \, . 
\end{align*}
We will also consider ''nice'' sequences Borel sets that behave well with respect to small perturbations by elements of $G$. The following definition is from \cite[Definition 1.1]{GorodnikNevoBook}.

\begin{definition}
An unbounded increasing sequence $(\tilde{D}_t)_{t > 0}$ of Borel subsets in $G$ is \emph{admissible} if there are constants $c, t_o, \varepsilon_o > 0$ such that for every $t > t_o$ and every $0 < \varepsilon < \varepsilon_o$,
\begin{align*}
\cO_{\varepsilon}\tilde{D}_t \cO_{\varepsilon} \subset \tilde{D}_{t + c\varepsilon} \quad \mbox{ and } \quad m_G(\tilde{D}_{t + \varepsilon}) \leq (1 + c\varepsilon) m_G(\tilde{D}_t) \, . 
\end{align*}
\end{definition}

\begin{remark}
\label{RemarkLiftedMetricBallsFormAnAdmissibleSequence}
In \cite[Prop. 3.15]{GorodnikNevoBook} it is shown for Riemannian symmetric spaces (and more generally a large class of symmetric spaces coming from $S$-algebraic groups) that lifted $K$-invariant metric balls 
$$\tilde{B}_r = \{ g \in G : g.o \in B_r(x_o) \}$$
form an admissible sequence indexed by the radius $r  > 0$. 
\end{remark}

One verifies from the definition of admissibility that
\begin{align*}
\tilde{D}_{t - c\varepsilon} \subset (\tilde{D}_t)_{\cO_{\varepsilon}}^- \subset \tilde{D}_t \subset (\tilde{D}_t)_{\cO_{\varepsilon}}^+ \subset \tilde{D}_{t + c\varepsilon} \, , 
\end{align*}
and in the proof of following Lemma, we fix
\begin{align*}
\rho_{\cO_{\varepsilon}} = \frac{\chi_{\cO_{\varepsilon}}}{m_G(\cO_{\varepsilon})} \, . 
\end{align*}

\begin{lemma}
\label{LemmaVarianceUpperBoundForAdmissibleSequences}
Let $(\tilde{D}_t)_{t > 0}$ be an admissible sequence of Borel sets in $G$ and set $D_t = \tilde{D}_t.o \subset X$. Suppose that $\mu$ is an invariant locally square-integrable random measure on $X$ with finite spherical $L^q$-coefficient $I_q(\mu) < +\infty$ for some $q \in [2, +\infty)$. Then there is a constant $C = C(\mu, q) > 0$ such that for sufficiently large $t > 0$ and sufficiently small $\varepsilon > 0$,
\begin{align*}
\norm{\bS_0\chi_{D_t}}_{L^2(\mu)} \leq C \Big( \frac{m_G(\tilde{D}_t)^{\frac{q-1}{q}}}{m_G(\cO_{\varepsilon})} + \varepsilon m_G(\tilde{D}_t) \Big) \, . 
\end{align*}
\end{lemma}

\begin{proof}
Let $t > t_o$. By Corollary \ref{CorollaryTriangleInequalityUpperBoundOnL2Discrepancy} and Lemma \ref{LemmaFiniteLqCoefficientImpliesBoundOnConvolutionStatistic} we have that 
\begin{align*}
\norm{\bS_0\chi_{D_t}}_{L^2(\mu)} \leq \,\,  &m_G((\tilde{D}_t)_{\cO_{\varepsilon}}^-)^{\frac{q - 1}{q}} I_q(\mu) \norm{\bS_0(P_K\rho_{\cO_{\varepsilon
}})}_{L^2(\mu)} \\
&+  m_G((\tilde{D}_t)_{\cO_{\varepsilon}}^+)^{\frac{q - 1}{q}} I_q(\mu) \norm{\bS_0(P_K\rho_{\cO_{\varepsilon}})}_{L^2(\mu)} \\
&+ i_{\mu} m_G((\tilde{D}_t)_{\cO_{\varepsilon}}^+ \backslash (\tilde{D}_t)_{\cO_{\varepsilon}}^-) \, . 
\end{align*}
Using that the sequence $\tilde{D}_t$ is admissible, we have that
\begin{align*}
&m_G((\tilde{D}_t)_{\cO_{\varepsilon}}^-) \leq m_G(\tilde{D}_t), \\
&m_G((\tilde{D}_t)_{\cO_{\varepsilon}}^+) \leq m_G(\tilde{D}_{t + c\varepsilon}) \leq (1 + c^2\varepsilon)m_G(\tilde{D}_t), \quad \mbox{ and } \\
&m_G((\tilde{D}_t)_{\cO_{\varepsilon}}^+ \backslash (\tilde{D}_t)_{\cO_{\varepsilon}}^-) \leq \Big( (1 + c^2\varepsilon) - \frac{1}{1 + c^2\varepsilon} \Big) m_G(\tilde{D}_t) \leq c^2(c^2 + 2)\varepsilon m_G(\tilde{D}_t) \, . 
\end{align*}
Moreover, using that we specified $\rho_{\cO_{\varepsilon}} = m_G(\cO_{\varepsilon})^{-1} \chi_{\cO_{\varepsilon}}$ along with the triangle inequality yields
\begin{align*}
\norm{\bS_0(P_K\rho_{\cO_{\varepsilon}})}_{L^2(\mu)} \leq \norm{\bS(P_K\rho_{\cO_{\varepsilon}})}_{L^2(\mu)} + i_{\mu} \leq \frac{\norm{\bS(P_K\chi_{\cO_{1}})}_{L^2(\mu)}}{m_G(\cO_{\varepsilon})} + i_{\mu} \leq 2 \frac{\norm{\bS(P_K\chi_{\cO_{1}})}_{L^2(\mu)}}{m_G(\cO_{\varepsilon})}
\end{align*}
for sufficiently small $\varepsilon > 0$. Gathering all of these estimates we get 
\begin{align*}
\norm{\bS_0\chi_{D_t}}_{L^2(\mu)} \leq 2 (2 + c^2\varepsilon) I_q(\mu) \norm{\bS(P_K\chi_{\cO_{1}})}_{L^2(\mu)} \frac{m_G(\tilde{D}_t)^{\frac{q-1}{q}}}{m_G(\cO_{\varepsilon})} + c^2(c^2 + 2)\varepsilon m_G(\tilde{D}_t) \, . 
\end{align*}
Taking 
\begin{align*}
C(\mu, q) = \max\Big\{ 2 (c^2 + 2) I_q(\mu) \norm{\bS(P_K\chi_{\cO_{1}})}_{L^2(\mu)}, c^2(c^2 + 2)\Big\}
\end{align*}
finishes the proof.
\end{proof}

\subsection{Proof of Theorem \ref{Theorem2}}
\label{Proof of Theorem 1.3}

Assume that $G$ is a connected simple matrix Lie group with finite center of real rank greater or equal to $2$ and let $K < G$ be a maximal compact subgroup. From finite upper local dimension in Assumption 2 of Section \ref{Random measures on commutative spaces}, we have that
\begin{align*}
m_G(\cO_{\varepsilon}) = m_X(B_{\varepsilon}(x_o)) \geq m_o\varepsilon^{d}   
\end{align*}
as $\varepsilon \rightarrow 0^+$ for some constants $m_o, d > 0$. By Remark \ref{RemarkMatrixCoefficientBoundsForHigherRankSimpleLieGroups} there is a $q \in [2, +\infty)$ such that $I_q(\mu) < +\infty$ for any invariant locally square-integrable random measure $\mu$. Moreover, since lifted metric balls $(\tilde{B}_r)_{r > 0}$ form an admissible sequence in $G$, then Lemma \ref{LemmaVarianceUpperBoundForAdmissibleSequences} tells us that 
\begin{align*}
\frac{\NVmu(r)}{m_X(B_r(x_o))^{2 - \delta}} = \frac{\norm{\bS_0\chi_{B_r(x_o)}}_{L^2(\mu)}^2}{m_X(B_r(x_o))^{2 - \delta}} \leq C(\mu, q)^2 \Big( m_o^{-1} \varepsilon^{-d} m_G(\tilde{B}_r)^{- \frac{1}{q} + \frac{\delta}{2}} + \varepsilon m_G(\tilde{B}_r)^{\frac{\delta}{2}}\Big)^2 
\end{align*}
for sufficiently large $r > 0$ and sufficiently small $\varepsilon > 0$. In addition to this, we take $r$ large enough so that the choice 
\begin{align*}
\varepsilon = m_G(\tilde{B}_r)^{-\delta'}
\end{align*}
remains valid for some fixed $\delta' > 0$ that is yet to be determined. This yields the upper bound
\begin{align*}
\frac{\NVmu(r)}{m_X(B_r(x_o))^{2 - \delta}} \leq C(\mu, q)^2 \Big( m_o^{-1}m_G(\tilde{B}_r)^{d \delta' - \frac{1}{q} + \frac{\delta}{2}} + m_G(\tilde{B}_r)^{\frac{\delta}{2} - \delta'}\Big)^2 \, . 
\end{align*}
Finally, taking $\delta', \delta > 0$ such that $\delta < 2\min\{\tfrac{1}{q} - d\delta', \delta'\}$, which can be done for $0 < \delta' < \tfrac{1}{dq}$, implies that 
\begin{align*}
\limsup_{r \rightarrow +\infty} \frac{\NVmu(r)}{m_X(B_r(x_o))^{2 - \delta}} = 0 \, . 
\end{align*}

\section{Determinantal point processes}
\label{Determinantal point processes}

The equivariant kernel determinantal point processes (DPPs) of interest are introduced in the first Subsection and we compute variances of linear statistics, followed by a formula for the Bartlett spectral measure in Subsection \ref{The Bartlett spectral measure of equivariant kernel DPPs}. The general framework for DPPs with kernels coming from unitary representations of covering groups is given in Subsection \ref{DPPs from unitary representations of covering groups} together with an application to infinite polyanalytic ensembles of pure type. In the last Subsection we compute the Bartlett spectral measure of the zero set for the standard hyperbolic Gaussian analytic function in the unit disk.

\subsection{Variances of linear statistics for equivariant kernel DPPs}
\label{Variances of linear statistics for equivariant kernel DPPs}

Let $L : X \times X \rightarrow \C$ be a continuous Hermitian positive-definite kernel, meaning $L(x_1, x_2) = \overline{L(x_2, x_1)}$ for all $x_1, x_2 \in X$ and 
\begin{align*}
\int_X \int_X f(x_1) \overline{f(x_2)} L(x_1, x_2) dm_X(x_1) dm_X(x_2) \geq 0 
\end{align*}
for every $f \in \Borelbndinfty(X)$. Let $c : G \times X \rightarrow \C$ be a measurable map such that $|c(g, x)| = 1$ for all $g \in G$, $x \in X$. We consider kernels $L$ as described that satisfy the equivariance property
\begin{align}
\label{EqEquivariantKernel}
L(g.x_1, g.x_2) = c(g, x_1) \overline{c(g, x_2)} L(x_1, x_2)
\end{align}
for all $g \in G$ and all $x_1, x_2 \in X$. In particular, $L(x, x) = L(x_o, x_o)$ for all $x \in X$. Such a kernel $L$ then defines a locally trace class operator 
\begin{align*}
Lf(x_1) = \int_X f(x_2) L(x_1, x_2) dm_X(x_2) \, , \quad f \in \Borelbndinfty(X)
\end{align*}
in the sense that 
\begin{align*}
\mathrm{Tr}(\chi_B L \chi_B) = \int_B L(x, x) dm_X(x) = L(x_o, x_o) m_X(B) < +\infty  
\end{align*}
for all bounded Borel sets $B \subset X$, see \cite[Lemma, p.65]{ReedSimonScatteringTheoryIII}. By \cite[Theorem 5.2.17]{baccelli:hal-02460214} there is a determinantal point process $\mu_L$ on $X$, uniquely determined by its $n$-point correlation densities 
\begin{align*}
\bE_{\mu_L}\Big( \sum_{p_1 \neq ... \neq p_n} f_1(p_1) ... f_n(p_n) \Big) = \int_{X^n} f_1(x_1) ... f_n(x_n) \det(L(x_i, x_j))_{i, j = 1}^n dm_X^{\otimes n}(x_1, ..., x_n)
\end{align*} 
for all $f_1, ..., f_n \in \Borelbndinfty(X)$. In particular, the intensity of $\mu_L$ is $L(x_o, x_o)$ since
\begin{align*}
\bE_{\mu_L}(\bS f) = \int_X f(x) L(x, x) dm_X(x) = L(x_o, x_o) \int_X f(x) dm_X(x) 
\end{align*}
for all $f_1, f_2 \in \Borelbndinfty(X)$, and the $2$-point correlation density is 
\begin{align*}
\bE_{\mu_L}\Big( \sum_{p_1 \neq p_2} &f_1(p_1) \overline{f_2(p_2)} \Big) = \int_{X^2} f_1(x_1) \overline{f_2(x_2)} (L(x_o, x_o)^2 - |L(x_1, x_2)|^2) dm_X^{\otimes 2}(x_1, x_2) \\
&= \bE_{\mu_L}(\bS f_1) \overline{\bE_{\mu_L}(\bS f_2)}  - \int_{X^2} f_1(x_1) \overline{f_2(x_2)} |L(x_1, x_2)|^2 dm_X^{\otimes 2}(x_1, x_2)  \, . 
\end{align*}
With this in mind, we can compute the variance of a linear statistic with respect to $\mu_L$ by
\begin{equation}
\begin{split}
\label{EqVarianceOfEquivariantDPP}
\Var_{\mu_L}(\bS f) &=  \bE_{\mu_L}\Big( \sum_{p} |f(p)|^2 \Big) + \bE_{\mu_L}\Big( \sum_{p_1 \neq p_2} f(p_1) \overline{f(p_2)} \Big) - |\bE_{\mu_L}(\bS f)|^2 \\
&= L(x_o, x_o) \int_X |f(x)|^2 dm_X(x) - \int_{X^2} f(x_1) \overline{f(x_2)} |L(x_1, x_2)|^2 dm_X^{\otimes 2}(x_1, x_2) \, . 
\end{split}
\end{equation}

\subsection{The Bartlett spectral measure of equivariant kernel DPPs}
\label{The Bartlett spectral measure of equivariant kernel DPPs}

Let $L : X \times X \rightarrow \C$ be a kernel as in the previous Subsection and $\mu_L$ the associated determinantal point process. 

\begin{lemma}
\label{LemmaVarianceDenistyOfDPPs}
The function $\kappa_L(g) := |L(x_o, g.x_o)|^2$ is continuous and bi-$K$-invariant on $G$. Moreover, if $f \in \Borelbndinfty(X)$ then the $\mu_L$-variance of the linear statistic $\bS f$ can be written as 
\begin{align*}
\Var_{\mu_L}(\bS f) = L(x_o, x_o)(\check{\varphi}_f * \overline{\varphi}_f)(e) - \int_G (\check{\varphi}_f * \overline{\varphi}_f)(g) \kappa_L(g) dm_G(g) \, . 
\end{align*}
\end{lemma}

\begin{proof}
Since the kernel $L$ is assumed to be continuous it is clear that $\kappa_L$ is continuous. To show that $\kappa_L$ is bi-$K$-invariant, we first note that for every $g_1, g_2 \in G$,
\begin{align*}
L(g_1.x_o, g_2.x_o) = L(g_1.x_o, g_1(g_1^{-1} g_2).x_o) = c(g_1, x_o) \overline{c(g_1, g_1^{-1}g_2.x_o)} L(x_o, g_1^{-1}g_2.x_o) \, . 
\end{align*}
Since $|c(g, x)| = 1$ for every $g \in G$ and every $x \in X$ by assumption, we find that $|L(g_1.x_o, g_2.x_o)|^2 = \kappa_L(g_1^{-1}g_2)$. Thus for every $k_1, k_2 \in K$ and every $g \in G$,
\begin{align*}
\kappa_L(k_1 g k_2) = |L(k_1^{-1}.x_o, gk_2.x_o)|^2 = |L(x_o, g.x_o)|^2 = \kappa_L(g) \, . 
\end{align*}
For the second part we use the definition of the $G$-invariant measure $m_X$ on $X$,
\begin{align*}
m_X(f) = \int_G \varphi_f(g) dm_G(g)
\end{align*}
and Equation \ref{EqVarianceOfEquivariantDPP} together with $|L(g_1.x_o, g_2.x_o)|^2 = \kappa_L(g_1^{-1}g_2)$ to write
\begin{align*}
\Var_{\mu_L}(\bS f) &= L(x_o, x_o) \int_G |\varphi_f(g)|^2 dm_G(g) - \int_{G^2} \varphi_f(g_1) \overline{\varphi_f(g_2)} \kappa_L(g_1^{-1}g_2) dm_G^{\otimes 2}(g_1, g_2) \\
&= L(x_o, x_o) \int_G |\varphi_f(g)|^2 dm_G(g) - \int_{G^2} \varphi_f(g_1) \overline{\varphi_f(g_1g_2)} \kappa_L(g_2) dm_G^{\otimes 2}(g_1, g_2) \, . 
\end{align*}
Finally, observing that
\begin{align*}
\int_G \varphi_f(g_1) \overline{\varphi_f(g_1g_2)} dm_G(g_1) = \int_G \check{\varphi}_f(g_1) \overline{\varphi_f(g_1^{-1}g_2)} dm_G(g_1) = (\check{\varphi}_f * \overline{\varphi}_f)(g_2)
\end{align*}
for every $g_2 \in G$ yields the desired formula. 
\end{proof}

If $\kappa_L(g) = |L(x_o, g.x_o)|^2$ is in addition square-integrable with respect to the Haar measure $m_G$, then by the Plancherel formula in Remark \ref{RemarkPlancherelMeasure} we have the following explicit formula for the diffraction measure $\sigma_{\mu_L}$ of the DPP $\mu_L$.

\begin{corollary}[Diffraction measure of an equivariant kernel DPP]
\label{CorollaryDiffractionMeasureofEquivariantKernalDPP}
If the kernel $L : X \times X \rightarrow \C$ satisfies 
\begin{align*}
\int_X |L(x_o, x)|^4 dm_X(x) < + \infty
\end{align*}
then   
\begin{align*}
\Var_{\mu_L}(\bS f) = \int_{\cS^+} |\hat{\varphi}_f(\omega)|^2 (L(x_o, x_o) - \hat{\kappa}_L(\omega)) d\sigma_{\cP}(\omega)
\end{align*}
for all $f \in \Borelbndinfty(X)^K$, where $\sigma_{\cP}$ is the spherical Plancherel measure for the pair $(G, K)$ from Remark \ref{RemarkPlancherelMeasure}.
\end{corollary}

\subsection{DPPs from unitary representations of covering groups}
\label{DPPs from unitary representations of covering groups}

Let $\tilde{G}$ be a locally compact second countable group and suppose that there is a continuous surjective group homomorphism $a : \tilde{G} \rightarrow G$. Then the continuous transitive action of $G$ on $X$ lifts to a continuous transitive action of $\tilde{G}$ on $X$ by defining $\tilde{g}.x = a(\tilde{g}).x$ for all $\tilde{g} \in \tilde{G}$ and all $x \in X$, and we denote by $\tilde{K} = \Stab_{\tilde{G}}(x_o)$ the corresponding stabilizer, which possibly is a non-compact subgroup. We will also let $b : G \rightarrow \tilde{G}$ denote a measurable section of $a$ such that $b(Q) \subset \tilde{G}$ is pre-compact for every pre-compact $Q \subset G$ as in \cite[Lemma 1.1]{MackeyInducedRepsI}. Throughout this Subsection we will make the following assumption:

\noindent\textbf{Assumption 3}: The sections $s : X \rightarrow G$ and $b : G \rightarrow \tilde{G}$ are continuous. 

With this assumption we get a continuous section $\tilde{s} := b \circ s : X \rightarrow \tilde{G}$ such that $\tilde{s}(x_o) \in \tilde{K}$, which yields a natural continuous $\tilde{K}$-valued cocycle by the following Lemma.

\begin{lemma}
\label{LemmaKtildeValuedCocycleFromSection}
The map $\tilde{\gamma}_s : \tilde{G} \times X \rightarrow \tilde{K}$ given by $\tilde{\gamma}_s(g, x) = \tilde{s}(\tilde{g}.x)^{-1}\tilde{g}\tilde{s}(x)$ is a well-defined continuous map satisfying the cocycle identity
\begin{align*}
\tilde{\gamma}_s(\tilde{g}_1\tilde{g}_2, x) = \tilde{\gamma}_s(\tilde{g}_1, \tilde{g}_2.x)\tilde{\gamma}_s(\tilde{g}_2, x)
\end{align*}
for all $\tilde{g}_1, \tilde{g}_2 \in \tilde{G}$ and all $x \in X$.
\end{lemma}

\begin{proof}
Since $\tilde{s}(\tilde{g}.x)\tilde{\gamma}_s(g, x).x_o = \tilde{g}\tilde{s}(x).x_o = \tilde{g}.x$ then $\tilde{\gamma}_s(g, x) \in \tilde{K}$ and from this it is clear that $\tilde{\gamma}_s$ is also continuous by Assumption 2 above. To see that the cocycle identity holds, we simply compute
\begin{align*}
\tilde{\gamma}_s(\tilde{g}_1, \tilde{g}_2.x)\tilde{\gamma}_s(\tilde{g}_2, x) &= \tilde{s}(\tilde{g}_1\tilde{g}_2.x)^{-1}\tilde{g}_1\tilde{s}(\tilde{g}_2.x) \tilde{s}(\tilde{g}_2.x)^{-1}\tilde{g}_2\tilde{s}(x) \\
&= \tilde{s}(\tilde{g}_1\tilde{g}_2.x)^{-1}\tilde{g}_1\tilde{g}_2\tilde{s}(x) = \tilde{\gamma}_s(\tilde{g}_1\tilde{g}_2, x) \, . 
\end{align*}
\end{proof}

The DPPs that we are interested in for this Subsection requires the existence of a unitary $\tilde{G}$-representation that is square-integrable modulo the kernel of the surjective homomorphism $a : \tilde{G} \rightarrow G$.

\noindent\textbf{Assumption 4}: There is a unitary $\tilde{G}$-representation $(\pi, \cH)$ such that
\begin{align*}
\int_G |\langle \pi(b(g))v, v \rangle_{\cH}|^2 dm_G(g) < +\infty 
\end{align*}
for every $v \in \cH$. Moreover, we fix a character $\xi : \tilde{K} \rightarrow \bT^1$ and assume that there is a unit vector $v \in \cH \backslash \{0\}$ such that $\pi(\tilde{k})v = \xi(\tilde{k})v$ for every $\tilde{k} \in \tilde{K}$.

\noindent With Assumptions 3 and 4 we are able to define the kernels of interest for this Subsection.

\begin{definition}
\label{DefinitionGtildeKernels}
Let $(\pi, \cH)$, $\xi$ and $v$ be as in Assumption 4. The kernel $L : X \times X \rightarrow \C$ associated to $(\pi, \xi, v)$ is the function
\begin{align*}
L(x_1, x_2) = \langle \pi(\tilde{s}(x_1))v, \pi(\tilde{s}(x_2))v \rangle_{\cH} \, . 
\end{align*}
\end{definition}

One verifies that a kernel $L$ as in this definition is continuous, Hermitian and positive-definite. To construct a DPP $\mu_L$ from such a kernel $L$, we show that it is in fact an equivariant kernel as introduced in Subsection \ref{Variances of linear statistics for equivariant kernel DPPs}. 

\begin{lemma}
Let $L$ be a kernel as in Definition \ref{DefinitionGtildeKernels}. Then, for every $g \in G$ and every $x_1, x_2 \in X$,
\begin{align*}
L(g.x_1, g.x_2) = \overline{\xi(\tilde{\gamma}_s(b(g), x_1))} \xi(\tilde{\gamma}_s(b(g), x_2)) L(x_1, x_2) \, . 
\end{align*}
\end{lemma}

\begin{proof}
Using that $\pi(\tilde{\gamma}_s(b(g), x)^{-1})v = \overline{\xi(\tilde{\gamma}_s(b(g), x))}v$, we find that
\begin{align*}
\pi(\tilde{s}(g.x))v &= \pi(\tilde{s}(b(g).x) \tilde{\gamma}_s(b(g), x)) \tilde{\gamma}_s(b(g), x)^{-1})v \\
&= \pi(\tilde{s}(b(g).x) \tilde{\gamma}_s(b(g), x))\pi(\tilde{\gamma}_s(b(g), x)^{-1})v \\
&= \overline{\xi(\tilde{\gamma}_s(b(g), x))} \pi(b(g)\tilde{s}(x))v \\
&= \overline{\xi(\tilde{\gamma}_s(b(g), x))} \pi(b(g)) \pi(\tilde{s}(x))v 
\end{align*}
from which it follows that 
\begin{align*}
L(g.x_1, g.x_2) &= \langle \pi(\tilde{s}(g.x_1))v, \pi(\tilde{s}(g.x_2))v \rangle_{\cH} \\
&= \overline{\xi(\tilde{\gamma}_s(b(g), x_1))} \xi(\tilde{\gamma}_s(b(g), x_2)) \langle \pi(b(g)) \pi(\tilde{s}(x_1))v, \pi(b(g)) \pi(\tilde{s}(x_2))v \rangle_{\cH} \\
&= \overline{\xi(\tilde{\gamma}_s(b(g), x_1))} \xi(\tilde{\gamma}_s(b(g), x_2)) \langle \pi(\tilde{s}(x_1))v, \pi(\tilde{s}(x_2))v \rangle_{\cH} \\
&= \overline{\xi(\tilde{\gamma}_s(b(g), x_1))} \xi(\tilde{\gamma}_s(b(g), x_2)) L(x_1, x_2) \, . 
\end{align*}
\end{proof}
Setting $c(g, x) = \overline{\xi(\tilde{\gamma}_s(b(g), x))}$ we see that $L$ is an equivariant kernel as in Equation \ref{EqEquivariantKernel}, so there is a well-defined invariant determinantal point process $\mu_L$ on $X$ associated to $L$. 

\begin{remark}
We emphasize that although $\tilde{\gamma}_s$ is a cocycle, the map $c(g, x) = \overline{\xi(\tilde{\gamma}_s(b(g), x))}$ is in general not, but only equivalent to a cocycle modulo some $2$-cocycle. 
\end{remark}

Moreover, Lemma \ref{LemmaVarianceDenistyOfDPPs} holds for $\mu_L$ and Corollary \ref{CorollaryDiffractionMeasureofEquivariantKernalDPP} also holds since we assumed that the unitary $\tilde{G}$-representation $(\pi, \cH)$ was square-integrable modulo the kernel of the surjective homomorphism $a : \tilde{G} \rightarrow G$. To summarize, we can write the $\mu_L$-variance of a linear statistic $\bS f$, $f \in \Borelbndinfty(X)$, as
\begin{align*}
\Var_{\mu_L}(\bS f) = \int_{\cS^+} |\hat{\varphi}_f(\omega)|^2 (1 - \hat{\kappa}_L(\omega)) d\sigma_{\cP}(\omega) \, , 
\end{align*}
where
\begin{align*}
\hat{\kappa}_L(\omega) = \int_G |\langle \pi(b(g))v, v \rangle_{\cH}|^2 \omega(g^{-1}) dm_G(g) \, .  
\end{align*}

\subsection{Weyl-Heisenberg ensembles}
\label{Weyl-Heisenberg ensembles}
Consider $\C^d$ with the symplectic form 
$$\zeta(z_1, z_2) = \Im(\langle z_1, z_2 \rangle)$$
and let $G = \U(d) \ltimes \C^d$ be the corresponding group of unitary motions of $\C^d$ with stabilizer $K = \U(d) \ltimes \{0\}$ of the origin. The \emph{Heisenberg motion group} associated to $G$ can be realized as a central extension of $G$ by $\R$ using the symplectic form $\zeta$, 
\begin{align*}
\tilde{G} = G \oplus_{\zeta} \R = \U(d) \ltimes (\C^d \oplus_{\zeta} \R)
\end{align*}
with the group law 
$$(k_1, z_1, t_1)(k_2, z_2, t_2) = (k_1k_2, z_1 + k_1.z_2, t_1 + t_2 - \tfrac{1}{2}\Im(z_1, k_1.z_2)) \, . $$
The subgroup $\bH_d = \C^d \oplus_{\zeta} \R$ is the $(2d + 1)$-dimensional \emph{Heisenberg group}. Note that the $\tilde{G}$-stabilizer 
$$\tilde{K} = \U(d) \ltimes (\{0\} \oplus_{\zeta} \R) \cong \U(d) \times \R $$
is non-compact. Here, we fix the continuous section $s : \C^d \rightarrow G$, $s(z) = (1, z)$, the surjective homomorphism $a : \tilde{G} \rightarrow G$ given by $a(k, z, t) = (k, z)$ and the canonical continuous section $b : G \rightarrow \tilde{G}$ of $a$, $b(k, z) = (k, z, 0)$, yielding a section $\tilde{s} : \C^d \rightarrow \tilde{G}$ given by $\tilde{s}(z) = b(s(z)) = (1, z, 0)$. 

It is a fact that the pair $(\tilde{G}, K)$ is Gelfand, see for example \cite[p. 342]{KoranyiHeisenbergMotionGroupGelfandPair}, and the $K$-spherical unitary irreducible $\tilde{G}$-representations $(\pi_{\lambda, n}, \cH_{\lambda, n})$ that we will take interest in are, up to unitary equivalence, parametrized by $\lambda > 0$ and $n \in \Z_{\geq 0}$ with associated spherical functions $\omega_{\lambda, n} \in \cS^+$ explicitly given by 
\begin{align*}
\omega_{\lambda, n}(k, z, t) = \langle \pi_{\lambda, n}(k, z, t)v_{\lambda, n}, v_{\lambda, n} \rangle_{\cH_{\lambda, n}} = \e^{i\lambda t} \cL_n^{(d-1)}(\tfrac{1}{2} |\lambda| \norm{z}^2) \e^{-\frac{1}{4}|\lambda| \norm{z}^2} \, ,
\end{align*}
see for example \cite[Prop. 3.2.4, p. 119]{ThangaveluHeisenbergBook}. Here, $v_{\lambda, n} \in \cH_{\lambda, n}$ is a $K$-invariant unit vector and
\begin{align*}
\cL_n^{(\alpha)}(r) = \frac{r^{-\alpha}\e^r}{n!} \frac{d^n}{dr^n}(r^{n + \alpha}\e^{-r}) 
\end{align*}
denotes a generalized Laguerre polynomial. The kernels of interest are 
\begin{align*}
L_{\lambda, n}(z_1, z_2) = \langle \pi_{\lambda, n}(1, z_1, 0)v_{\lambda, n}, \pi_{\lambda, n}(1, z_2, 0)v_{\lambda, n} \rangle_{\cH_{\lambda, n}} \, , 
\end{align*}
which can be rewritten as 
\begin{align*}
L_{\lambda, n}(z_1, z_2) &= \omega_{\lambda, n}(1, z_1 - z_2, \tfrac{1}{2}\Im(\langle z_1, z_2\rangle)) \\
&= \e^{\tfrac{i\lambda}{2} \Im(\langle z_1, z_2 \rangle)} \cL_n^{(d-1)}(\tfrac{1}{2} |\lambda| \norm{z_1 - z_2}^2) \e^{-\frac{1}{4}|\lambda| \norm{z_1 - z_2}^2} .
\end{align*}
Note that if $k \in \U(d)$ and $w \in \C^d$ then $g = (k, w) \in G$ and 
\begin{align*}
L_{\lambda, n}(g.z_1, g.z_2) &= \e^{\tfrac{i\lambda }{2}\Im(\langle g.z_1, g.z_2 \rangle)} \cL_n^{(d-1)}(\tfrac{1}{2} |\lambda| \norm{z_1 - z_2}^2) \e^{-\frac{1}{4}|\lambda| \norm{z_1 - z_2}^2} \\
&= \e^{\tfrac{i\lambda }{2}(\Im(\langle g.z_1, g.z_2 \rangle - \langle z_1, z_2 \rangle))} L_{\lambda, n}(z_1, z_2)\, . 
\end{align*}
Since 
\begin{align*}
\langle g.z_1, g.z_2 \rangle - \langle z_1 , z_2 \rangle = (\langle k.z_1, w \rangle + \tfrac{1}{2}\norm{w}^2) + \overline{(\langle k.z_2, w \rangle + \tfrac{1}{2}\norm{w}^2)}
\end{align*}
then defining $c_{\lambda}(g, z) = \e^{\frac{i\lambda}{2}\Im(\langle k.z, w \rangle + \tfrac{1}{2}\norm{w}^2)} \in \bT^1$ allows us to write
\begin{align*}
L_{\lambda, n}(g.z_1, g.z_2) &= c_{\lambda}(g, z_1) \overline{c_{\lambda}(g, z_2)} L_{\lambda, n}(z_1, z_2) \, ,
\end{align*}
so that $L_{\lambda, n}$ defines an equivariant kernel as in Subsection \ref{Variances of linear statistics for equivariant kernel DPPs}. If say $\lambda > 0$ is positive, then the formula for $L_{\lambda, n}$ can be simplified further to
\begin{align*}
L_{\lambda, n}(z_1, z_2) &= \cL_n^{(d-1)}(\tfrac{1}{2} \lambda \norm{z_1 - z_2}^2) \e^{-\frac{1}{4}\lambda (\norm{z_1}^2 + \norm{z_2}^2 - 2 \langle z_1, z_2 \rangle)} \, . 
\end{align*}
We denote the associated invariant locally square-integrable determinantal point process in $\C^d$ by $\mu_{\lambda, n} := \mu_{L_{\lambda, n}}$.

\begin{remark}
When $d = 1$, $\lambda = 2\pi$ and $n = 0$, the point process $\mu_{2\pi, 0}$ is the \emph{infinite Ginibre ensemble} in $\C$ with kernel 
\begin{align*}
L_{2\pi, 0}(z_1, z_2) = \e^{\pi z_1 \overline{z}_2  - \frac{\pi}{2}(|z_1|^2 + |z_2|^2)} \, . 
\end{align*}
For general dimension $d$ and general $n \in \Z_{\geq 0}$ the corresponding determinantal point process $\mu_{2\pi, n}$ in $\C^d$ are \emph{infinite polyanalytic ensembles of pure type} as defined in \cite{AbreuPereiraRomeroTorquato}.
\end{remark}
The function $\kappa_{\lambda, n}(g) = |L_{\lambda, n}(0, g.0)|^2$ satisfies
\begin{align*}
\int_{G} \kappa_{\lambda, n}(g) dm_G(g) &= \int_{\C^d} |L_{\lambda, n}(0, z)|^2 dA^{\otimes d}(z) \\
&= \int_{\C^d} \cL_n^{(d-1)}(\tfrac{1}{2} |\lambda| \norm{z}^2)^2 \e^{-\frac{1}{2}|\lambda| \norm{z}^2} dA^{\otimes d}(z) < + \infty \, , 
\end{align*}
where $A$ denotes the Lebesgue measure on $\C$, so by Corollary \ref{CorollaryDiffractionMeasureofEquivariantKernalDPP} we can write the variance of linear statistics as 
\begin{align*}
\Var_{\mu_{\lambda, n}}(\bS f) = \int_{\C^d} |\hat{\varphi}_f(\zeta)|^2 (1 - \hat{\kappa}_{\lambda, n}(\zeta)) dA^{\otimes d}(\zeta)
\end{align*}
where 
\begin{align*}
\hat{\kappa}_{\lambda, n}(\zeta) = 2^{d - 1} (d-1)! \int_{\C^d} \cL_n^{(d-1)}(\tfrac{1}{2} |\lambda| \norm{z}^2)^2 \e^{-\frac{1}{2}|\lambda| \norm{z}^2} \frac{J_{d - 1}(2\pi \norm{\zeta}\norm{z})}{(\norm{\zeta}\norm{z})^{d - 1}} dA^{\otimes d}(z) \, . 
\end{align*}

\begin{remark}
When $d = 1$, $\hat{\kappa}_{\lambda, n}$ can be calculated explicitly:
\begin{align*}
\hat{\kappa}_{\lambda, n}(\zeta) &= \int_{\C} \cL_n(\tfrac{1}{2} |\lambda| |z|^2)^2 \e^{-\frac{1}{2}|\lambda| |z|^2} J_{0}(2\pi |\zeta| |z|) dA(z) \\
&= 2\pi \int_0^{\infty} \cL_n(\tfrac{1}{2} |\lambda| s^2)^2 \e^{-\frac{1}{2}|\lambda| s^2} J_{0}(2\pi |\zeta| s) s ds \, . 
\end{align*}
From \cite[Equation 7.422(2), p. 812]{GradshteynRyzhik} with $\nu = \sigma = 0, m = n, \alpha = |\lambda|/2, y = 2\pi|\zeta|$, we get
\begin{align*}
\int_0^{\infty} \cL_n(\tfrac{1}{2} |\lambda| s^2)^2 \e^{-\frac{1}{2}|\lambda| s^2} J_{0}(2\pi |\zeta| s) s ds = \frac{1}{|\lambda|} \cL_n(\tfrac{2\pi^2}{|\lambda|} |\zeta|^2)^2 \e^{-\frac{2\pi^2}{|\lambda|} |\zeta|^2} \, , 
\end{align*}
so that 
\begin{align*}
\hat{\kappa}_{\lambda, n}(\zeta) &= \frac{2\pi}{|\lambda|} \cL_n(\tfrac{2\pi^2}{|\lambda|} |\zeta|^2)^2 \e^{-\frac{2\pi^2}{|\lambda|} |\zeta|^2} \, . 
\end{align*}
The variance of a linear statistic is thus given by 
\begin{align*}
\Var_{\mu_{\lambda, n}}(\bS f) = \int_{\C} |\hat{\varphi}_f(\zeta)|^2 \Big(1 - \frac{2\pi}{|\lambda|} \cL_n(\tfrac{2\pi^2}{|\lambda|} |\zeta|^2)^2 \e^{-\frac{2\pi^2}{|\lambda|} |\zeta|^2}\Big) dA(\zeta) \, . 
\end{align*}
Note in particular that $\hat{\kappa}_{\lambda, n}(0) = 2\pi/|\lambda|$, so the infinite polyanalytic ensembles $\mu_{\lambda, n}$ are spectrally hyperuniform if and only if $\lambda = \pm 2\pi$.
\end{remark}

\begin{remark}
In general, the DPPs $\mu_{\lambda, n}$ are (spectrally) hyperuniform whenever $\hat{\kappa}_{\lambda, n}(1) = 1$, and this quantity coincides with the notion of \emph{formal dimension} of the representation $\pi_{\lambda, n}$ modulo the center $Z = Z(\tilde{G}) \cong \R$,
\begin{align*}
d_{\pi_{\lambda, n}} &= \int_{\tilde{G}/Z} |\langle \pi_{\lambda, n}(\tilde{g})v_{\lambda, n}, v_{\lambda, n} \rangle_{\cH}|^2 dm_{\tilde{G}/Z}(\tilde{g}Z) \\
&= \int_{\C^d} |\langle \pi_{\lambda, n}(1, z, 0)v_{\lambda, n}, v_{\lambda, n} \rangle_{\cH}|^2 dA(z) = \hat{\kappa}_{\lambda, n}(1) \, . 
\end{align*}
\end{remark}

\subsection{Bergman kernel determinantal point process in the unit disk}
\label{Bergman kernel determinantal point process in the unit disk}

In this Subsection we consider the the random zero set in the open unit disk $\bD \subset \C$ of the Gaussian analytic function
\begin{align*}
F_1(z) = \sum_{n = 0}^{\infty} a_n z^n \, , \quad |z| < 1 \, , 
\end{align*}
where $a_n$ standard complex Gaussian random variables. One shows that the random zero set of $F_1$ is invariant under the action of $\SU(1, 1)$ on $\bD$ by Möbius transformations. Peres and Vir\'ag in \cite{PeresVirag} showed that the random zero set of the random analytic function $F_1$ is equivalent to the invariant determinantal point process governed by the \emph{Bergman kernel}
\begin{align*}
L_o(z_1, z_2) = \frac{1}{\pi(1 - z_1\overline{z}_2)^2} \, , \quad z_1, z_2 \in \bD \, . 
\end{align*}
In this Subsection we compute the Bartlett spectral measure of this DPP.
 
\subsubsection{Preliminaries on the Poincaré disk}

Consider the open unit disk $\bD = \{ z \in \C : |z| < 1 \}$ endowed with the Poincar\'e metric 
$$d(z_1, z_2) = \arccosh\Big( 1 +  \frac{2 |z_1 - z_2|^2}{(1 - |z_1|^2)(1 - |z_2|^2)} \Big) \, . $$
The group $G = \SU(1, 1)$ acts isometrically and transitively on $\bD$ by Möbius transformations with maximal compact stabilizer $K = \U(1)$ realized as diagonal matrices, so that $\bD = G/K$. An invariant measure on $\bD$ is 
\begin{align*}
m_{\bD}(f) = \int_{\bD} f(z) \frac{dA(z)}{\pi(1 - |z|^2)^2} \, , \quad f \in \Borelbndinfty(\bD) \, , 
\end{align*}
where $A$ denotes the standard Lebesgue measure on the complex plane, satisfying $A(\bD) = \pi$. The $\U(1)$-spherical functions for $\SU(1, 1)$ can be expressed in terms of Legendre functions,
\begin{align*}
\omega_{\lambda}(g) = P_{-(\frac{1}{2} + i\lambda)}(\cosh(d(g.0, 0))) = \frac{1}{\pi} \int_0^{\pi} (\cosh(d(g.0, 0)) - \sinh(d(g.0, 0))\cos(t))^{-(\frac{1}{2} + i\lambda)} dt   
\end{align*}
and the space $\cS^+$ positive-definite spherical functions is homeomorphic to the space of parameters
\begin{align*}
\lambda \in \Omega_{\bD} := i(0, \tfrac{1}{2}] \cup [0, +\infty) \subset \C \, ,  
\end{align*}
where the real $\lambda$ correspond to principal series representations and imaginary $\lambda$ to the complementary series representations. Given this parametrization we will write $\hat{\varphi}(\lambda) = \hat{\varphi}(\omega_{\lambda})$ for the spherical transform of a function $\varphi$ on $G$. Moreover, the spherical Plancherel measure for $(\SU(1, 1), \U(1))$ can be explicitly realized on $\Omega_{\bD}$ as 
\begin{align*}
d\sigma_{\cP}(\omega_{\lambda}) = \chi_{(0, +\infty)}(\lambda)  \pi \lambda \tanh(\tfrac{\pi \lambda}{2}) d\lambda \, . 
\end{align*}

\subsubsection{Bergman kernels}

Consider the \emph{Bergman kernel} on $\bD$, given by
\begin{align*}
L_o(z_1, z_2) = \frac{1}{\pi(1 - z_1\overline{z}_2)^2} \, , 
\end{align*}
which is the reproducing kernel of the Bergman space $\cA^2(\bD, A)$ of square-$A$-integrable holomorphic functions on $\bD$. If we let 
\begin{align*}
g = \begin{pmatrix} u & v \\ \overline{v} & \overline{u} \end{pmatrix} \, , \quad |u|^2 - |v|^2 = 1
\end{align*}
be a Möbius transformation in $\SU(1, 1)$, then 
\begin{align*}
L_o(g.z_1, g.z_2) &= \frac{1}{\pi} \Big( 1 - \frac{uz_1 + v}{\overline{v}z_1 + \overline{u}}\frac{\overline{u}\overline{z}_2 + \overline{v}}{v \overline{z}_2 + u} \Big)^{-2} \\
&= \frac{(\overline{v}z_1 + \overline{u})^2 (v \overline{z}_2 + u)^2 }{\pi ((\overline{v}z_1 + \overline{u}) (v \overline{z}_2 + u) - (uz_1 + v)(\overline{u}\overline{z}_2 + \overline{v}))^2} \\
&= \frac{(\overline{v}z_1 + \overline{u})^2 (v \overline{z}_2 + u)^2 }{\pi (1 - z_1 \overline{z}_2)^2} = c_o(g, z_1) \overline{c_o(g, z_2)} L_o(z_1, z_2) \, ,
\end{align*}
where $c_o(g, z) = (\overline{v}z +\overline{u})^2$ is the equivariance factor, in fact a cocycle. Typically, we have $|c_o(g, z)| \neq 1$ so we modify the Bergman kernel to the following equivariant kernel,
\begin{align*}
L(z_1, z_2) = \frac{L_o(z_1, z_2)}{L_o(z_1, z_1)^{1/2} L_o(z_2, z_2)^{1/2}} = \frac{(1 - |z_1|^2)(1 - |z_2|^2)}{(1 - z_1 \overline{z}_2)^2} \, .  
\end{align*}
By \cite[Remark 5.1.2+5.1.7]{baccelli:hal-02460214} with $f(z) = L_o(z, z) = (1 - |z|^2)^2$, the determinantal point process $\mu_L$ is equivalent to the determinantal point process $\mu_{L_o}$. Defining 
\begin{align*}
c(g, z) = \frac{c_o(g, z)}{|c_o(g, z)|} = \frac{\overline{v}z + \overline{u}}{v \overline{z} + u} 
\end{align*}
then $|c(g, z)| = 1$ for all $g \in \SU(1, 1)$ and all $z \in \bD$, and it is immediate from the definition of $L$ that 
\begin{align*}
L(g.z_1, g.z_2) = c(g, z_1) \overline{c(g, z_2)} L(z_1, z_2) \, .  
\end{align*}
Moreover,
\begin{align*}
\kappa_L(g) = |L(0, g.0)|^2 = (1 - |g.0|^2)^2 
\end{align*}
for all $g \in \SU(1, 1)$ and 
\begin{align*}
\int_{\SU(1, 1)} \kappa_L(g) dm_{\SU(1,1)}(g) = \int_{\bD} |L(0, z)|^2 dm_{\bD}(z) = \frac{1}{\pi} \int_{\bD} dA(z) = 1 < +\infty \, , 
\end{align*}
so by Corollary \ref{CorollaryDiffractionMeasureofEquivariantKernalDPP} the variance of a linear statistic with respect to the associated invariant determinantal point process $\mu_L$ is 
\begin{align*}
\Var_{\mu_L}(\bS f) =  \int_{0}^{\infty} |\hat{\varphi}_f(\lambda)|^2 (1 - \hat{\kappa}_L(\lambda)) \pi \lambda \tanh(\tfrac{\pi \lambda}{2}) d\lambda 
\end{align*}
for all radial $f \in \Borelbndinfty(\bD)$. For our modified Bergman kernel $L$, the spherical transform $\hat{\kappa}_L$ of $\kappa_L(g) = (1 - |g.0|^2)^2$ has been calculated by Unterberger-Upmeier in \cite[Prop. 3.39, p.591]{UnterbergerUpmeier} as
\begin{align*}
\hat{\kappa}_L(\lambda) = \frac{1}{\pi} \int_{\bD} P_{-(\tfrac{1}{2} + i\lambda)}(d(z, 0)) dA(z) =  | \Gamma(\tfrac{3}{2} + i\lambda) |^2 \, .
\end{align*}
Finally, the $\mu_L$-variance and Bartlett spectral measure of $\mu_L$ are given by
\begin{align}
\label{EqBergmanDPPVariance}
\Var_{\mu_L}(\bS f) = \int_{0}^{\infty} |\hat{\varphi}_f(\lambda)|^2 \Big(1 - | \Gamma(\tfrac{3}{2} + i\lambda) |^2\Big) \pi\lambda \tanh\Big(\frac{\pi \lambda}{2}\Big) d\lambda \, . 
\end{align}
As mentioned in the Introduction, the point process $\mu_L$ is not spectrally hyperuniform, see Definition \ref{DefHyperbolicSpectralHyperuniformity}, since
\begin{align*}
\limsup_{\varepsilon \rightarrow 0^+} \frac{\sigma_{\mu}((0, \varepsilon])}{\sigma_{\cP}((0, \varepsilon])} = \limsup_{\varepsilon \rightarrow 0^+} \frac{\int_0^{\varepsilon}(1 - | \Gamma(\tfrac{3}{2} + i\lambda) |^2) \pi\lambda \tanh(\tfrac{\pi \lambda}{2}) d\lambda}{\int_0^{\varepsilon} \pi\lambda \tanh(\tfrac{\pi \lambda}{2}) d\lambda} = 1 - \Gamma(\tfrac{3}{2})^2 > 0 \, . 
\end{align*}



\appendix

\section{Random distributions and heat kernel hyperuniformity}
\label{Random distributions and heat kernelhyperuniformity}

Test function spaces, the corresponding distributions and the Bochner-Minlos Theorem are presented in the first two Subsections. We prove that every positive tempered Radon measure on the spherical dual $\cS^+$ of Euclidean and real hyperbolic spaces can be realized as the Bartlett spectral measure of a random tempered distribution in Subsections \ref{Gaussian random distributions on commutative spaces} and \ref{Heat kernel hyperuniformity on Euclidean and real hyperbolic spaces}. In the latter Subsection we also prove the equivalence of spectral hyperuniformity and heat kernel hyperuniformity for Euclidean and real hyperbolic spaces.

\subsection{Test function spaces}
\label{Test function spaces}

Let $(G, K)$ be a lcsc Gelfand pair with $K < G$ compact and let $X = G/K$ be the associated proper commutative metric space. Throughout this Appendix we will let $\sA_o(G) \subset \sL^1(G, \R)$ be a real $\lambda_G$-invariant nuclear Frech\'et $*$-subalgebra such that the spherical transforms of bi-$K$-averaged functions $\sA(G, K)$ separate points in $\cS^+$ as in Assumption 1 from Section \ref{Spectral measures}. Moreover, we will assume that the section $s : X \rightarrow G$ is such that $\sA_o(X) = P_K(\sA_o(G)) \subset \sL^1(X, \R)$ has the structure of a real $\lambda_X$-invariant nuclear Frech\'et space. We will denote by $\sA(G), \sA(X)$ the complexifications of $\sA_o(G), \sA_o(X)$ respectively and $\sA_o'(G), \sA_o'(X)$ will be the respective real dual spaces endowed with the \emph{strong topology}.

\begin{example}
We are particularly interested in the following cases:
\begin{enumerate}
    \item When $G = X = \R^d$ one can consider the nuclear Frech\'et algebra $\sA_o(\R^d) = \sS(\R^d, \R)$ of real-valued Schwartz functions, whose Frech\'et topology is generated by the seminorms
    \begin{align*}
    \norm{f}_{n, m} = \sup_{x \in \R^d} (1 + \norm{x})^{n} |\partial^{m} f(x)| 
    \end{align*}
    for all $n \in \Z_{\geq 0}$ and all multi-indices $m \in \Z_{\geq 0}^d$. The strong dual $\sS'(\R^d, \R)$ is the space of real-valued tempered distributions on $\R^d$, in other words distributions of polynomial growth. More generally, there is for every $\delta \in [0, +\infty)$ a \emph{Gelfand-Shilov space} $\sA_o(X) = \sS_{\delta}(\R^d, \R)$ of smooth functions generated by the seminorms
    \begin{align*}
    \norm{f}_{\delta, n, m} = \sup_{x \in \R^d} (1 + \norm{x})^{n} \e^{\delta\norm{x}} |\partial^{m} f(x)| \, , 
    \end{align*}
    and the corresponding strong dual space $\sS_{\delta}'(\R^d, \R)$ consists of distributions on $\R^d$ whose asymptotic growth are bounded by $(1 + \norm{x})^n \e^{\delta\norm{x}}$ for some $n \geq 0$. For nuclearity of the Schwartz- and Gelfand-Shilov spaces, see \cite[Chapter II, Section 2, p.86]{GelfandShilovGeneralizedFunctionsVolume2} for general definitions and \cite[Chapter IV, Section 3.2, p.178]{GelfandShilovGeneralizedFunctionsVolume3} for a proof with weights $M_p(x) = (1 + \norm{x}^2)^p \e^{\delta\norm{x}}$.
    \item Consider real hyperbolic space $\H^d = G_d/K_d$ with $G_d = \SO^{\circ}(1, d)$ (or more generally a semisimple connected Lie group with finite center) and $K_d = \SO(d)$ (in general a maximal compact subgroup). On $G_d$, one can consider the real \emph{Harish-Chandra $L^p$-spaces} $\sA_o(G_d) = \sC^p(G_d, \R)$ for $p \in (0, 2]$ generated by seminorms
    \begin{align*}
    \norm{f}_{p, n, D_l, D_r} = \sup_{g \in G_d} (1 + \ell(g))^{n} \Xi(g)^{-2/p} |(D_lfD_r)(g)| \, , 
    \end{align*}
    where $D_l, D_r$ are left- resp. right-invariant differential operators on $G_d$ and $\ell(g)$ is the hyperbolic distance between a fixed basepoint $x_o$ and $g.x_o$. Here, $\Xi$ is a spherical function for $(G_d, K_d)$ known as the \emph{Harish-Chandra $\Xi$-function}, and for real $d$-dimensional hyperbolic space it is given by
    \begin{align}
    \label{EqHarishChandraXiFunction}
    \Xi(g) = \frac{2^{\frac{d-1}{2}}\Gamma(\frac{d}{2})}{\sqrt{\pi} \, \Gamma(\frac{d - 1}{2})} \sinh(\ell(g))^{2 - d}\int_0^{\ell(g)} (\cosh(\ell(g)) - \cosh(s))^{\frac{d-3}{2}} ds \, .
    \end{align}
    Regarding nuclearity, it is known that the subspace $\sC^p(G_d, K_d, \R)$ of bi-$K_d$-invariant functions is isomorphic to a Gelfand-Shilov space, see for example \cite[Theorem 2]{Anker1991TheSF}, so it is nuclear and this subspace will suffice for our purposes. We mention briefly that nuclearity of the subspace $\sC^p(G_d, \R)^{K_d}$ of left-$K_d$-invariant functions can be deduced from a Paley-Wiener-type result for the Helgason-Fourier transform due to Helgason for $p = 2$ and Eguchi in \cite[Theorem 4.1.1, p.193]{EguchiHCLpSpaces} for general $0 < p \leq 2$.
    
\end{enumerate}
\end{example}

\subsection{The Bochner-Minlos Theorem}
\label{The Bochner-Minlos Theorem}

Let $\sA_o(X) \subset \sL^1(X, \R)$ be a real $G$-invariant nuclear Frech\'et space as before. A function $\Lambda : \sA_o(X) \rightarrow \C$ is \emph{positive-definite} if for every finite collection of complex numbers $z_1, \dots, z_n \in \C$ and functions $f_1, \dots f_n \in \sA_o(X)$,
\begin{align*}
\sum_{i, j = 1}^n z_i \overline{z}_j \Lambda(f_i - f_j) \geq 0 \, .
\end{align*}
Every positive-definite functions on general real nuclear Frech\'et spaces can by the following Theorem be realized as the characteristic function of a unique probability measure on the strong dual space.

\begin{theorem}[Bochner-Minlos]
\label{TheoremBochnerMinlos}
Let $\sA$ be a real nuclear Frech\'et vector space and let $\Lambda: \sA \rightarrow \C$ be a positive-definite continuous function such that $\Lambda(0) = 1$. Then there is a unique probability measure $\mu_{\Lambda}$ on the strong dual space $\sA'$ such that
\begin{align*}
\Lambda(f) = \int_{\sA'} \e^{i\xi(f)} d\mu_{\Lambda}(\xi) \, , \quad \forall \, f \in \sA \, . 
\end{align*}
\end{theorem}
If $\Lambda : \sA_o(X) \rightarrow \C$ is in addition $G$-invariant, then by uniqueness of the measure $\mu_{\Lambda}$ from the Bochner-Minlos Theorem we must have $g_*\mu_{\Lambda} = \mu_{\Lambda}$ for all $g \in G$, so $\mu_{\Lambda}$ is a $G$-invariant probability measure on $\sA_o'(X)$. 

Let $\mu$ be a fixed $G$-invariant probability measure on $\sA_o'(X)$. We refer to $\mu$ as a \emph{random $\sA_o'$-distribution}. Moreover, and we define for every function $f \in \sA_o(X)$ a generalized linear statistic $\bS f : \sA_o'(X) \rightarrow \C$ by setting $\bS f(\xi) = \xi(f)$. If all linear statistics $\bS f$ have finite second moment with respect to $\mu$ and $t \mapsto \Lambda(t \cdot f)$ is $C^2$ in a neighbourhood of $t = 0$ for every $f$, then the moments can be computed using $\Lambda$ as
\begin{align}
\label{EquationFirstAndSecondMomentOfRandomDistribution}
\int_{\sA_o'(X)} \xi(f) d\mu(\xi) = -i \frac{d}{dt} \Lambda(t \cdot f) \Big|_{t = 0}   \mbox{ and }
\int_{\sA_o'(X)} \xi(f)^2 d\mu(\xi) = - \frac{d^2}{dt^2} \Lambda(t \cdot f) \Big|_{t = 0} \, , 
\end{align}
where $\Lambda$ is the positive-definite characteristic function of $\mu$. We call a random $\sA'_o$-distribution \emph{$\sA_o$-square-integrable} if the complex-valued generalized linear statistics $\bS f$, $f \in \sA(X)$, have finite second moment. In that case, the $G$-equivariant map $\alpha_0 : \sA(G) \rightarrow L^2(\sA_o(X), \mu)$ given by
\begin{align*}
\alpha_0(\varphi)(\xi) = \bS(P_K\varphi)(\xi) - \Emu(\bS(P_K\varphi)) = \xi(P_K\varphi) + i \frac{d}{dt} \Lambda(P_K(t\cdot \varphi))\Big|_{t = 0} 
\end{align*}
satisfies the conditions of Theorem \ref{TheoremSpectralMeasureForAlgebraGMaps}, so there is a unique spectral measure $\sigma_{\mu} = \sigma_{\alpha_0} \in \Radonplus(\cS^+)$ such that
\begin{align*}
\Varmu(\bS f) = \norm{\alpha_0(\varphi_f)}_{L^2(\mu)}^2 = \int_{\cS^+} |\hat{\varphi}_f(\omega)|^2 d\sigma_{\mu}(\omega)
\end{align*}
for all $K$-invariant complex-valued $f \in \sA(X)$. 

\subsection{Gaussian random distributions on commutative spaces}
\label{Gaussian random distributions on commutative spaces}

In this Subsection we, roughly speaking, aim to prove a converse of Theorem \ref{Theorem1} for Gaussian random distributions. For the statement, we let $\sA_o(G), \sA_o(X)$ be function spaces as described in Subsection \ref{Test function spaces}. For the statement we recall the bi-$K$-averaging operation
\begin{align*}
\varphi^{\natural}(g) = \int_{K \times K} \varphi(k_1 g k_2) dm_K^{\otimes 2}(k_1, k_2) \, , \quad \varphi \in \sA_o(G) \, . 
\end{align*}

\begin{theorem}
Let $\sigma \in \Radonplus(\cS^+)$ and assume that the linear functional 
\begin{align*}
b_{\sigma}(\varphi) = \int_{\cS^+} \hat{\varphi}^{\natural}(\omega) d\sigma(\omega) \, , \quad \varphi \in \sA_o(G)
\end{align*}
is continuous. Then there is a unique $G$-invariant Gaussian probability measure $\mu_{\sigma}$ on $\sA_o'(X)$ such that 
\begin{align*}
\Var_{\mu_{\sigma}}(\bS f) = \int_{\cS^+} |\hat{\varphi}_f(\omega)|^2 d\sigma(\omega) 
\end{align*}
for all $K$-invariant $f \in \sA_o(X)$.
\end{theorem}

\begin{remark}
This variance formula readily extends to complex-valued $K$-invariant functions $f \in \sA(X)$.
\end{remark}

To prove this Theorem we first construct a positive-definite Gaussian function $\Lambda_{\sigma} : \sA_o(G) \rightarrow \C$ such that $\Lambda_{\sigma}(0) = 1$ from $b_{\sigma}$ and invoke the Bochner-Minlos Theorem to get an invariant probability measure on $\sA_o'(G)$. After that we can push said invariant probability measure down to an invariant probability measure on $\sA_o'(X)$ and verify the formula for the variance.

\begin{proof}
We first note that since $\sigma$ is a positive measure on the positive-definite spherical functions, $b_{\sigma}$ defines a positive-definite continuous $\sA_o(G)$-distribution on $G$. To see this we apply Fubini, 
\begin{align*}
b_{\sigma}(\check{\varphi} * \varphi) &= \int_{\cS^+} \hat{(\check{\varphi} * \varphi)^{\natural}}(\omega) d\sigma(\omega) \\
&= \int_{K \times K} \int_{\cS^+} \Big( \int_G (\check{\varphi} * \varphi)(k_1 g k_2) \omega(g^{-1}) dm_G(g) \Big) d\sigma(\omega) dm_K^{\otimes 2}(k_1, k_2) \\
&= \int_{\cS^+} \Big( \int_G (\check{\varphi} * \varphi)(g) \omega(g^{-1}) dm_G(g) \Big) d\sigma(\omega)\geq 0 \, , \quad \forall \, \varphi \in \sA_o(G) \, . 
\end{align*}
This allows us to define a positive-definite continuous bilinear form $B_{\sigma}$ on $\sA_o(G)$ by $B_{\sigma}(\varphi_1, \varphi_2) = b_{\sigma}(\check{\varphi}_2 * \varphi_1)$. Note that
\begin{align*}
B_{\sigma}(\lambda_G(g)\varphi_1, \lambda_G(g)\varphi_2) = B_{\sigma}(\varphi_1, \varphi_2)
\end{align*}
for all $g \in G$ since 
\begin{align*}
((\lambda_G(g)\varphi_2)^{\vee} * (\lambda_G(g)\varphi_1))(h) &= \int_G \varphi_2(g^{-1}\ell^{-1}) \varphi_2(g^{-1}\ell^{-1}h) dm_G(\ell) \\
&= \int_G \varphi_2(\ell^{-1}) \varphi_2(\ell^{-1}h) dm_G(\ell) = (\check{\varphi}_2 * \varphi_1)
\end{align*}
for all $\varphi_1, \varphi_2 \in \sA_o(G)$. With such a bilinear form we define a $\lambda_G$-invariant Gaussian function $\Lambda_{\sigma} : \sA_o(G) \rightarrow \C$ by
\begin{align*}
\Lambda_{\sigma}(\varphi) = \e^{-\frac{1}{2}B_{\sigma}(\varphi, \varphi)} \, , \quad \varphi \in \sA_o(G) \, . 
\end{align*}
One readily checks that $\Lambda_{\sigma}$ is positive-definite, $\Lambda_{\sigma}(0) = 1$, so by the Bochner-Minlos Theorem there is a unique $G$-invariant probability measure $\tilde{\mu}_{\sigma}$ on $\sA_o'(G)$ such that 
\begin{align*}
\e^{-\frac{1}{2}B_{\sigma}(\varphi, \varphi)} = \int_{\sA_o'(G)} \e^{i \zeta(\varphi)} d\tilde{\mu}_{\sigma}(\zeta) \, , \quad \forall \, \varphi \in \sA_o(G) \, . 
\end{align*}
Restricting this equality to functions of the form $\varphi_f(g) = f(g.o)$ for $f \in \sA_o(X)$ provides us with a $G$-invariant probability measure $\mu_{\sigma}$ on $\sA_o'(X)$ satisfying 
\begin{align*}
\e^{-\frac{1}{2}B_{\sigma}(\varphi_f, \varphi_f)} = \int_{\sA_o'(X)} \e^{i \xi(f)} d\mu_{\sigma}(\xi) \, , \quad f \in \sA_o(X) \, .  
\end{align*}
Since this expression remains positive-definite on $\sA_o(X)$, the measure $\mu_{\sigma}$ is unique by the Bochner-Minlos Theorem. Note also that since $\mu_{\sigma}$ is Gaussian, it has well-defined moments of all orders. Finally, the map 
\begin{align*}
\R \ni t \longmapsto \e^{-\frac{t^2}{2}B_{\sigma}(\varphi_f, \varphi_f)} \in \R_{> 0}
\end{align*}
is $C^2$ for every $f \in \sA_o(X)$, so we can compute the $\mu_{\sigma}$-expectation and variance of linear statistics:
\begin{align*}
\bE_{\mu_{\sigma}}(\bS f) = \int_{\sA_o'(X)} \xi(f) d\mu_{\sigma}(\xi) = -i \frac{d}{dt} \e^{-\frac{t^2}{2}B_{\sigma}(\varphi_f, \varphi_f)} \Big|_{t = 0} = 0 \, , 
\end{align*}
and 
\begin{align*}
\Var_{\mu_{\sigma}}(\bS f) &= \bE_{\mu_{\sigma}}((\bS f)^2) = \int_{\sA_o'(X)} \xi(f)^2 d\mu(\xi) = - \frac{d^2}{dt^2} \e^{-\frac{t^2}{2}B_{\sigma}(\varphi_f, \varphi_f)}\Big|_{t = 0} \\
&= B_{\sigma}(\varphi_f, \varphi_f) = \int_{\cS^+} \hat{(\check{\varphi}_f * \varphi_f)}(\omega) d\sigma(\omega) \, . 
\end{align*}
In particular, if $f \in \sA_o(X)$ is $K$-invariant, then
\begin{align*}
\Var_{\mu_{\sigma}}(\bS f) = \int_{\cS^+} |\hat{\varphi}_f(\omega)|^2 d\sigma(\omega) \, . 
\end{align*}
\end{proof}

\subsection{Heat kernel hyperuniformity on Euclidean and real hyperbolic spaces}
\label{Heat kernel hyperuniformity on Euclidean and real hyperbolic spaces}

\subsubsection{Spectral hyperuniformity}

\textbf{The Euclidean setting}: The positive-definite spherical functions on $\R^d$ are precisely the characters $x \mapsto \e^{-i \langle x, y \rangle}$ with $y \in \R^d$, and the Bartlett spectral measure $\sigma_{\mu}$ of a translation invariant locally square-integrable distribution can be shown to be a positive tempered measure using the Paley-Wiener Theorem. Thus we restrict our attention to translation invariant random tempered distributions $\mu$ on $\R^d$. 
Spectral hyperuniformity of a random (tempered) distribution is defined as follows.

\begin{definition}[Euclidean spectral hyperuniformity]
A translation invariant locally square-integrable random (tempered) distribution is \emph{spectrally hyperuniform} if 
\begin{align*}
\lim_{\varepsilon \rightarrow 0^+} \frac{\sigma_{\mu}(B_{\varepsilon}(0))}{\Vol_{\R^d}(B_{\varepsilon}(0))} = 0 \, . 
\end{align*}
\end{definition}
This definition equivalent to that of geometric hyperuniformity when $\mu$ is a translation invariant locally square-integrable random measure on $\R^d$, see \cite[Prop. 3.3]{Bjorklund2022HyperuniformityAN}. If we let $\sigma_{\mu}^{\rad}$ denote the radial average of $\sigma_{\mu}$ on $(0, +\infty)$, defined by
\begin{align*}
\int_0^{\infty} F(\lambda) d\sigma_{\mu}^{\rad}(\lambda) = \int_{\R^d} F(\norm{y}) d\sigma_{\mu}(y) \, , \quad F \in \sS(\R)^{\even}
\end{align*}
then spectral hyperuniformity of $\mu$ can equivalently be formulated as 
\begin{align*}
\lim_{\varepsilon \rightarrow 0^+} \frac{\sigma^{\rad}_{\mu}((0, \varepsilon])}{\sigma_{\cP}((0, \varepsilon])} = 0 \, , 
\end{align*}
where $d\sigma_{\cP}(\lambda) = \Vol_{d-1}(S^{d-1})\lambda^{d-1} d\lambda$ is the Plancherel measure for the Euclidean motion group $\SO(d) \ltimes \R^d$ with $\sigma_{\cP}((0, \varepsilon]) \asymp \varepsilon^d$ as $\varepsilon \rightarrow 0^+$. 

\textbf{The real hyperbolic setting}: Consider real $d$-dimensional hyperbolic space $\H^d = G_d/K_d$, where again $(G_d, K_d) = (\SO^{\circ}(1, d), \SO(d))$. The positive-definite spherical functions $\cS^+$ for $(G_d, K_d)$ can be parametrized by the subset
\begin{align*}
\Omega_d = i(0, \tfrac{d-1}{2}] \cup [0, +\infty) \subset \C \, , 
\end{align*}
where the two intervals parametrize equivalence classes of the complementary- and principal series spherical representations respectively. The homeomorphism is $\Omega_d \rightarrow \cS^+$, $\lambda \mapsto \omega_{\lambda}$, where
\begin{align*}
\omega_{\lambda}(g) = \frac{2^{\frac{d-1}{2}}\Gamma(\frac{d}{2})}{\sqrt{\pi} \, \Gamma(\frac{d - 1}{2})} \sinh(\ell(g))^{2 - d}\int_0^{\ell(g)} (\cosh(\ell(g)) - \cosh(s))^{\frac{d-3}{2}} \cos(\lambda s) ds \, .
\end{align*}
Note that there Harish-Chandra $\Xi$-function from Equation \ref{EqHarishChandraXiFunction} is $\Xi = \omega_0$. Using a Paley-Wiener type Theorem due to Trombi-Varadarajan for Harish-Chandra $L^p$-spaces, one can similarly to the Euclidean case prove that Bartlett spectral measures of invariant locally square-integrable random distributions are ''tempered''. To state the Theorem, we consider for open strips $T_r = i(-r, r) \times \R \subset \C$ the \emph{$r$-Paley-Wiener space} $\PW(T_r)$ of rapidly decaying holomorphic functions on $T_r$ that extend continuously to the closure $\overline{T}_r \subset \C$. A proof of the following can be found in \cite[Theorem 1]{Anker1991TheSF}.

\begin{theorem}[Trombi-Varadarajan, real hyperbolic space]
\label{TheoremTrombiVaradarajanAnkerPaleyWiener}
The spherical transform extends to a $*$-algebra isomorphism $\sC^p(G_d, K_d) \rightarrow \PW(T_{2/p - 1})^{\even}$.
\end{theorem}

Since $\Omega_d \backslash \{i(d-1)/2\} \subset T_{(d-1)/2}$, this Theorem implies that the Bartlett spectral measure $\sigma_{\mu}$ is tempered on $\Omega_d$ and that we can restrict our attention to random distributions $\mu$ that are supported on the strong dual of the Harish-Chandra space $\sC^{4/(d+1)}(G_d)$. 
The following definition of spectral hyperuniformity was introduced for random measures on $\bH^d$ in \cite[Definition 1.5]{Björklund2024HyperuniformityOfRandomMeasuresOnEuclideanAndHyperbolicSpaces}.

\begin{definition}[Real hyperbolic spectral hyperuniformity]
\label{DefHyperbolicSpectralHyperuniformity}
A $G_d$-invariant locally square-integrable random distribution $\mu$ on $\H^d$ is \emph{spectrally hyperuniform} if
\begin{align*}
\sigma_{\mu}(i[0, \tfrac{d-1}{2}]) = 0 \quad \mbox{ and } \quad \lim_{\varepsilon \rightarrow 0^+} \frac{\sigma_{\mu}((0, \varepsilon])}{\sigma_{\cP}((0, \varepsilon])} = 0 \, . 
\end{align*}
\end{definition}
Here, the spherical Plancherel measure is $d\sigma_{\cP}(\lambda) = \chi_{(0, +\infty)}(\lambda) |c_d(\lambda)|^{-2} d\lambda$ on $\Omega_d$, where $c_d$ denotes the Harish-Chandra $c$-function for $G_d$, and $\sigma_{\cP}((0, \varepsilon]) \asymp \varepsilon^3$ as $\varepsilon \rightarrow 0^+$. We note that a necessary condition for spectral hyperuniformity is $\supp(\sigma_{\mu}) \subset (0, +\infty)$, which by Theorem \ref{TheoremTrombiVaradarajanAnkerPaleyWiener} implies that $\sigma_{\mu}$ is a tempered positive measure on $(0, +\infty)$. A class of random distributions $\mu$ that satisfy this are random $\sC^2$-distributions, in other words almost surely tempered in the sense of Harish-Chandra.

\subsubsection{A general rank one criteria for spectral hyperuniformity}

To summarize the previous Subsubsection, spectral hyperuniformity of an invariant locally square-integrable random distribution $\mu$ on $X = \R^d, \H^d$ requires that the Bartlett spectral measure $\sigma_{\mu}$ is supported on $(0, +\infty)$ and that 
\begin{align*}
\limsup_{\varepsilon \rightarrow 0^+} \frac{\sigma_{\mu}((0, \varepsilon])}{\varepsilon^{\alpha}} = 0
\end{align*}
where $\alpha = d$ for $X = \R^d$ and $\alpha = 3$ for $X = \H^d$. We prove the following reformulation of this property.

\begin{proposition}
\label{PropTemperedHyperuniformity}
Let $F : \R \rightarrow \R$ is an even Schwartz function which does not vanish on $[-1, 1]$ and let $\sigma \in \Radonplus((0, +\infty))$ be a tempered measure. Then 
\begin{align*}
\lim_{\varepsilon \rightarrow 0^+} \frac{\sigma((0, \varepsilon])}{\varepsilon^{\alpha}} = 0
\end{align*}
if and only if
\begin{align*}
\lim_{s \rightarrow +\infty} s^{\alpha} \int_0^{\infty} F(s\lambda)^2 d\sigma(\lambda) = 0 \, . 
\end{align*}
\end{proposition}

The heart of the proof of this Proposition lies in the following technical Lemma.

\begin{lemma}
\label{LemmaHeatKernelHyperuniformityConditions}
Suppose that there are constants $\alpha, \beta > 0$, $\gamma \in (0, 1)$, $A, B, C \geq 0$ and Borel functions $ \theta, \psi : [0, +\infty) \rightarrow [0, +\infty)$ satisfying

(i) $\theta$ is increasing and such that 
\begin{itemize}
    \item for every $\varepsilon > 0$ there is an $r_{\varepsilon} \in (0, 1)$ with $r_{\varepsilon} \rightarrow 0$ as $\varepsilon \rightarrow 0^+$, and such that $\theta(r) \leq A \varepsilon r^{\alpha}$ for all $r \in [0, r_{\varepsilon}^{1 - \gamma}]$,
    \item there is a constant $M > 0$ such that $\theta(r) \leq B r^{\beta}$ for all $r \geq M$.
\end{itemize}

(ii) $\psi$ is such that
\begin{enumerate}
    \item $\int_0^{\infty} r^{\alpha} \psi(r) dr < +\infty$,
    \item $R^{\alpha/\gamma} \int_R^{\infty} \psi(r) dr \rightarrow 0$ as $R \rightarrow +\infty$,
    \item $R^{\alpha - \beta} \int_R^{\infty} r^{\beta} \psi(r) dr \rightarrow 0$ as $R \rightarrow +\infty$.
\end{enumerate}
Then
\begin{align*} 
r_{\varepsilon}^{-(\alpha + 1)} \int_0^{\infty} \psi(r/r_{\varepsilon}) \theta(r) dr \longrightarrow 0 \, , \quad \varepsilon \rightarrow 0^+  \, . 
\end{align*}
\end{lemma}

\begin{proof}
Let $I(\varepsilon) = \int_0^{\infty} \psi(r/r_{\varepsilon}) \theta(r) dr$ and split the integral as $I(\varepsilon) = I_1(\varepsilon) + I_2(\varepsilon) + I_3(\varepsilon)$, where
\begin{align*}
I_1(\varepsilon) &= \int_0^{r_{\varepsilon}^{1 - \gamma}} \psi(r/r_{\varepsilon}) \theta(r) dr \\
I_2(\varepsilon) &= \int_{r_{\varepsilon}^{1 - \gamma}}^{M} \psi(r/r_{\varepsilon}) \theta(r) dr \\
I_3(\varepsilon) &= \int_M^{\infty} \psi(r/r_{\varepsilon}) \theta(r) dr \, . 
\end{align*}
By the assumptions on $\theta$, we find that 
\begin{align*}
I_1(\varepsilon) &\leq A \varepsilon  \int_0^{r_{\varepsilon}^{1 - \gamma}} r^{\alpha} \psi(r/r_{\varepsilon}) dr \leq A \varepsilon r_{\varepsilon}^{\alpha + 1} \int_0^{\infty} r^{\alpha} \psi(r) dr \, , \\
I_2(\varepsilon) &\leq \theta(M) \int_{r_{\varepsilon}^{1 - \gamma}}^{M} \psi(r/r_{\varepsilon}) dr \leq \theta(M) r_{\varepsilon} \int_{r_{\varepsilon}^{-\gamma}}^{\infty} \psi(r) dr \\
I_3(\varepsilon) &\leq B \int_{M}^{\infty} r^{\beta} \psi(r/r_{\varepsilon}) dr = B r_{\varepsilon}^{\beta + 1} \int_{M r_{\varepsilon}^{-1}}^{\infty} r^{\beta} \psi(r) dr \, . 
\end{align*}
Dividing by $r_{\varepsilon}^{\alpha + 1}$ and using the assumptions made on $\psi$, we get that
\begin{align*}
r_{\varepsilon}^{-(\alpha + 1)}I(\varepsilon) \leq A \varepsilon \int_0^{\infty} r^{\alpha} \psi(r) dr &+ \theta(M) (r_{\varepsilon}^{-\gamma})^{\alpha/\gamma} \int_{r_{\varepsilon}^{-\gamma}}^{\infty} \psi(r) dr \\
&+  B M^{\beta - \alpha} (M r_{\varepsilon}^{-1})^{\alpha - \beta} \int_{M r_{\varepsilon}^{-1}}^{\infty} r^{\beta} \psi(r) dr \longrightarrow 0 
\end{align*}
as $\varepsilon \rightarrow 0^+$. 
\end{proof}
\begin{proof}[Proof of Proposition \ref{PropTemperedHyperuniformity}]
For the right implication, set $\theta(r) = \sigma((0, r])$ and \\ $\psi(r) = -2F'(r)F(r)$, so that 
\begin{align*}
\int_0^{\infty} |F(s\lambda)|^2 d\sigma(\lambda) &= \int_0^{\infty} \sigma\big(\big\{ \lambda \geq 0 : F(s\lambda)^2 \geq r \big\}\big) dr \\
&= -2 s \int_0^{\infty} F'(sr)F(sr) \sigma((0, r]) dr = s \int_0^{\infty} \psi(sr) \theta(r) dr \, . 
\end{align*}
Now, since $\sigma$ is tempered and satisfies $\sigma((0, \varepsilon]) = o(\varepsilon^{\alpha})$, then $\theta$ satisfies the two first conditions in Lemma \ref{LemmaHeatKernelHyperuniformityConditions} with $r_{\varepsilon} = s^{-1}$. Moreover, since $F$ is Schwartz, then $\psi = -2 F' F$ is Schwartz and thus $|\psi|$ satisfies the three latter conditions in Lemma \ref{LemmaHeatKernelHyperuniformityConditions}, so we get that
\begin{align*}
s^{\alpha} \int_0^{\infty} F(s\lambda)^2 d\sigma(\lambda) \leq (s^{-1})^{-(\alpha + 1)} \int_0^{\infty} |\psi(r/s^{-1})| \theta(r) dr \rightarrow 0 \, , \quad s \rightarrow +\infty
\end{align*}
as desired. 

Conversely, consider the lower bound
\begin{align*}
s^{\alpha} \int_0^{\infty} F(s\lambda)^2 d\sigma(\lambda) \geq s^{\alpha} \int_0^{1/s} F(s\lambda)^2 d\sigma(\lambda) \geq \sup_{t \in [0, 1]} F(t)^2 s^{\alpha} \sigma([0, s^{-1}]) \, . 
\end{align*}
If 
\begin{align*}
\lim_{s \rightarrow +\infty} s^{\alpha} \int_0^{\infty} F(s\lambda)^2 d\sigma(\lambda) = 0
\end{align*}
then setting $\varepsilon = s^{-1}$ yields
\begin{align*}
\limsup_{\varepsilon \rightarrow 0^+} \frac{\sigma((0, \varepsilon])}{\varepsilon^{\alpha}} \leq \frac{1}{\sup_{t \in [0, 1]} F(t)^2} \lim_{\varepsilon \rightarrow 0} \varepsilon^{-\alpha} \int_0^{\infty} F(\lambda/\varepsilon)^2 d\sigma(\lambda)  = 0 \, . 
\end{align*}
\end{proof}

We will next apply Proposition \ref{PropTemperedHyperuniformity} to the function $F(t) = \e^{-t^2}$ when proving the equivalence between spectral- and heat kernel hyperuniformity.

\subsubsection{Euclidean heat kernel hyperuniformity}

The Euclidean heat kernel is the family of functions $h_{\tau} : \R^d \rightarrow \R_{\geq 0}$ with $\tau > 0$ given by
\begin{align*}
h_{\tau}(x) = \frac{1}{(4\pi \tau)^{d/2}} \e^{-\frac{\norm{x}^2}{4\tau}} \, ,
\end{align*}
or equivalently defined by its Fourier transform as 
\begin{align*}
\hat{h}_{\tau}(y) = \frac{1}{(4\pi \tau)^{d/2}} \int_{\R^d} \e^{-\frac{\norm{x}^2}{4\tau}} \e^{-i \langle x, y \rangle} dx = \e^{-\tau\norm{y}^2} \, , \quad  y \in \R^d \, . 
\end{align*}

\begin{theorem}
\label{ThmEuclideanHeatKernelHyperuniformity}
An invariant locally square-integrable random tempered distribution $\mu$ on $\R^d$ is spectrally hyperuniform if and only if 
\begin{align*} 
%
\lim_{\tau \rightarrow +\infty} \, \tau^{d/2} \Varmu(\bS h_{\tau}) = 0 \, .   
\end{align*}
\end{theorem}

\begin{proof}
Since the heat kernel $h_{\tau}$ is radial, the variance of its linear statistic with respect to $\mu$ can be written in terms of the radialized Bartlett spectral measure $\sigma_{\mu}^{\rad}$ as 
\begin{align*}
\tau^{d/2} \Varmu(\bS h_{\tau}) =  \tau^{d/2}  \int_0^{\infty} \hat{h}_{\tau}(\lambda)^2 d\sigma_{\mu}^{\rad}(\lambda) = \tau^{d/2}  \int_0^{\infty} \e^{-2\tau \lambda^2} d\sigma_{\mu}^{\rad}(\lambda) \, . 
\end{align*}

As mentioned, the measure $\sigma_{\mu}$ is tempered as a consequence of the Paley-Wiener Theorem, so letting $s = \sqrt{\tau}$, $\alpha = d$, $F(t) = \e^{-t^2}$, then by Proposition \ref{PropTemperedHyperuniformity} we have that
\begin{align*}
\lim_{\tau \rightarrow +\infty} \tau^{d/2} \Varmu(\bS h_{\tau}) = \lim_{s \rightarrow +\infty} s^{\alpha} \int_0^{\infty} F(s\lambda)^2 d\sigma_{\mu}^{\rad}(\lambda) = 0
\end{align*}
if and only if $\sigma_{\mu}^{\rad}((0, \varepsilon]) = o(\varepsilon^{d})$ as $\varepsilon \rightarrow 0^+$.
\end{proof}

\subsubsection{Hyperbolic heat kernel hyperuniformity}

The heat kernel $h_{\tau} : G_d \rightarrow \C$ for $\bH^d$ is, similarly to the Euclidean case, defined as the inverse spherical transform of a normalized Gaussian on the space $\Omega_d \subset \C$ of parameters for the positive-definite spherical functions. More specifically, it is the smooth rapidly decaying bi-$K_d$-invariant function
\begin{align*}
h_{\tau}(g) = \int_0^{\infty} \e^{-t(\frac{(d-1)^2}{4} + \lambda^2)} \omega_{\lambda}(g) \frac{d\lambda}{|c_d(\lambda)|^2} \, . 
\end{align*}
To get a sense of the asymptotic behavior of this function, we mention the following uniform bounds due to Davies-Mandouvalos in \cite[Theorem 3.1, p. 186]{Davies1988HeatKB},
\begin{align*}
h_{\tau}(g) \asymp \tau^{-\frac{d}{2}}(1 + \tau + \ell(g))^{\frac{d - 3}{2}} (1 + \ell(g)) \e^{-\frac{(d-1)^2}{4}\tau - \frac{d-1}{2}\ell(g) - \frac{1}{4\tau} \ell(g)} 
\end{align*}
for all $g \in G_d$ and all $\tau \in (0, +\infty)$.

\begin{theorem}
\label{ThmHyperbolicHeatKernelHyperuniformity}
An invariant locally square-integrable random Harish-Chandra distribution $\mu$ on $\H^d$ is spectrally hyperuniform if and only if 
\begin{align*} 
\lim_{\tau \rightarrow +\infty} \, \tau^{3/2} \e^{\frac{(d-1)^2}{2} \tau} \,  \Varmu(\bS h_{\tau}) = 0 \, .   
\end{align*}
\end{theorem}

\begin{proof}
The proof is similar to the Euclidean case in Theorem \ref{ThmEuclideanHeatKernelHyperuniformity}. In terms of the Bartlett spectral measure, the heat kernel variance with respect to $\mu$ can be written as 
\begin{equation}
\begin{split}
\label{EqHyperbolicHeatKernelHyperuniformityIdenitity}
 \tau^{3/2} \e^{\frac{(d-1)^2}{2} \tau} \,  \Varmu(\bS h_{\tau}) &= \tau^{3/2}  \e^{\frac{(d-1)^2}{2} \tau} \Big( \int_0^{\infty} \hat{h}_{\tau}(\lambda)^2 d\sigma_{\mu}^{(p)}(\lambda) +  \int_0^{\frac{d-1}{2}} \hat{h}_{\tau}(i\lambda)^2 d\sigma_{\mu}^{(c)}(\lambda) \Big)\\
&=   \tau^{3/2}  \int_0^{\infty} \e^{-2\tau\lambda^2} d\sigma_{\mu}^{(p)}(\lambda) +  \tau^{3/2}  \int_0^{\frac{d-1}{2}} \e^{2\tau\lambda^2} d\sigma_{\mu}^{(c)}(\lambda) \, , 
\end{split}
\end{equation}
where $\sigma_{\mu}^{(p)}$ and $\sigma_{\mu}^{(c)}$ are the restrictions of $\sigma_{\mu}$ to $(0, +\infty)$ and $i[0, \tfrac{d-1}{2}]$ respectively.

First, assume that $\mu$ is spectrally hyperuniform, in other words that $\sigma_{\mu}^{(c)} = 0$ and $\sigma_{\mu}^{(p)}((0, \varepsilon]) = o(\varepsilon^3)$ as $\varepsilon \rightarrow 0^+$. Since $\sigma_{\mu}$ is positive and tempered and 
$$\hat{h}_{\tau}(\lambda) = \e^{-\tau(\frac{(d-1)^2}{4} + \lambda^2)}$$
is a positive decaying Schwartz function, Proposition \ref{PropTemperedHyperuniformity} yields
\begin{align*}
\lim_{\tau \rightarrow +\infty} \, \tau^{3/2} \e^{\frac{(d-1)^2}{2} \tau} \,  \Varmu(\bS h_{\tau}) = \lim_{s \rightarrow +\infty} s^3  \int_0^{\infty} \e^{-2 (s\lambda)^2} d\sigma_{\mu}^{(p)}(\lambda) = 0 \, . 
\end{align*}
Conversely, suppose that 
$$\tau^{3/2} \e^{\frac{(n-1)^2}{2} \tau} \Varmu(\bS h_{\tau}) \rightarrow 0 $$
as $\tau \rightarrow +\infty$. Then it is clear from Equation \eqref{EqHyperbolicHeatKernelHyperuniformityIdenitity} that we must have $\sigma_{\mu}^{(c)} = 0$, and so 
\begin{align*}
\lim_{s \rightarrow +\infty} s^3  \int_0^{\infty} \e^{-2 (s\lambda)^2} d\sigma_{\mu}^{(p)}(\lambda) = 0 \, . 
\end{align*}
Since $\sigma_{\mu}^{(p)}$ is a tempered measure on $(0, +\infty)$, Proposition \ref{PropTemperedHyperuniformity} implies that
\begin{align*}
\frac{\sigma_{\mu}^{(p)}((0, \varepsilon])}{\varepsilon^3} \longrightarrow 0 \, , \quad \varepsilon \rightarrow 0^+ \, ,
\end{align*}
so $\mu$ is spectrally hyperuniform.
\end{proof}


\printbibliography

@article{Bjorklund2022HyperuniformityAN,
  title={Hyperuniformity and non-hyperuniformity of quasicrystals},
  author={Björklund, M. and Hartnick, T.},
  journal={Math. Ann.},
  year={2023},
  %url={https://arxiv.org/abs/2210.02151}
}

@article{StillingerTorquato,
  title = {Local density fluctuations, hyperuniformity, and order metrics},
  author = {Torquato, S. and Stillinger, F. H.},
  journal = {Phys. Rev. E},
  volume = {68},
  numpages = {25},
  year = {2003},
  publisher = {American Physical Society},
}

@misc{Björklund2024HyperuniformityOfRandomMeasuresOnEuclideanAndHyperbolicSpaces,
      title={Hyperuniformity of random measures on Euclidean and hyperbolic spaces}, 
      author={M. Björklund and M. Byléhn},
      year={2024},
      eprint={2405.12737},
      archivePrefix={arXiv},
      primaryClass={math.PR},
     % url={https://arxiv.org/abs/2405.12737}, 
}

@article {StepanyukHyperuniformityOnFlatTori,
    AUTHOR = {Stepanyuk, T. A.},
     TITLE = {Hyperuniform point sets on flat tori: deterministic and
              probabilistic aspects},
   JOURNAL = {Constr. Approx.},
  FJOURNAL = {Constructive Approximation. An International Journal for
              Approximations and Expansions},
    VOLUME = {52},
      YEAR = {2020},
    NUMBER = {2},
     PAGES = {313--339},
}

@article {GrabnerHyperuniformityOnSpheresProbabilisticAspects,
    AUTHOR = {Brauchart, J. S. and Grabner, P. J. and Kusner, W.
              and Ziefle, J.},
     TITLE = {Hyperuniform point sets on the sphere: probabilistic aspects},
   JOURNAL = {Monatsh. Math.},
  FJOURNAL = {Monatshefte f\"ur Mathematik},
    VOLUME = {192},
      YEAR = {2020},
    NUMBER = {4},
     PAGES = {763--781},
}

@article {GrabnerHyperuniformityOnSpheresDeterministicAspects,
    AUTHOR = {Brauchart, J. S. and Grabner, P. J. and Kusner,
              W.},
     TITLE = {Hyperuniform point sets on the sphere: deterministic aspects},
   JOURNAL = {Constr. Approx.},
  FJOURNAL = {Constructive Approximation. An International Journal for
              Approximations and Expansions},
    VOLUME = {50},
      YEAR = {2019},
    NUMBER = {1},
     PAGES = {45--61},
}

@misc{Grabner2024HyperuniformPointsetsOnProjectiveSpaces,
      title={Hyperuniform point sets on projective spaces}, 
      author={J. S. Brauchart and P. J. Grabner},
      year={2024},
      eprint={2403.03572},
      archivePrefix={arXiv},
      primaryClass={math.CA},
     % url={https://arxiv.org/abs/2403.03572}, 
}

@misc{byléhn2024hyperuniformityregulartrees,
      title={Hyperuniformity in regular trees}, 
      author={Byléhn, M.},
      year={2024},
      eprint={2409.10998},
      archivePrefix={arXiv},
      primaryClass={math.PR},
      %url={https://arxiv.org/abs/2409.10998}, 
}

@article{deConickDunlopHuillet,
author = {De Coninck, J. and Dunlop, F. and Huillet, T.},
year = {2008},
month = {02},
pages = {725-744},
title = {On the correlation structure of some random point processes on the line},
volume = {387},
journal = {Physica A: Statistical Mechanics and its Applications},
%doi = {10.1016/j.physa.2007.10.018}
}

@article {Buckley,
    AUTHOR = {Buckley, J.},
     TITLE = {Fluctuations in the zero set of the hyperbolic {G}aussian
              analytic function},
   JOURNAL = {Int. Math. Res. Not. IMRN},
  FJOURNAL = {International Mathematics Research Notices. IMRN},
      YEAR = {2015},
    NUMBER = {6},
     PAGES = {1666--1687},
   %   ISSN = {1073-7928,1687-0247},
   MRCLASS = {60G99 (60G55)},
  MRNUMBER = {3340370},
      % DOI = {10.1093/imrn/rnt269},
      % URL = {https://doi.org/10.1093/imrn/rnt269},
}

@article {MassanedaPridhnani,
    AUTHOR = {Massaneda, X. and Pridhnani, B.},
     TITLE = {Volume fluctuations of random analytic varieties in the unit
              ball},
   JOURNAL = {Indiana Univ. Math. J.},
  FJOURNAL = {Indiana University Mathematics Journal},
    VOLUME = {64},
      YEAR = {2015},
    NUMBER = {6},
     PAGES = {1667--1695},
  %    ISSN = {0022-2518,1943-5258},
   MRCLASS = {32C30 (60G15)},
  MRNUMBER = {3436231},
MRREVIEWER = {M.\ Klimek},
      % DOI = {10.1512/iumj.2015.64.5693},
      % URL = {https://doi.org/10.1512/iumj.2015.64.5693},
}

@article{Davies1988HeatKB,
  title={Heat Kernel Bounds on Hyperbolic Space and Kleinian Groups},
  author={Davies, E. B. and Mandouvalos, N.},
  journal={Proc. London Math. Soc.},
  year={1988},
  pages={182-208},
  %url={https://api.semanticscholar.org/CorpusID:121088816}
}

@article{Jenni1984UeberDE,
  title={Ueber den ersten Eigenwert des Laplace-Operators auf ausgew{\"a}hlten Beispielen kompakter Riemannscher Fl{\"a}chen},
  author={Jenni, F.},
  journal={Comment. Math. Helv.},
  year={1984},
  volume={59},
  pages={193-203},
  %url={https://api.semanticscholar.org/CorpusID:123822980}
}

@article{AbreuPereiraRomeroTorquato,
    AUTHOR = {Abreu, L. D. and Pereira, J. M. and Romero, J. L. and
              Torquato, S.},
     TITLE = {The {W}eyl-{H}eisenberg ensemble: hyperuniformity and higher
              {L}andau levels},
   JOURNAL = {J. Stat. Mech. Theory Exp.},
  FJOURNAL = {Journal of Statistical Mechanics: Theory and Experiment},
      YEAR = {2017},
    NUMBER = {4},
    PAGES = {043103, 16},
  %    ISSN = {1742-5468},
   MRCLASS = {81R12},
  MRNUMBER = {3652713},
MRREVIEWER = {Dauren\ B.\ Bazarkhanov},
      % DOI = {10.1088/1742-5468/aa68a7},
      % URL = {https://doi.org/10.1088/1742-5468/aa68a7},
}

@article {PeresVirag,
    AUTHOR = {Peres, Y. and Vir\'ag, B.},
     TITLE = {Zeros of the i.i.d.\ {G}aussian power series: a conformally
              invariant determinantal process},
   JOURNAL = {Acta Math.},
  FJOURNAL = {Acta Mathematica},
    VOLUME = {194},
      YEAR = {2005},
    NUMBER = {1},
     PAGES = {1--35},
   %   ISSN = {0001-5962,1871-2509},
   MRCLASS = {60G99 (60G15)},
  MRNUMBER = {2231337},
MRREVIEWER = {Tomohiro\ Sasamoto},
      % DOI = {10.1007/BF02392515},
      % URL = {https://doi.org/10.1007/BF02392515},
}

@article {UnterbergerUpmeier,
    AUTHOR = {Unterberger, A. and Upmeier, H.},
     TITLE = {The {B}erezin transform and invariant differential operators},
   JOURNAL = {Comm. Math. Phys.},
  FJOURNAL = {Communications in Mathematical Physics},
    VOLUME = {164},
      YEAR = {1994},
    NUMBER = {3},
     PAGES = {563--597},
   %   ISSN = {0010-3616,1432-0916},
   MRCLASS = {58G15 (22E70 32M15 46N50 47G99 47N50 81R30)},
  MRNUMBER = {1291245},
       %URL = {http://projecteuclid.org/euclid.cmp/1104270949},
}

@phdthesis{BarkersThesis,
    author = {Barker, W.H.},
    title = {Positive definite distributions on semi-simple Lie groups},
    school = {MIT},
    year = {1973}
}

@book{EinsiedlerWardUnitaryBook,
    author = {Einsiedler, M. and Ward, T.},
    title = {Unitary Representations and Unitary Duals},
    publisher = {Springer},
    year = {2024}
}

@book{Loomis1953AnIT,
  author={Loomis, L. H.},
  title={An Introduction to Abstract Harmonic Analysis},
  publisher = {D. Van Nostrand company},
  year={1953}
}

@article{MackeyInducedRepsI,
 author = {Mackey, G.W.},
 journal = {Ann. Math.},
 number = {1},
 pages = {101--139},
 publisher = {Annals of Mathematics},
 title = {Induced Representations of Locally Compact Groups I},
 volume = {55},
 year = {1952}
}

@book{ReedSimonScatteringTheoryIII,
    author = {Reed, M. and Simon, B.},
    title = {Methods of modern mathematical physics III: Scattering theory},
    publisher = {Academic press},
    year = {1979}
}

@book{baccelli:hal-02460214,
  TITLE = {{Random Measures, Point Processes, and Stochastic Geometry}},
  AUTHOR = {Baccelli, F. and Blaszczyszyn, B. and Karray, M.},
  PUBLISHER = {{Inria}},
  YEAR = {2024},
  MONTH = Jul,
  HAL_ID = {hal-02460214},
  HAL_VERSION = {v2},
  %URL = {https://inria.hal.science/hal-02460214},
  %PDF = {https://inria.hal.science/hal-02460214v2/file/PointProcesses51.pdf},
}

@inbook{KoranyiHeisenbergMotionGroupGelfandPair,
    author = {Koranyi, A.},
    title = {Harmonic Analysis and Group Representations},
    publisher = {Springer},
    year = {1980},
    chapter = {Some Applications of Gelfand Pairs in Classical Analysis}
}

@book{ThangaveluHeisenbergBook,
    author = {Thangavelu, S.} ,
    title = {Harmonic Analaysis on thee Heisenberg Group},
    publisher = {Springer},
    year = {1998}
}

@article{Anker1991TheSF,
  title={The spherical Fourier transform of rapidly decreasing functions. A simple proof of a characterization due to Harish-Chandra, Helgason, Trombi, and Varadarajan},
  author={Anker, J-P.},
  journal={J. Funct. Anal.},
  year={1991},
  volume={96},
  pages={331-349},
  %url={https://api.semanticscholar.org/CorpusID:119788481}
}

@article {Bartlett,
    AUTHOR = {Bartlett, M. S.},
     TITLE = {The spectral analysis of point processes},
   JOURNAL = {J. Roy. Statist. Soc. Ser. B},
  FJOURNAL = {Journal of the Royal Statistical Society. Series B.
              Methodological},
    VOLUME = {25},
      YEAR = {1963},
     PAGES = {264--296},
  %    ISSN = {0035-9246},
   MRCLASS = {62.85},
  MRNUMBER = {171334},
MRREVIEWER = {P.\ M.\ Lee},
 %URL ={http://links.jstor.org/sici?sici=0035-9246(1963)25:2<264:TSAOPP>2.0.CO;2-X&origin=MSN},
}

@book{GradshteynRyzhik,
    author = {Gradshteyn, I. S. and Ryzhik, I. M.},
    title = {Table of Integrals, Series and Products, seventh edition},
    publisher = {Academic Press},
    year = {2007}
}

@book{GelfandShilovGeneralizedFunctionsVolume2,
    author = {Gelfand, I. M. and Shilov, G.E.},
    title = {Generalized functions, Vol II},
    publisher = {Academic Press},
    year = {1968}
}

@book{GelfandShilovGeneralizedFunctionsVolume3,
    author = {Gelfand, I. M. and Shilov, G.E.},
    title = {Generalized functions, Vol III},
    publisher = {Academic Press},
    year = {1967}
}

@book{GorodnikNevoBook,
    author = {Gorodnik, A. and Nevo, A.},
    title = {The Ergodic Theory of Lattice Subgroups},
    publisher = {Princeton University Press},
    year = {2010}
}

@incollection {GorodnikHigherOrderCorrelations,
    AUTHOR = {Gorodnik, A.},
     TITLE = {Higher order correlations for group actions},
 BOOKTITLE = {Geometric and ergodic aspects of group actions},
    SERIES = {Infosys Sci. Found. Ser.},
     PAGES = {83--133},
 PUBLISHER = {Springer, Singapore},
      YEAR = {2019},
}

@book{BorelWallachBook,
    author = {Borel, A. and Wallach, N.},
    title = {Continuous Cohomology, Discrete Subgroups, and Representations of Reductive Groups},
    publisher = {American Mathematical Society},
    year = {1991}
}

@article {EguchiHCLpSpaces,
    AUTHOR = {Eguchi, M.},
     TITLE = {Asymptotic expansions of {E}isenstein integrals and {F}ourier
              transform on symmetric spaces},
   JOURNAL = {J. Functional Analysis},
  FJOURNAL = {Journal of Functional Analysis},
    VOLUME = {34},
      YEAR = {1979},
    NUMBER = {2},
     PAGES = {167--216},
}

@book {CornulierDeLaHarpeIsometryGroupsOfProperMetricSpaces,
    AUTHOR = {Cornulier, Y. and de la Harpe, P.},
     TITLE = {Metric geometry of locally compact groups},
    SERIES = {EMS Tracts in Mathematics},
    VOLUME = {25},
      NOTE = {Winner of the 2016 EMS Monograph Award},
 PUBLISHER = {European Mathematical Society (EMS), Z\"urich},
      YEAR = {2016},
     PAGES = {viii+235},
}

@article{LastHomogeneousSpaces,
    AUTHOR = {Last, G\"unter},
     TITLE = {Stationary random measures on homogeneous spaces},
   JOURNAL = {J. Theoret. Probab.},
  FJOURNAL = {Journal of Theoretical Probability},
    VOLUME = {23},
      YEAR = {2010},
    NUMBER = {2},
     PAGES = {478--497},
}

@article {BHPI,
    AUTHOR = {Bj\"orklund, Michael and Hartnick, Tobias and Pogorzelski,
              Felix},
     TITLE = {Aperiodic order and spherical diffraction, {I}:
              auto-correlation of regular model sets},
   JOURNAL = {Proc. Lond. Math. Soc. (3)},
  FJOURNAL = {Proceedings of the London Mathematical Society. Third Series},
    VOLUME = {116},
      YEAR = {2018},
    NUMBER = {4},
     PAGES = {957--996},
}

@article {BHPII,
    AUTHOR = {Bj\"orklund, Michael and Hartnick, Tobias and Pogorzelski,
              Felix},
     TITLE = {Aperiodic order and spherical diffraction, {II}: translation
              bounded measures on homogeneous spaces},
   JOURNAL = {Math. Z.},
  FJOURNAL = {Mathematische Zeitschrift},
    VOLUME = {300},
      YEAR = {2022},
    NUMBER = {2},
     PAGES = {1157--1201},
}

@article {BHPIII,
    AUTHOR = {Bj\"orklund, Michael and Hartnick, Tobias and Pogorzelski,
              Felix},
     TITLE = {Aperiodic order and spherical diffraction, {III}: the shadow
              transform and the diffraction formula},
   JOURNAL = {J. Funct. Anal.},
  FJOURNAL = {Journal of Functional Analysis},
    VOLUME = {281},
      YEAR = {2021},
    NUMBER = {12},
     PAGES = {Paper No. 109265, 59},
}

{\small\textsc{Department of Mathematics, Chalmers and University of Gothenburg, Gothenburg, Sweden \\
 \textit{Email address}: {\tt micbjo@chalmers.se}}}

{\small\textsc{Department of Mathematics, Chalmers and University of Gothenburg, Gothenburg, Sweden \\
 \textit{Email address}: {\tt bylehn@chalmers.se}}}

\end{document}